\documentclass[a4paper,11pt]{article}
\usepackage[utf8]{inputenc}
\usepackage{amsfonts}
\usepackage{amsmath}
\usepackage{amssymb}
\usepackage{amsthm}
\usepackage{color,a4wide}
\usepackage{algorithm,algpseudocode, algorithmicx}
\usepackage{graphicx}
\usepackage[caption=false]{subfig}
\usepackage{tikz}
\usetikzlibrary{arrows}
\tikzstyle{block}=[draw opacity=0.7,line width=1.4cm]
\usepackage{pgfplots}

\usepackage{array,multirow}
\newcolumntype{B}[1]{>{\centering\arraybackslash}m{#1}}
\usepackage{colortbl}

\theoremstyle{plain}
\newtheorem{corollary}{Corollary}[section]
\theoremstyle{plain}
\newtheorem{theorem}{Theorem}[section]
\theoremstyle{plain}
\newtheorem{assumption}{Assumption}[section]
\theoremstyle{definition}
\newtheorem{definition}{Definition}[section]
\theoremstyle{plain}
\newtheorem{lemma}{Lemma}[section]
\theoremstyle{remark}
\newtheorem{remark}{Remark}[section]
\theoremstyle{remark}
\newtheorem{example}{Example}[section]

\DeclareMathOperator*{\argmin}{\arg \min}
\DeclareMathOperator*{\arrow}{\longrightarrow}
\DeclareMathOperator{\bv}{\text{BV}}
\DeclareMathOperator{\hilbert}{\mathcal{H}} 
\DeclareMathOperator{\domain}{\mathcal{U}} 
\DeclareMathOperator{\range}{\mathcal{V}}
\DeclareMathOperator{\tv}{\text{TV}}
\DeclareMathOperator{\tvast}{\tv_\ast}

\newcommand{\dom}{\text{dom}}
\newcommand{\energy}{E}
\newcommand{\clean}{f}
\newcommand{\fidelfct}{F} 
\newcommand{\fidelity}[2]{F(#1, #2)} 
\newcommand{\linbound}[2]{\mathcal{L}(#1, #2)} 
\newcommand{\minsol}{u^\dagger}
\newcommand{\noisy}{f^\delta}
\newcommand{\N}{\mathbb{N}} 
\newcommand{\R}{\mathbb{R}} 
\newcommand{\Z}{\mathbb{Z}} 
\newcommand{\rangedata}{f_\regparambold^\dagger}
\newcommand{\regdomain}{A} 
\newcommand{\regparam}{\alpha} 
\newcommand{\regparambold}{{\boldsymbol{\regparam}}} 
\newcommand{\betabold}{{\boldsymbol{\beta}}} 
\newcommand{\regfct}{J} 
\newcommand{\regfctarg}[2][\regparambold]{\regfct(#2, #1)} 
\newcommand{\regop}{R} 
\newcommand{\regopit}{R_I} 
\newcommand{\regoparg}[2][\regparambold]{\regop(#2, #1)} 
\newcommand{\regopitarg}[2][\regparambold]{\regopit(#2, #1)} 
\newcommand{\regsol}[1][\regparambold]{u^{#1}}
\newcommand{\Rinf}{\mathbb{R} \cup \{ \infty \}} 
\newcommand{\scdata}{g}
\newcommand{\ul}{u_\lambda} 
\newcommand{\select}{{\mathcal S}}
\newcommand{\surrogate}{\regfct_\tau}
\newcommand{\surrogatearg}[2][\regparam]{\surrogate(#2, #1)}
\newcommand{\oprange}{{\mathcal R}}
\newcommand{\nullspace}{{\mathcal N}}
\newcommand{\distance}{D}
\newcommand{\wlim}[1][\domain]{\arrow_{\tau_{#1}}}

\newcommand{\bregdis}[4][\regfctarg{\cdot}]{D_{#1}^{#2}(#3, #4)} 
\newcommand{\symbreg}[3][\regfctarg{\cdot}]{D_{#1}^{\text{symm}}(#2, #3)} 
\newcommand{\prox}[1][\regfctarg{\cdot}]{\left( I + \partial #1 \right)^{-1}}

\begin{document}
\title{Modern Regularization Methods for Inverse Problems}
\author{Martin Benning and Martin Burger}
\date{December 18, 2017}
\maketitle

\begin{abstract}
Regularization methods are a key tool in the solution of inverse problems. They are used to introduce prior knowledge and make the approximation of ill-posed (pseudo-)inverses feasible. In the last two decades interest has shifted from linear towards nonlinear regularization methods even for linear inverse problems. The aim of this paper is to provide a reasonably comprehensive overview of this development towards modern nonlinear regularization methods, including their analysis, applications, and issues for future research. 

In particular we will discuss variational methods and techniques derived from those, since they have attracted particular interest in the last years and link to other fields like image processing and compressed sensing. We further point to developments related to statistical inverse problems, multiscale decompositions, and learning theory.

\end{abstract}

\section{Introduction}

Starting from the development of tomography and related techniques, the last fifty years have seen a constant rise of interest in the development of  {\em inverse problems} as a research field, in mathematics as well as applied fields such as medical imaging, geophysics and oil industry, or steel industry to mention only a few (cf. e.g. \cite{bertero1998introduction,cakoni2005survey,chadan1997introduction,colton2012inverse,colton2012surveys,engl2012inverse,Groetsch,isakov2006inverse,isakov2008inverse,Natterer02,Natterer:MathematicalMethods,tarantola1982inverse,tarantola2005inverse}). Connected with the rise of interest in inverse problems is the development and analysis of {\em regularization methods}, which are a necessity in most inverse problems due to their ill-posedness (cf. e.g. \cite{tikhonov1987ill,Engl:Regularization}). In particular there is usually no continuous dependence between the data and the solution of the inverse problem, hence in the presence of measurement errors one rather solves approximate problems with a stable dependence instead. The controlled construction and analysis of such modified problems is called regularization, usually with a regularization parameter encoding the level of the approximation. 

The canonical example of an ill-posed inverse problem at the abstract level is the linear operator equation
\begin{equation} \label{eq:basicequation}
K u = \clean,
\end{equation} 
with a linear operator $K$ between Banach spaces, whose generalized inverse $K^\dagger$ is unbounded. A regularization method is then some parametric approximation $R_\alpha$ of $K^\dagger$, which has better stability properties. In the case of linear regularization methods, $R_\alpha$ is a family of bounded linear operators converging pointwise to $K^\dagger$ on the domain of the latter as $\alpha \rightarrow 0$. A key question in this respect is the convergence in the case of noisy data, related to the choice of the regularization parameter $\alpha$ in dependence on the noise level $\delta$, the latter being a bound for the noise in the deterministic setting or some kind of variance in a stochastic setting. 

While at the end of the 20th century, a rather complete understanding of such linear regularization methods was available based on spectral decompositions of the operators, the case of nonlinear regularization methods, i.e., nonlinear maps $R_\alpha$ (possibly even multi-valued) became a field of intensive study. This was driven in particular by developments related to variational methods such as total variation techniques (cf. \cite{rudin1992nonlinear,Acar:AnalysisOfBV,tvzoo}) or sparsity and compressed sensing (cf.\cite{Donoho:CompressedSensing,donoho2006stable,candes2002recovering}), but also by statistical approaches such as advanced Bayesian prior models (cf. \cite{LSS09,HL11,KLNS12}). Due to the rise of big data and learning techniques there is further interest in applying such paradigms to inverse problems in recent years. This is a somehow delicate task, since in most inverse problems there are no ground truth data, but only results that have been reconstructed with a certain regularization method and specific noise. Hence, it poses a lot of particular challenges for future research.

In this paper we will provide a survey of developments on modern (nonlinear) regularization methods in the last decades, their analysis and applications. Moreover, we will try to provide a quite structured overview of this field, including some fundamentals of nonlinear regularization methods. In particular we will give clear definitions of what to expect from a regularization method and its convergence reminiscent of the rather complete treatment of linear regularization methods in the seminal book by Engl, Hanke, and Neubauer \cite{Engl:Regularization} now dating back more than twenty years.

Throughout the paper we assume that $K:\domain \rightarrow \range$ is a bounded linear operator on Banach spaces $\domain$ and $\range$. In many parts there are obvious extensions to nonlinear operators and even metric spaces, but we mainly leave them out in order to increase readability, some links to such extensions are given at the end of the paper. 

We will start with a rather historical exposition on regularization methods in the next section and then proceed to nonlinear variational models, which are the class of methods driving most development on nonlinear regularizations. Section 4 will discuss some basic properties of and requirements on regularization methods, which are then discussed in detail for variational regularizations in Section 5. Subsequently, we turn to iterative regularization methods in Section 6. As a result of some insights in these sections we are led to a discussion on bias and scales in regularization methods in Section 7 and Section 8 will provide some examples of applications. Section 9 will discuss advanced aspects such as nonlinear regularization methods for nonlinear inverse problems and links to machine learning. Finally we conclude and provide an outlook to relevant future topics in Section 10.

\section{A Little History of Regularization Methods}

It seems rather difficult to date back the origin of regularization methods, but it is common now to identify it with the pioneering work of Tikhonov (cf. \cite{tikhonov1943stability,tikhonov1963solution,tikhonov1966stability}) and the subsequent strong developments in the Russian community in the 1960s (cf. e.g. \cite{ivanov1962linear,bakushinskii1967general}). The starting motivation obviously comes from the concept of {\em ill-posedness}, negating the definition of a well-posed problem. The latter, consisting of existence, uniqueness, and stable dependence upon the input data is usually attributed to the work of Hadamard in the context of partial differential equations (cf. \cite{hadamard1902problemes,hadamard1923lectures}), however the third condition was not clearly formulated in those problems and seems to have found its way as an equally important one later, e.g. in the work of John \cite{john1960continuous}. As a motivation for regularization theory and in particular for their convergence, the lack of stability seems to be the most crucial issue however. 

Already in the early works it was understood that in order to have any chance to compute meaningful solutions, the problem needs to be approximated by well-posed ones, usually a family parametrized by the regularization parameter. The obvious first answer of a topologist like Tikhonov was to restrict the domain to a compact set in some topology (or some kind of family thereof), leading to the concept of conditional well-posedness. A natural choice in a Hilbert space are norm balls around zero, which are compact in the weak topology. The radius of the balls (or its inverse) can naturally serve as a regularization parameter. This was also called {\em selection method} and the corresponding solutions were phrased quasi-solutions. Given a minimization problem, e.g. least-squares $\Vert Ku -f \Vert^2$ for \eqref{eq:basicequation} in Hilbert spaces, it is a short way to the variational formulation (see Section 3 for a detailed discussion of variational models) of what is now called {\em Tikhonov} or {\em Tikhonov-Phillips} regularization. Indeed, with an appropriate Lagrange parameter $\alpha$, this is equivalent to the variational problem
\begin{equation}
\hat u = \argmin_{u \in \domain} \frac{1}2 \Vert Ku - f \Vert^2 + \frac{\alpha}2 \Vert u\Vert^2. 
\end{equation}
Some of the early work in the Soviet community was already formulated in a much more general variational way, replacing the least-squares term by some discrepancy and the regularization by some appropriate functional, somehow a precursor of the modern theory. At this time the study was restricted to a rather abstract way focusing on convergence proofs, neither strong motivations for other functionals in inverse problems nor further methods for quantitative estimates were available. The concepts and methods were further developed in the Soviet literature, including the question of the regularization parameter choice in dependence on the noise level. Instead of giving a detailed overview we here refer to the influential book by Tikhonov and Arsenin \cite{tikhonovsolution}, which also made the results more broadly accessible.

As an alternative approach a lot of work also considered what Tikhonov called the {\em regularization method} (and what seems to be the first appearance of this term in literature), namely the approximation of $K$ by regular operators, respectively of its generalized inverse by bounded operators. In parallel there was similar development in the western community, a similar approach as the conditional well-posedness by Tikhonov was developed by Phillips for integral equations of the first kind in \cite{phillips1962technique}, consequently the term Tikhonov-Phillips regularization is also used in literature. In a discrete setting of statistical regression, a similar idea to deal with ill-conditioned problems was developed under the term {\em ridge regression} (cf. \cite{hoerl1959optimum,hoerl1970ridge}). A related approach to solve ill-posed problems for partial differential equations was the quasi-reversibility method (cf. \cite{lattes1967methode}), although hardly analyzed in the setting of a regularization method.
 
A different route to the construction of regularization methods was taken by Backus and Gilbert \cite{backus1968resolving} from a very applied perspective. Using linear filters the noisy data were smoothed to be in the range of the forward operator $K$, subsequently a direct inversion (or application of the generalized inverse) can be performed. It took quite a while until such methods were understood in a unified way with other regularizations such as Tikhonov regularization (cf. \cite{Engl:Regularization}), the key step was to relate the smoothing action of the filters to the operator $K$ respectively its adjoint. This was made clear later in the linear functional strategy by Anderssen \cite{anderssen1986linear} and also in the development of the approximate inverse method by Louis \cite{louis1996approximate}, which turned out to be highly useful in tomography problems, where explicit reconstruction formulas and fast methods for the computation of the inverse are available. 

In the seventies and eighties of the last century the study of linear regularization methods was progressing further, with a study of many different regularization techniques such as iterative regularization by early stopping of stable iteration methods, truncated singular value decompositions, regularization by discretization and projection
(cf. e.g. \cite{nashed1974regularization,nashed1974generalized,wahba1977practical,elden1977algorithms,bakushinskii1977methods,bakushinskii1979principle,bates1983truncated,engl1987choice,hansen1987truncatedsvd}).
Most work was based on using spectral methods for the construction and detailed analysis of regularization methods. This includes the basic analysis of linear regularization methods in Hilbert spaces, the convergence as noise level and regularization parameter tend to zero, as well as first error estimates in dependence of the noise level (cf. e.g. \cite{nashed1974convergence,groetsch1979extrapolation,natterer1984error,neubauer1988posteriori}). Moreover, various asymptotic parameter choice rules were suggested and investigated, either founded by theory such as the discrepancy principle or other a-posteriori rules using the noise level (cf. \cite{morozov1966regularization,bakushinskii1973proof,raus1984residue,engl1985optimal,engl1987discrepancy,engl1987optimal,gfrerer1987posteriori,engl1988posteriori,raus1992regularization}) or heuristic ones such as quasi-optimality or the L-curve method (cf. \cite{tikhonovsolution,bakushinskii1984remarks,thompson1991study,hansen1992analysis}). The development of linear regularization methods in the early nineties was rather complete, culminating in the seminal book by Engl, Hanke, and Neubauer \cite{Engl:Regularization} that provides a unified overview.

From the application point of view strong focus was put on models with integral equations of the first kind and image reconstruction in tomography became a driving field of application (cf. \cite{Natterer02,Natterer:MathematicalMethods} and references therein). In parallel various applications of inverse problems in partial differential equations such as inverse scattering or parameter identifications became relevant and were tackled by regularization methods (cf. e.g. \cite{payne1975improperly,kravaris1985identification,colton1988inverse,banks1989estimation,colton1990inverse}). This drove the interest in regularization theory from linear towards nonlinear problems.

The end of the eighties marks the beginning of the systematic analysis of regularization methods for nonlinear inverse problems (replacing $K$ by a nonlinear operator), in particular with the papers by Seidman and Vogel \cite{seidman1989well} giving a well-posedness and convergence analysis of Tikhonov regularization for such problems and by Engl, Kunisch and Neubauer \cite{engl1989convergence} providing first error estimates respectively convergence rates. Many techniques had to be developed to avoid spectral theory arguments that are not available for nonlinear operators, it is not surprising that many of those ideas were also influential for nonlinear regularization (of linear inverse problems). In the nineties there was a boost of studies for nonlinear inverse problems, in particular a theory of iterative regularization methods was worked out, which is particularly attractive since the nonlinear problems had to be solved anyway with iterative methods. Prominent examples are Landweber and steepest-descent methods (cf. e.g. \cite{hanke1995convergence}), regularized Newton methods (cf. e.g. \cite{kaltenbacher1997some}), and iterated Tikhonov methods (cf. e.g.\cite{scherzer1993convergence}). We refer to \cite{kaltenbacher2009iterative} for a comprehensive overview.

 In parallel another paradigm evolved in particular in the image processing community from the seminal papers of Rudin, Osher, Fatemi \cite{ROF} and Mumford-Shah \cite{mumford1989optimal}, who proposed nonlinear variational models to solve denoising (and in the second case also segmentation) problems. From a regularization point of view this means that a nonlinear regularization method is used to solve a linear inverse problem, a rather unusual idea at this time. From a technical point of view it poses additional challenges of analyzing schemes in anisotropic Banach spaces like the space of functions of bounded variation, while previous theory was formulated mainly in Hilbert spaces. In the case of variational regularization methods basic well-posedness and convergence analysis can be carried out using techniques from variational calculus \cite{Acar:AnalysisOfBV,Eggermont}, while quantitative estimates need completely novel approaches. Early progress in this direction was made for maximum entropy regularization (cf.  \cite{Eggermont,engl1993convergence}), in this case the regularization technique could be related directly to regularization of nonlinear inverse problems in Hilbert spaces by a change of variables (cf.  \cite{engl1993convergence}). However, in a more general setup the convergence rate theory remained quite open until the dawn of the 21st century, when strong progress was made by employing techniques from convex analysis to variational regularization methods. We mention at this point that some of these more geometric ideas were also hidden in earlier work on regularization in Hilbert spaces with convex constraints (cf. \cite{neubauer1988tikhonov,eicke1992iteration}). The improved understanding of variational regularization methods in Banach spaces subsequently led to a variety of other techniques and variants such as iterative regularization methods derived from those, which we will discuss in further detail in the course of this paper. 

Another driving force for investigating regularization methods in Banach spaces became ideas of sparsity including wavelet shrinkage and the variational counterpart of regularization $\ell^1$-type norms, e.g. in Besov spaces (cf. \cite{donoho1992superresolution,donoho1995adapting}). This led in parallel to the field of compressed sensing, where the focus was rather on designing the appropriate measurement setups for optimal compression than to improve reconstructions on a given inverse problem (cf. \cite{Donoho:CompressedSensing,Candes:Robust,candes2002recovering,candestao1,candestao2,candes2007sparsity,donoho2006stable}). Despite the fact that the usual setting in compressed sensing is rather a finite dimensional one, many arguments based on convex analysis are closely related. 

In recent years these techniques also evolved into many practical applications, in particular in the image reconstruction community. The whole list of applications where the methods made impact in different ways might deserve a survey paper for itself. In order to illustrate the change in the first decade of the twentieth century we just provide the following table showcasing the typical state of the art used for inverse problems in medical imaging before or around the year 2000 and the one typically used ten years later: 

\vspace*{6pt} 

\noindent
\begin{tabular}{|l|l|l|}
\hline Modality &        State of the art before 2000	     &       State of the art after 2010  \\ \hline
Full CT &		Filtered Backprojection	&	Filtered Backprojection  \\
Undersampled CT &		Filtered Backprojection	&	TV-type / Wavelet Sparsity  \\
PET / SPECT &	Filtered Backprojection / EM & 	EM-TV / Dynamic Sparsity \\
Photacoustics &	-				& TV-type / Wavelet Sparsity \\
EEG/MEG & LORETA 		& Spatial Sparsity / Bayesian \\
ECG-BSPM	& L2 Tikhonov		&	L1 of normal derivative\\
Microscopy	& None, linear Filter &		TV-type  / Shearlet Sparsity \\ 
PET-CT/MR	&	- &TV-type anatomical priors \\ \hline
\end{tabular}

\vspace*{6pt}

Note that (with the exception of the statistically motivated EM-algorithm) all state of the art methods before 2000 were linear regularization methods. This is completely changed with the exception of fully sampled CT, where there is neither a nullspace nor significant noise, hence the regularization plays a minor role. The details of most other methods, mainly based on variational models, will become clear in the next section.

\section{Variational Modeling}\label{sec:varmod}

The variational approach to regularization methods became very popular in the last decades, since it allows for an intuitive approach to modeling, a framework for its basic analysis, and also a variety of computational methods to be applied, in particular in the case of convex regularization functionals. The key idea to construct a variational regularization method for \eqref{eq:basicequation} consists of finding two functionals: a data fidelity term $\fidelfct$ measuring the distance between $Ku$ and $\clean$ (respectively its noisy version $\noisy$) and a regularization functional $\regfct$ favouring appropriate minimizers respectively penalizing potential solutions with undesired structures. Instead of simply fitting $u$ to data, i.e. minimizing $\fidelfct(Ku,f)$, a weighted version is minimized to obtain
\begin{equation} \label{eq:basicvariational}
	\hat u \in \text{arg}\min_u \left( \fidelfct(Ku,\noisy) + \alpha \regfct(u) \right),
\end{equation}
where $\alpha > 0$ is the regularization parameter controlling the influence of the two terms on the minimizer. Since the problem should approach the pure minimization of the data fidelity in the noise-less case it is natural to think about $\alpha$ as a small parameter. 

The choice of the data fidelity is often straightforward, e.g. as some kind of least squares term (squared norm distance in a Hilbert space), or motivated from statistical arguments by some likelihood functional for the noise. In the latter case the variational model can be interpreted as a regularized likelihood model, the data term usually corresponds to the negative log likelihood of the noise model. A prominent example is the case of additive Gaussian noise, which leads to a least-squares data term $\frac{1}2 \Vert K u - f \Vert^2$, where the specific Hilbert space norm to be used is determined by the covariance operator of the noise. Appropriate choices for the latter can have significant impact, e.g. choosing likelihoods for Poisson noise appearing in photon count data leads to strong improvement over least squares terms in particular in large noise regimes (cf. e.g. \cite{Brune:FBEMTV,Brune2009}).  Throughout this paper we will assume that $\fidelfct$ is Frechet-differentiable on $\range$ unless further noticed.

The choice of a regularization functional seems less natural at first glance. Based on the original ideas by Tikhonov the key ingredient for a successful regularization are its topological properties, thus frequently the regularization functional is chosen as some power of a norm (or seminorm) in a Banach space. Classical examples are Tikhonov-Phillips in Hilbert spaces like $L^2(\Omega)$, $H^1(\Omega)$, or in some sequence space $\ell^2(\mathbb{N})$. As a generalization in function spaces, regularization functionals depending on the gradient (or higher order derivatives) of $u$ became popular. Those correspond to a rather direct intuition when smooth solutions are preferable due to prior knowledge. Nonsmooth and oscillatory functions will lead to large or even infinite values of the derivatives and thus very high values of the regularization functionals. Hence, they are no suitable candidates as a minimizer of \eqref{eq:basicvariational}. 

In many cases in inverse problems such as image reconstruction one is rather interested in nonsmooth solutions and in particular their discontinuity sets. A simple class of such are piecewise constant functions with reasonable edge sets, which are not contained in any Sobolev space $W^{k,p}(\Omega)$ for $k,p \geq 1$, since their gradient is already a concentrated measure (cf. \cite{Ambrosio:BoundedVariation,evansgariepy}). This motivates to use the space of functions of bounded variations $BV(\Omega)$, which consists of all functions in $L^1(\Omega)$ whose distributional gradients are vectorial Radon measures. The regularization with the total variation, i.e., 
\begin{equation}
 TV(u) = |u|_{BV} = \int_\Omega d|Du|,
\end{equation}
where $Du$ is the gradient measure of $u$, proposed for denoising by Rudin-Osher-Fatemi \cite{rudin1992nonlinear} and the subsequent popularity of investigating such methods can be seen as the advent of modern regularization methods.

The details of reconstructions to be achieved strongly depend on the specific norm used however. 
It is common folklore that the regularization functional is chosen such that desired solutions matching prior knowledge have a small value of $\regfct$ and are thus preferred as the appropriate solutions. This is however true only to some extent, but the overall effect of a regularization functional is rather determined by the effect it has on possible minimizers than purely a comparison of functional values. Consider as a simple example one-dimensional total variation regularization. It will of course rather prefer solutions with small total variation over oscillatory functions with high variation. On the other hand, it still selects among functions with the same total variation. Structural results on the solution of total variation regularization problems show that canonical solutions for noisy data are piecewise constant, even if the exact solution is not (cf. \cite{ring2000structural,chambolle2010introduction,jalalzai2016some}). This means that total variation actively selects piecewise constant solutions over smooth solutions that have the same total variation, i.e. are a-priori indistinguishable by the regularization functional. The reason for this behaviour can be seen by inspecting the optimality condition, given by (assuming $\fidelfct$ to be Fr\'{e}chet-differentiable)
\begin{equation} \label{eq:basicoptimality}
	K^* \partial_x \fidelfct (Ku,f) + \alpha p = 0, \qquad p \in \partial \regfct(u).
\end{equation}
Here $\partial_x$ denotes the (partial) Fr\'{e}chet-derivative in the first argument, and $\partial \regfct(u)$ is the subdifferential of $\regfct$ at position $u$, see \cite[Section 23]{Rockafellar:ConvexAnalysis}, or \cite{ekelandtemam,bausch11}. Solving for the subgradient $p$ we always obtain a relation of the form $p=K^* \tilde w$ for some $\tilde w \in \range$, i.e. the variational method will select smooth subgradients due to the smoothing properties of the operator $K$ and its adjoint. We will detail the relation between the properties of the subgradients of the solution for total variation and other examples of regularization in the next sections.

In a stochastic setup, the variational approach is often formulated from Bayesian estimation (cf. \cite{kaipio2006statistical,Stuart}, in particular {\em maximum a-posteriori probability} (MAP) estimators. Assume for the sake of simpler presentation that we are in a finite-dimensional setting for the inverse problem $Ku=f$ and can write down probability densities for the prior $\pi_0(u)$ and the likelihood $\pi(f|u)$ of measuring the data $f$ given the true solution $u$. Then Bayes' theorem provides the posterior probability density via
\begin{equation}
	 \pi(u|f) = \frac{1}{\pi_*(f)}  \pi(f|u) \pi_0(u), 
\end{equation}
with 
\begin{equation}
	\pi_*(f) = \int \pi(f|u) \pi_0(u) ~du
\end{equation}
being the effective prior probability on the data. A MAP estimate $\hat u$ is defined as a maximizer of the posterior probability density, respectively a minimizer of its negative logarithm. Since the part $\pi_*(f)$ independent of $u$ is irrelevant for the minimizer, we thus have
\begin{equation}
	\hat u \in \text{arg}\min_u \left( - \log \pi(f|u) - \log \pi_0(u) \right).
\end{equation}
This formulation is closely related to the variational modelling point of view when interpreting $-\log \pi(f|u)$ as a data fidelity and $- \log \pi_0(u)$ as the regularization term. Indeed, for many standard stochastic (noise) models one obtains 
\begin{equation}
	\pi(f|u) \sim \exp(-\fidelfct(Ku,f) ) .
\end{equation}
Examples are additive Gaussian noise leading to a least-squares fidelity and Poisson noise leading to the Kullback-Leibler divergence. 
Assuming further that the prior is related to some regularization functional $\regfct$
\begin{equation}
	\pi_0(u) \sim \Phi(-\regfct(u))
\end{equation}
for some monotone function $\Phi$, we see that the MAP estimation problem becomes
\begin{equation}
	\hat u \in \text{arg}\min_u \fidelfct(Ku,f) - \log( \Phi(-\regfct(u))).
\end{equation}
This problem can be reformulated in a more conventional form, even if the prior $\Phi$ is not exactly specified. By a standard argument we see that there exists $\gamma > 0$ such that 
$$ 	\hat u \in \text{arg}\min_{u,  \regfct(u) \leq \gamma} \fidelfct(Ku,f) , $$
and with the existence of a Lagrange parameter $\alpha > 0$ for the constraint $J(u) \leq \gamma$ (which is easily verified for a scalar constraint) we obtain
\begin{equation}
	\hat u \in \text{arg}\min_u \fidelfct(Ku,f) + \alpha \regfct(u).
\end{equation}
We mention that similar reasoning in infinite dimensions is not as straightforward, even the definition of the MAP estimate is a non-obvious task (cf. \cite{Dashti13,HelinBurger15}). Recent results however provide a good characterization in many relevant cases (cf. \cite{HelinBurger15,lie2017equivalence,agapiou2017sparsity} ). A relation between Bayesian estimators and the variational approach also exists beyond the MAP estimate by the {\em Bayes cost} method. Given a cost $\psi$ measuring a distance on the input space, the Bayes cost approach looks for a minimizer of the posterior expecation of $\psi$, i.e.,
\begin{equation}
	\hat u \in \text{arg}\min_u \int \psi(u,v) ~\pi(v|f)~dv,
\end{equation}
i.e. a functional that depends in a more implicit way on the data and the forward model. 

\subsection{Total variation and related regularizations}

As mentioned above total variation regularization has been one of the driving examples in developing regularization methods in Banach spaces starting from \cite{rudin1992nonlinear,Acar:AnalysisOfBV}. Since then it has been a constant source of motivation for further developing mathematical analysis (cf. e.g. \cite{chambolle1997image,strong1996exact,scherzer1998denoising,chavent1997regularization,ring2000structural,strong2003edge,Burger:ConvergenceRates,caselles2007discontinuity,allard2007total}), computational optimization techniques for nonsmooth problems (cf. e.g. \cite{chan1999nonlinear,Vogel2002,chambolle2004algorithm,kunisch2004total,chambolle2011first}), and development of advanced models (cf. e.g. \cite{TV2,osher2005iterative,burger2005nonlinear,burger2006nonlinear,Bredies:TGV,hu2012higher,lenzen2013adaptive,benningbrunemueller}).

The key step for modern analysis and computational methods is the (pre-)dual formulation of total variation
\begin{equation} \label{rigorousBV}
TV(u) = |u|_{BV} := \sup_{g  \in C_0^\infty(\Omega)^d, g \in {\cal C}} \int_\Omega u \nabla \cdot g ~dx, 
\end{equation}
with the convex set 
$$ {\cal C} = \{ g \in L^\infty(\Omega)~|~ \vert g(x) \vert \leq 1 \text { a.e. in } \Omega \}.$$
This characterization allows to understand the structure of subgradients as elements of ${\cal C}$ absolutely continuous with respect to the gradient measure $D$ such that 
$$ \int_\Omega g \cdot dDu = |u|_{BV}. $$
The optimality condition \eqref{eq:basicoptimality}
\begin{equation} 
	K^* \partial_x \fidelfct (Ku,f) + \alpha \nabla \cdot g = 0,
\end{equation}
where $g$ is a vector field such that $g |Du|$ is a polar decomposition of the vector measure (cf. \cite{Ambrosio:BoundedVariation}). 

In spatial dimension one the structure of solutions can be understood directly from the optimality condition. If there is an open set where $u$ is not constant, either with positive or negative derivative, then $g$ equals $+1$ or $-1$, hence its derivative vanishes. Thus, in such regions the generalized residual $K^* \partial_x \fidelfct (Ku,f)$ vanishes. In the case of noisy data this is usually not happening for larger sets, thus $u$ is typically piecewise constant. In higher spatial dimension this is not completely true, but still the case $|g(x)| < 1 $ is the canonical one, so in many cases solutions are piecewise constant. On the other hand, piecewise constant structures are not optimal in all instances, in particular total variation methods are well-known to exhibit staircasing phenomena, i.e. smoothly varying parts in the solution are often approximated by piecewise constant structures with many jumps resembling a stair structure. For this sake many modifications and variants of total variation regularization have been investigated in the last decades. An immediate option are higher-order total variation approaches, that formally replace the one-norm of the gradient by the one-norm of a higher-order derivative like the Laplacian, the Hessian or the symmetric part of the Hessian (cf. e.g. \cite{TV2,CEP,hinterberger2006variational,papafitsoros2014combined}). The disadvantage of such an approach is that solutions of the regularization model will be too regular and discontinuity sets (edges) are lost. In view of \eqref{rigorousBV} such approaches can be characterized by ${\cal C}$ not being a bounded set in $L^\infty(\Omega)$, but rather being derivatives of bounded measurable functions. 

An alternative model trying to take advantage of total variation and higher-order total variation is a decomposition into two or more parts, i.e., $u=u_1+u_2$ with $u_1$ and $u_2$ being regularized differently. This has been proposed in this context for the first time in \cite{chambolle1997image} as an infimal convolution of first and second order total variation, the effective regularization functional is given by
$$ J(u) = \inf_{u_1+u_2 = u} \left( |u_1|_{BV} + |\nabla u_2|_{BV}\right) . $$
A popular alternative became the TGV-type models as proposed by \cite{Bredies:TGV}, which effectively do not decompose $u$ but the gradient measure $Du$ into $Du_1$ and some vector field $u_2$. One version of the regularization functional is then given by 
$$ J(u) = \inf_{Du_1+u_2 = Du} \left( |u_1|_{BV} + |u_2|_{BV}\right) . $$
The fact that the higher-order part is an arbitrary vector field provides additional freedom that can be benefitial compared to the infimal convolution model (cf. \cite{Bredies:TGV,benningbrunemueller,PhDthesisJoana}). We also mention that the original TGV-model in \cite{Bredies:TGV} does not use a bounded variation model for $u_2$, but only bounded deformations, i.e. the symmetric part of the gradient. Moreover, the approach can be formulated for arbitrary order of regularization. 
In the dual formulation \eqref{rigorousBV} approaches like infimal convolution or TGV still lead to ${\cal C}$ being a subset of the unit ball in $L^\infty(\Omega)$, which implies
$$ J(u)  \leq |u|_{BV} \qquad \forall u \in BV(\Omega). $$
On the other hand, for many of them a lower bound inequality can be shown at least when excluding a low (finite) dimensional nullspace (cf. \cite{benningbrunemueller}), i.e. there exists a positive constant $c$ and some linear functionals $\ell_i$ such that 
$$ J(u) \geq c |u|_{BV}  \qquad \forall u \in BV(\Omega), \quad \text{such that }\ell_i(u) = 0, i=1,\ldots M. $$
Hence, $J$ is an equivalent norm on the subspace of $BV$ excluding the nullspace. For the combination of first- and second-order derivatives the nullspace naturally consists of piecewise affine functions (thus $M=d+1$). For a further discussion and advanced aspects we refer to 
\cite{Bredies:TGV,benningbrunemueller,ranftl2013minimizing,bredies2014regularization,bredies2015tgv,bredies2015tgv2,burger2015infimal,burger2016infimal,bergounioux2016mathematical,SetzerSteidlTeuber,holler2014infimal,gao2017infimal,bergounioux2018anisotropic}.

In certain cases it is also interesting to use total variation regularization on some transform of the image. Motivated by research in image analysis taking into account orientations via local Radon transforms (cf. \cite{krause2008retinal}), in \cite{burger2014total} total variation regularization on the Radon transform respectively combined with total variation on the image itself was investigated to promote piecewise constant images with very thin structures resembling lines. In \cite{PhDthesisJoana} total variation on the spherical Radon transform (equivalent to circular Hough transform in computer vision) was investigated in order to reconstruct small circular structures.  

Another variant are total variation regularization methods for vector fields, e.g. arising for color images
(cf. \cite{bresson2008fast,blomgren1998color}), flow fields (cf. \cite{hinterberger2002analysis,zach2007duality}) or joint reconstruction problems (cf. \cite{knoll2017joint}). 
While many aspects remain the same as in the scalar case it is particularly interesting which matrix norm is used for $Du$, respectively which dual norm for $g$, noticing that this becomes a matrix in \eqref{rigorousBV}.

\subsection{Sparsity Regularization}

Total variation regularization, in particular its discrete version, can be interpreted as a functional favouring sparsity, in this case of the gradient. The paradigm of sparsity has developed in parallel to the total variation regularization (cf. \cite{donoho1992superresolution}). A key insight driving sparsity priors was the (approximate) sparsity of signals and natural images in wavelet bases (cf. \cite{mallat93,huang1999statistics,mallat2008wavelet,starck2010sparse}). Further improvements were made by replacing the orthonormal bases by frames (cf. \cite{christensen2003introduction}) such as curvelets (cf. \cite{candes2000curvelets,candes2000curveletsb}) or shearlets (cf. \cite{labate2005sparse,guo2007optimally,kutyniok2012introduction}). 

Sparsity is naturally measured by the $\ell^0$-norm, the number of nonzero entries. Since the minimization of $\ell^0$ is highly non-convex and even NP complete, it is usually relaxed to the convex $\ell^1$-norm. In the {\em analysis formulation} a frame system $\phi_i$ is used to test sparsity of $\langle u, \phi_i \rangle$, the corresponding regularization functional is given by
\begin{equation}
  J(u) = \sum_i |\langle u, \phi_i \rangle|.
\end{equation}
If $(\phi_i)$ is an orthonormal system, this is equivalent to the {\em synthesis formulation}, which is based on writing 
\begin{equation}
  J(u) = \sum_i \vert c_i \vert \qquad \text{where } u= \sum_i c_i \phi_i.
\end{equation}
Note that in general the two formulations may differ for  frames (cf. \cite{elad2007analysis}).

In the analysis formulation we can effectively define the variational problem on the coeffcient vector $c$, i.e.
$$ \tilde K: \ell^2(\N) \rightarrow \range, \quad c \mapsto \sum_i c_i K\phi_i $$
and compute
$$ \hat u = \sum_i \hat c_i \phi_i, \quad c_i \in \text{arg}\min_c \fidelfct(\tilde K c, f) + \alpha \vert c \vert_1. $$
The corresponding optimality condition is given by
$$
(\tilde K^* \partial_x \fidelfct (\tilde Kc,f))_i + \alpha s_i = 0, $$
with $s_i$ being a multivalued sign of $c_i$, i.e. an element of $[-1,1]$ for $c_i =0$. If $\tilde K$ is a bounded linear operator on $\ell^2(\N)$, then its adjoint maps into the same space, and hence $(s_i) \in \ell^2(\N)$. This implies in particular that $\vert s_i \vert < 1 $, hence $c_i = 0$, for $i$ sufficiently large. Thus, we always obtain some sparsity with this model. 

In the analysis formulation the optimality condition is given by
$$ K^* \partial_x \fidelfct (Ku,f) + \alpha s_i \phi_i = 0,$$
instead, with $s_i$ being a multivalued sign for $\langle u, \phi_i \rangle$. Here the understanding of the sparsity property is more complicated, the $s_i$ are actually related to the residual via the linear system
$$ \sum_j \langle \phi_i , \phi_j \rangle s_j = - \frac{1}\alpha \langle K \phi_i, \partial_x \fidelfct (Ku,f) \rangle. $$
We refer to \cite{vaiter2013robust,vaiter2013local} for a detailed analysis in this case. 
Sparsity models for inverse problems have been studied with different frames and applications extensively in the last decade (cf. e.g. \cite{cotter2005sparse,chaux2007variational,colonna2010radon,recht2010guaranteed})

There are several relevant extensions of sparsity priors to multidimensional systems, in particular in a synthesis type formulation
$$ u = \sum_{i,j} c_{ij} \phi_i \otimes \psi_j .$$
The different dimensions are often space (characterized by basis functions $\phi_i$) and time or frequency (characterized by basis functions $\psi_j$). Instead of overall sparsity more detailed prior knowledge can be introduced. The most popular example is joint or collaborative sparsity, which means that only few of the basis functions, e.g. in the second dimension, can be used to explain the solution. This means that $c_{\cdot j}$ vanishes for most $j$, respectively also any norm of it. A common regularization for this case is the joint or collaborative sparsity prior
$$ J(u) = \sum_j \Vert c_{\cdot j} \Vert_{\ell^r}, $$
usually with $r=2$ or $r=\infty$ (cf. \cite{duarte2005joint,teschke2007iterative,fornasier2008recovery,gholami2010regularization,lee2011compressive})
An alternative type of prior knowledge is local sparsity, which means that for each $i$ only few basis functions $\psi_j$ are used. The term local is due to an imaging interpretation of the $\phi_i$ as basis functions local in space (e.g. for each pixel). This is a common issue in dynamic or spectral imaging, where one can assume that only few materials and their characteristic evolutions or spectral curves can be found in each pixel. A regularization functional proposed for this issue (cf. \cite{heins2015locally}) is
$$ J(u) = \max_i \Vert c_{i \cdot } \Vert_{\ell^1} +  \beta \sum_i \Vert c_{i \cdot } \Vert_{\ell^1}.$$

An infinite-dimensional extension of the above sparsity models is sparsity in a space of Radon measures, i.e. the regularization functional is given as the total variation norm of the measure $u$ 
$$ J(u) = \int_\Omega d|u| =\sup_{g \in C_0(\Omega), \Vert g \Vert_\infty \leq 1} \int_\Omega g ~du.$$
This yields a convex regularization functional for reconstructing multiple peaks at unknown locations and has been proposed for inverse problems in \cite{bredies2013inverse}, respectively for superresolution problems in \cite{candes2013super,candes2014towards,Fernandez:Noise}. The reconstruction properties in deconvolution problems have been analyzed in \cite{duval2015sparse,denoyelle2017support}, asymptotics from finite-dimensional problems with sparsity priors are found in \cite{heins2014reconstruction,duval2015sparse}.

\subsection{Low Rank Regularization}

In many applications one seeks a decomposition of the form
\begin{equation}
 U = \sum_{i} \Phi_i \otimes \Psi_i 
\end{equation}
with unknown $\Phi_i, \Psi_i$ and the additional prior knowledge that there are as few elements as possible in the sum. In a finite dimensional setting this means that the matrix $U$ has low rank, i.e. the rank of $U$ would be the obvious regularization functional. However, since the rank is very far from being convex, several relaxations have been proposed instead. The most popular one, originally proposed for matrix completion problems, is the nuclear norm (cf. \cite{candes2009exact,recht2010guaranteed,cai2010singular,cai2013fast,yang2013seismic})
\begin{equation}
\Vert U \Vert_* = \sum \sigma_i,
\end{equation}
where $\sigma_i$ are the singular values of $U$.

In many applications the low rank part alone does not suffice to model the structure of solutions, frequently a low rank plus sparsity (L+S) model is employed instead (cf. \cite{otazo2015low}), which is again based on a decomposition
\begin{equation}
 J(u) = \inf_{u_1+u_2 = u}  \left( \Vert u_1 \Vert_* + \Vert T u_2 \Vert_1 \right)
\end{equation}
with a sparsifying transform $T$ (often some derivative like in total variation). In particular in videos the low rank part captures background and certain slow dynamics, while the sparse part captures the key changes. 

For inverse problems an infinite-dimensional function space setting would be more appropriate, which has not yet been investigated. In particular a formulation in a space of trace class operators between Hilbert spaces $H_1$ and $H_2$ (cf. \cite{reed1978iv}) would be natural. Let us mention that the choice of Hilbert spaces $H_i$ opens novel opportunities for improved regularization that are so far unexploited, even in the finite-dimensional case.

\subsection{Infimal Convolutions} 

As we have seen above, infimal convolution is a versatile tool to combine different regularization approaches, respectively define a novel functional that combines their advantages. We want to highlight this approach in the following by providing formal definitions:

\begin{definition}
Let $J_i: \domain \rightarrow \R \cup \{+\infty\}$, $i=1,2$ be proper convex functionals. Then their infimal convolution
$J_1\square J_2:\domain \rightarrow \R \cup \{+\infty\}$ is defined via
\begin{equation}
 (J_1 \square J_2) (u) = \inf_{v \in \domain} (J_1(u-v) + J_2(v)).
\end{equation}
\end{definition}

Even more general, we can define an infimal convolution for an arbitrary number of convex functionals:
\begin{definition}
Let $J_i: \domain \rightarrow \R \cup \{+\infty\}$, $i=1,\ldots,M$ be proper convex functionals. Then their infimal convolution
$J:\domain \rightarrow \R \cup \{+\infty\}$ is defined via
\begin{equation}
 J(u) = \inf_{u_i \in \domain,\sum u_i = u}  \sum_{i=1}^M J_i(u_i)
\end{equation}
\end{definition}

We mention that a-priori it is unclear whether the infima above are actually minima. If a minimizer $v$ exists for the infimal convolution of $J_1$ and $J_2$, it can be used to deduce optimality conditions, since
$$p \in \partial J(u) \quad \text{if } p \in \partial J_1(u-v) \cap \partial J_2(v). $$ 
As the above examples for sparsity and in particular higher-order total variation show that there is quite some freedom in designing infimal convolution models for regularization. Consequently, a lot of options for future research remain open and interesting results are still to be expected.

\subsection{Bregman Distances}

From a single regularization functional several variants can be constructed by using a nontrivial prior $u_0$ and the so-called Bregman distance (originally introduced in \cite{bregman67} for proximal-point type methods). Instead of shifting the functional directly from $J(u)$ to $J(u-u_0)$, the approach in the Bregman distance performs a shift in the convex conjugate. In the original formulation this amounts to the following:
\begin{definition}
Let $J: \domain \rightarrow \R \cup \{+\infty\}$ be a convex functional and let $p_0 \in \partial J(u_0)$. Then the Bregman distance between $u \in \domain$ and $u_0 \in \domain$ with subgradient $p_0$ is given by
\begin{equation}
\distance_{J}^{p_0}(u,u_0) := J(u) - J(u_0) - \langle p_0, u - u_0 \rangle
\end{equation}
\end{definition}
Note that the Bregman distance is not a strict distance, i.e. it can vanish for $u \neq u_0$ if $J$ is not strictly convex. It is also not symmetric, but can be made symmetric by taking a sum of one-sided distances (cf. \cite{burger2016bregman} for a more detailed discussion). For absolutely one-homogeneous regularization functionals as above, the identity
$J(u_0) = \langle p_0, u_0 \rangle$ holds, thus the Bregman distance becomes 
\begin{equation} \label{eq:Bregmanonehomog}
\distance_{J}^{p_0}(u,u_0) := J(u) - \langle p_0, u  \rangle,
\end{equation}
thus it is effectively independent of $u_0$, only the subgradient $p_0$ matters. This is particularly relevant if the subdifferential of $J $ is not a singleton or vice versa a subgradient $p_0 \in \partial J(u_0)$ can be an element of the subdifferential also at other values of $u$.

Note that in the typical case of $u_0 =0$ being a minimizer of $J$, i.e. $0 \in \partial J(0)$, the regularization with $J$ can be reinterpreted as penalizing the Bregman distance to $u_0=0$. In
\cite{bleyer2009tikhonov} a basic analysis of such a variational regularization was carried out. The topic received recent interest in particular in the context of TV-type regularization in imaging, since it allows to introduce structural information. The key insight in total variation is that the subgradient encodes information about the discontinuity set, more precisely $p=\nabla \cdot g$, with $g$ being equal to the unit normal vector to the discontinuity set where it is regular. This is again related to \eqref{eq:Bregmanonehomog}, the total variation does not depend directly on $u_0$ and in particular the contrast in the image. It rather vanishes for all $u$ of the form
$$ u(x) = f(u_0(x))$$ 
with a monotonically increasing function $f$, i.e. a simple contrast change (cf. \cite{resmerita06}). Assuming that $g$ is a vector field realizing the supremum in the dual definition of the total variation, the Bregman distance becomes 
$$
\distance_{TV}^{p_0}(u,u_0)  = |u|_{BV} - \int_\Omega( \nabla \cdot g_0 ) u~dx = 
\int_\Omega( \nabla \cdot (g-g_0) ) u~dx,
$$
and if $u$ is piecewise constant with regular discontinuity set $S_u$ 
$$ 
\distance_{TV}^{p_0}(u,u_0)  =  \int_{S_u} [u]
(g-g_0) \cdot\nu~d\sigma = \int_{S_u} [u]
(1-g_0 \cdot \nu)~d\sigma,
$$
where $[u]$ denotes the jump along $S_u$ and $\nu$ the unit normal (oriented such that $[u]$ is positive). One thus observes that the Bregman distance measures differences in the discontinuity set and its orientation, which is perfect for imaging applications with a structural prior (cf. \cite{kaipio1999inverse}) that mainly yields information about edges, i.e. discontinuity sets. An example are anatomical priors in medical imaging, where a high resolution modality such as CT or MR is used to obtain information about organ boundaries and other anatomical features, which are the natural candidates for edge sets in functional modalities like PET, SPECT, or MR imaging with special contrast. In some cases also a joint reconstruction is of interest, the most obvious case being color or hyperspectral images, where naturally intensity changes at the same locations, usually even in the same direction (cf. \cite{moeller12,Moeller:ColorBregmanTV}). 

In some applications one may find contrast inversion, i.e., the jump of the two images along the discontinuity set has different sign. In such cases the normals are parallel, which means they point into opposite directions and hence lead to large values in the Bregman distance. A potential solution to avoid such issues is the infimal convolution of Bregman distances, in this cases with the two normal fields and thus subgradients of opposite sign (cf. \cite{Moeller:ColorBregmanTV,rasch2017dynamic})
$$ J = \distance_{TV}^{p_0}(\cdot,u_0) \square \distance_{TV}^{-p_0}(\cdot,-u_0) .$$
We also mention some other related approaches to modify total variation functionals such as the parallel level set models (cf. \cite{ehrhardt2014vector,ehrhardt2014joint,ehrhardt2016pet}), which can be related to the Bregman distance for total variation (cf. \cite{rasch2017joint}), or  
directional / structural total variation (cf. \cite{bungert2017blind,ehrhardt2016multicontrast,hintermuller2017function,grasmair2010anisotropic}), formally 
$$
TV_{g_0}(u) = \int_\Omega \vert (I-g_0 \otimes g_0) \nabla u \vert ~dx,
$$

\section{Fundamentals of Nonlinear Regularization}

Before discussing the detailed analysis of nonlinear regularization methods, we first aim at providing a suitable basis on how to understand regularization methods and their convergence. We start with the case of linear regularization methods in Hilbert spaces, recalling the abstract theory from \cite{Engl:Regularization}, and then try to work out a suitable analogue for the nonlinear case in Banach spaces.

\subsection{Abstract Linear Regularization Methods}

We start our exposition with a discussion of possible limits of regularization schemes.
In basically all linear methods such as Tikhonov regularization, truncated SVD or iterative regularization in Hilbert spaces it is clear which solutions are approximated as the regularization parameter tends to zero, namely the ones obtained from a generalized inverse. 
The following definitions are made to characterize these limiting solutions:
\begin{definition} Let $K:\domain \rightarrow \range$ be a bounded linear operator between Hilbert spaces and $f \in \range$. We call $\hat u \in \domain$ a {\em best approximate solution} of \eqref{eq:basicequation} if
\begin{equation}
\Vert K\hat u - f \Vert_{\range} \leq \Vert Ku - f \Vert_{\range}, \qquad \forall~ u \in \domain.
\end{equation}
Moreover, we call $\hat u$ a {\em minimal norm solution} if it is a best approximate solution and
\begin{equation}
\Vert \hat u \Vert_{\domain} \leq \Vert u \Vert_{\domain} \qquad \forall~ u \in \domain, ~\Vert K\hat u - f \Vert_{\range} =\Vert Ku - f \Vert_{\range}.
\end{equation}
\end{definition}
Note that due to the strict convexity of the square of a Hilbert space norm, the minimum solution - being its minimizer on a linear manifold - is a unique object. An abstract regularization method is now a collection of continuous operators approximating the (discontinuous) generalized inverse of $K$:
\begin{definition}
A family of bounded linear operators $R_\alpha: \range \rightarrow \domain$ defined for $\alpha$ in $(0,\alpha_0)$ is called {\em linear regularization operator}. Together with a parameter choice strategy $\alpha$ depending on the noise level $\delta$ and the data $f^\delta$, i.e., a function
\begin{equation}
\alpha: (0,\delta_0)  \times \range \rightarrow (0,\alpha_0)
\end{equation}
it is called {\em linear regularization method}. 

A linear regularization method is called {\em convergent}, if for all $f \in \oprange(K)$ the condition  
\begin{equation}
\lim_{\delta \rightarrow 0} \sup \{ \Vert R_{\alpha(\delta,f^\delta)}(f^\delta) - \minsol \Vert_{\domain} ~|~  
 f^\delta \in \range, \Vert f -f^\delta \Vert_{\range} \leq \delta \} = 0
\end{equation}
holds with $u^*$ being the minimum norm solution of \eqref{eq:basicequation}.
\end{definition}

For ill-posed problems it is well-known that convergence can be arbitrarily slow (cf. \cite{schock1985approximate}). Thus, convergence rates can be obtained only on a restricted subset  $M_\nu$ with a parameter $\nu > 0$ measuring the smoothness respectively order of convergence. The standard definition is given by: 
\begin{definition} 
A regularization method is called {\em convergent at order} $\nu$ on a set ${\mathcal M}_\nu$ if for all $f=K\minsol$, $\minsol \in {\mathcal M}_\nu$, there exists a constant $C_\nu$ such that for all data $f^\delta$ with $\Vert f^\delta - f \Vert \leq \delta$ the estimate
\begin{equation} 
\Vert \R_{\alpha(\delta,f^\delta)}(f^\delta) - \minsol \Vert \leq C_\nu \delta^\nu,
\end{equation}
holds.
\end{definition}
It is well-known that the set $M_\nu$ can be related to the source condition
$$ \minsol = (K^*K)^\mu w $$ for some $w \in \domain$ and appropriate $\mu > 0$ related to $\nu$ (cf. \cite{Engl:Regularization}). The constant $C_\nu$ is then related to the norm of $w$. The simplest cases of source conditions are $\mu = \frac{1}2$, which can be reformulated as 
$$ \minsol = K^* \tilde w, $$
for some $\tilde w \in \range$, and the case $\mu=1$. Source conditions induce conditional well-posedness of the problem, e.g. for $\mu=\frac{1}2$ one has for $u_i = K^* \tilde w_i$
$$ \Vert u_1 - u_2 \Vert^2 = \langle u_1 - u_2, K^* (\tilde w_1- \tilde w_2) \rangle = \langle K( u_1 - u_2), \tilde w_1-\tilde w_2 \rangle. $$
The Cauchy-Schwarz and triangle inequality then imply the H\"older stability
$$ \Vert u_1 - u_2 \Vert \leq C \sqrt{\Vert Ku_1 - K u_2 \Vert},$$
with $C=\sqrt{\Vert \tilde w_1\Vert +\Vert \tilde w_2\Vert }. $

\subsection{Extension to Nonlinear Methods}

The examples of variational regularization models in the previous section call for a more general theory of nonlinear regularization methods. While the concept of a best-approximate solution is rather straightforward to generalize, other aspects of convergence and limiting solutions are less obvious. In a general variational regularization, as in the examples discussed above, it would be natural to replace the minimum norm solution by a solution minimizing the regularization functional. The latter is not necessarily unique however, hence some possible multi-valuedness needs to be introduced in the characterization. Similar issues apply to the regularized problem and hence the definition of a regularization operator.
In the following we will try to provide a fundamental setting for nonlinear regularization methods. As in the case of linear regularizations we first generalize the possible types of solutions we would like to approximate. The generalization of the first notion is rather straightforward, we only allow for more general distance measures, e.g. functionals related to negative log-likelihoods for non-Gaussian distributions: 
\begin{definition} \label{eq:bestapproximatedef}
Given an error measure $\fidelfct: \range \times \range \rightarrow \R_+ \cup \{+\infty\}$, we call $\hat u \in \domain$ a {\em best approximate solution} of \eqref{eq:basicequation} with respect to $\fidelfct$ if 
\begin{equation}
\fidelfct(K \hat u, f) \leq \fidelfct(Ku, f) \qquad \forall~u \in \domain.
\end{equation}
\end{definition}

A suitable generalization of the definition of a minimum norm solution is more involved, in particular we would like to give a unified concept including the selection via minimizing a regularization functional or maximizing some prior probability. We encode the selection of specific solutions due to prior knowledge in a (multivalued) selection operator:
\begin{definition}
A multivalued operator $\select: \oprange(K) \rightrightarrows \domain$ is called {\em selection operator} if $\select(Ku) \subset u + \nullspace(K)$ for all $u \in \domain$. A best approximate solution $\hat u$ is called {\em prior selected solution} of \eqref{eq:basicequation} if and only if $\hat u \in \select{(K\hat u)}$.
\end{definition}

The general set-valued definition of a selection operator, which we use in order to take care of all the possible cases in regularization methods, also needs to use set-valued ways of convergence. For this sake we recall the definition of {\em Kuratowski convergence} in a metric space:
\begin{definition} Given a metric space $X$ with metric $d$ and - by abuse of notation  - for $x \in X$ and $S \subset X$
\begin{equation}
d(u,S) := \inf_{v \in S} d(u,v), 
\end{equation}
the Kuratowski limit inferior and superior of a sequence of sets $S_n \subset X$ are defined as follows:
\begin{align}
K-\lim\inf_n (S_n) &= \{ x \in X~|~\lim\sup_n d(x,S_n) = 0 \} \\
K-\lim\sup_n (S_n) &= \{ x \in X~|~\lim\inf_n d(x,S_n) = 0 \}. 
\end{align}
\end{definition}
For our sake the limsup will be of particular interest, we will use a minimal definition of stability often adopted in the literature on nonlinear methods after Seidman and Vogel \cite{seidman1989well} respectively Engl, Kunisch, and Neubauer \cite{engl1989convergence}. Stability is expressed by subsequences of selected solutions having a limit and each limit of a subsequence being a solution of the limiting problem. The liminf is less interesting, since there is no reason to ask that any solution of a problem can be the limit of approximate problems. We call an inverse problem stable if for $f_n \rightarrow f$ (usually in terms of norm convergence in $\range$) we have that 
\begin{equation} \label{eq:stabilitydefinition}
K-\lim\sup_n \select(f_n) \subset \select(f), \quad \text{ and } K-\lim\sup_n \select(f_n)  \neq \emptyset.
\end{equation}
The metric used for the Kuratowski limsup will usually be a metrization of some weak or even weak-star convergence in a Banach space, one might also use an extension of the definition to other distance measures.


Having defined what are the solutions we would like to approximate, the obvious next step is to define what actually is a (convergent) regularization method. We start in a deterministic setting, generalizing to a vectorial regularization parameter $\regparambold \in \R_+^M$ however, which is useful in many examples, e.g. the TGV and infimal convolution models with multiple parameters mentioned above. 
Given an error measure $\fidelfct$ and $\clean=K\minsol$ for some exact solution $\minsol \in \domain$, we call $\delta > 0$ noise level if it is the best available bound for available data $\noisy$, i.e.,
\begin{equation} \label{noiselevel}
\fidelity{\clean}{\noisy} \leq \delta.
\end{equation}
We will be interested  in the convergence of regularized solutions to prior selected solutions as the noise level tends to zero. For the ease of presentation and since this is available in almost any known example, we restrict ourselves to convergence with respect to a metric topology $\tau$, which is usually a weak or weak-star topology (on some bounded set in the Banach space). 

\begin{definition}
A family of multivalued operators $\regop(\cdot,\regparambold): \range \rightrightarrows \domain$ defined for $\regparambold$ in a subset $\regdomain$  of $\R^M$ is called {\em regularization operator}, if for each $\regparambold \in \regdomain$ the operator $\regop$ satisfies the stability property
\begin{equation}
\emptyset\neq K-\lim\inf_n \regop(f^{\delta_n}) \subset \regop(\noisy)
\end{equation} 
for all $\noisy \in \range$ and sequences $f^{\delta_n} \in \range$ converging to $\noisy$.
Together with a parameter choice strategy $\regparambold$ depending on the noise level $\delta$ and the data $\noisy$, i.e., a function
\begin{equation}
\regparambold: (0,\delta_0)  \times \range \rightarrow \regdomain,
\end{equation}
it is called {\em regularization method}. 

A regularization method is called {\em convergent}, if for all sequences $\delta_n \rightarrow 0$, data $f^{\delta_n}$ satisfying
\begin{equation} 
\fidelity{\clean}{f^{\delta_n}} \leq \delta_n,
\end{equation}
we have 
\begin{equation}
\emptyset\neq K-\lim\inf_n R_{\regparambold(\delta_n,f^{\delta_n})}(f^{\delta_n}) \subset \select(\clean).
\end{equation} 
\end{definition}
We mention that - besides the very general setup - our definition of a regularization method deviates from the usual theory since we do not assume any kind of convergence of the regularization parameter $\regparambold$. In the classical theory and most examples $\regparambold$ is a scalar positive value and assumed to converge to zero (or to infinity) as the noise level tends to zero. However, apart from the convenience there seems to be no reason to put such convergence into the definition. Note that in order to approximate a really ill-posed problem each clustering point of $\regparambold(\delta_n,f^{\delta_n})$ will automatically lie outside $\regdomain$. The canonical examples are $\regdomain=(0,\regparam_0)$ for variational regularization or $\regdomain=\mathbb{N}$ for iterative regularization, where the limiting parameter will converge to zero or infinity. However, we may also consider multi-parameter regularization, where it depends on the formulation whether each component of $\regparambold$ has a limit outside the admissible set. Take for example an infimal convolution of two functionals $R_1$ and $R_2$. If $\regparambold = (\regparam_1, \regparam_2)$ are the coefficients of $R_1$ and $R_2$, then obviously both should tend to zero in the limit. If however $\regparam_2$ is a relative parameter, i.e. $\regparam_1$ is the coefficient of $R_1$ and $\regparam_1 \regparam_2$ the coefficient of $R_2$, then it is natural to have a positive limit of $\regparam_2$. Another motivation for our general definition are recent approaches to learning regularization methods for inverse problems, where the $\regparambold$ can represent the parameters of the learning scheme. To get a consistent infinite-dimensional theory one could even generalize to non-parametric learning that would amount to choosing $\regparambold$ in some Banach space. Note that in the remainder of this article we will often write $\regparam$ instead of $\regparambold$ if $\regparambold$ is only a scalar.

In order to define convergence rates we will further need an error measure $\distance: \domain \times \domain \rightarrow \R_+ \cup \{+\infty\}$, since there is no natural norm measure as in the Hilbert space. Moreover, we need a restriction to appropriate classes of smoothness, which we denote by $M_\nu$ with a parameter $\nu > 0$ measuring the smoothness. 
\begin{definition}\label{def:dconvergence}
A regularization method is called $D$-convergent if 
\begin{equation}
\lim_{\delta \rightarrow 0} \sup \{ D( \regsol_\delta, \minsol ) ~|~  
\regsol_\delta \in \regoparg{\noisy}, \noisy \in \range, \fidelity{\clean}{\noisy} \leq \delta \} = 0.
\end{equation}

A regularization method is called {\em convergent at order} $\nu$ on a set if for all $f=K\minsol$, $\minsol \in {\mathcal M}_\nu$, there exists a constant $C_\nu$ such that for all data $g$ with \eqref{noiselevel} the estimate
\begin{equation} 
D(\regoparg[\regparambold(\delta,\noisy)]{\noisy},\minsol) \leq C_\nu \delta^\nu,
\end{equation}
holds.
\end{definition}
Of course the above definition only makes sense for suitable choices of the distance functional and the smoothness classes. Remember that in the classical linear Hilbert space theory those were just norms and spaces obtained by source conditions. We will discuss generalizations of such in the nonlinear setting in particular related to variational and iterative regularization methods in Banach spaces related to convex regularization functionals. Note also that more general rates than just polynomial ones have been considered in the literature (cf. e.g. \cite{hohage1997logarithmic,kaltenbacher2008note}). 

From an abstract point of view the key insight to generalize source conditions is the range of the regularization operator. It is easy to see for many linear regularization methods in Hilbert spaces that the source condition $\minsol = K^*\tilde w$ means that there exist some data $f^\dagger$ with $\minsol=R(f^\dagger,\regparambold)$. As examples take Tikhonov regularization
$$ \regop(\cdot,\regparambold) = (K^*K + \alpha I)^{-1} K^* = K^* (KK^* + \alpha I)^{-1}.$$ 
Due to the invertibility of $(KK^* + \alpha I)^{-1}$ the range of the regularization operator coincides with the range of $K^*$. Instead of defining source conditions at an abstract level we thus make the following 
\begin{definition}[Range condition]
An element $\minsol \in \select(\clean, \regparambold)$ for $\clean \in \oprange(K)$ satisfies the \emph{range condition} if $\minsol \in \oprange(\regoparg{\cdot})$, i.e. there exists $f_\regparambold^\dagger$ such that
$$ \minsol \in \regoparg{f_\regparambold^\dagger}.$$
\end{definition}
We mention that in the case of nonlinear variational methods (with quadratic fidelity), the equivalence of a nonlinear source condition and the range condition was shown in \cite{burgerosher}, confirming again the appropriateness of this definition.

Roughly speaking error estimates can now be obtained by some continuity property of the regularization operator, which implies
$$ d_{\domain}(u_\regparambold^\delta,\minsol) \leq C(\alpha) d_{\range}(\noisy,f_\regparambold^\dagger), $$
with appropriate distances $d_{\domain}$ and $d_{\range}$. With some kind of triangle inequality the right-hand side can be estimated by a distance between $\clean$ and $\noisy$, which is related to the noise level as well as a distance between $\clean$ and $f_\regparambold^\dagger$, which is related to the bias of the regularization. This will be discussed in detail for the case of variational regularization methods in Section 5.
A weaker concept are approximate source conditions (cf. e.g. \cite{Schusterbuch,burger2016large}) that effectively measure how well the range condition can be approximated. On the other hand stronger conditions can be obtained if $f_\regparambold^\dagger$ above is not arbitrary but in the range of the forward operator $K$.

\subsection{Stochastic Approaches}

In addition to the deterministic viewpoint a statistical approach has become popular also in infinite-dimensional problems more recently (cf. \cite{bissantz2004consistency,bissantz2007convergence,Cavalier2008,kekkonen14,gine2015mathematical,hohage2016inverse}). In such a setup the data $\noisy$ are considered to be random variables drawn from a measure $\mu_f$ centered around the exact data $f$ (often representing the expected value and $\delta$ some kind of variance). A regularization operator can then still be applied to each realization and defined in the same way, but we need a different definition of the noise level and the convergence of the regularization method. As a generalization of variance we use the statistical noise level in the mean
\begin{equation} \label{statisticalnoiselevel}
\mathbb{E}(\fidelfct(f,f^\delta)) = \delta.
\end{equation}

\begin{definition}
A regularization operator $\regop$ with a parameter choice strategy $\regparambold$ depending on the statistical noise level $\delta$ and the data $\noisy$, i.e., a function
\begin{equation}
\regparambold: (0,\delta_0)  \times \range \rightarrow \regdomain
\end{equation}
is called {\em statistical regularization method}. 

A statistical regularization method is called {\em convergent} if for all sequences $\delta_n \rightarrow 0$, random variables $f^{\delta_n}$ satisfying
\begin{equation} 
\mathbb{E}(\fidelfct(\clean,f^{\delta_n})) \leq \delta_n,
\end{equation}
and each choice of random variables $u_n \in R_{\regparambold_n}(f^{\delta_n})$ there exists a convergent subsequence $u_{n_k}$ in probability in the topology $\tau$ and the limiting random variable $\minsol$ satisfies $\minsol \in \select(f)$ with probability one.
\end{definition}

An extension of this viewpoint is the Bayesian approach to inverse problems, which does not only deal with point estimates, but analogous question for the full posterior distributions. This topic is beyond the scope of this survey, we refer to \cite{Kaipio:InverseProblems,neubauer2008convergence,Stuart,KLNS12,Castillo2014,castillo2014bernstein,kekkonen15,burger2016large,nickl2017nonparametric} for further details.

\section{Variational Regularization Methods}
 
We now return to \eqref{eq:basicvariational} with the viewpoint as in the previous section, we show how variational methods define a regularization operator and then proceed to its further analysis. 
In this canonical variational regularization method it is apparent how to choose the best approximate and prior selected solution according to Definition \ref{eq:bestapproximatedef}. First of all, the distance measure in the definition of the best approximate solution clearly coincides with the data fidelity. It is just the solution of the variational problem for $\regparambold$ in the boundary of $A$, in the simplest case of a scalar regularization parameter usually $\alpha =0$. Of course, the existence of such an element is not obvious, for this sake we define an effective range of the forward operator as 
\begin{equation}
\oprange_\fidelfct(K) = \left\{ f \in \range~\left|~\argmin_{u \in \domain, J(u) < \infty}   \fidelity{Ku}{f} \neq \emptyset \right. \right\}. 
\end{equation}

The selection operator is constructed by minimizing the regularization functional on the set of best approximate solutions. Let $f \in \oprange_\fidelfct(K)$,  then we define
\begin{equation} \label{eq:variationalselect}
\select(\clean, \regparambold) = \argmin_{u \in \domain} \left\{\regfctarg{u}~\left|~ u \in \argmin_{\tilde u \in \domain}   \fidelity{K\tilde u}{f} \right.  \right\}
\end{equation}
\begin{remark}
We want to point out that if $\regparambold = \regparam$ is just a scalar, the selection operator does not depend on $\regparam$ for regularization functionals of the form $\regfctarg[\regparam]{u} = \regparam \regfct_1(u)$. In this particular case we simply have
\begin{equation*}
\select(\clean) = \argmin_{u \in \domain} \left\{\regfct_1(u)~\left|~ u \in \argmin_{\tilde u \in \domain}   \fidelity{K\tilde u}{f} \right.  \right\} \, ,
\end{equation*}
as the minimizer is not affected by multiplication with a positive scalar. As mentioned above there are also cases where the selection operator only requires a subset of the parameters as its argument, for example in case of infimal convolution regularizations of the form $\regfctarg{u} := \inf_v \regparam_1 \left( \regfct_1(u - v) + \regparam_2 \regfct_2(v) \right)$, for $\regparambold = (\regparam_1, \regparam_2)$ and $\regdomain = (0, \infty) \times (0, \infty)$. Here $\select(\clean, \regparambold) = \select(\clean, \regparam_2)$ only depends on $\regparam_2$.
\end{remark}
We will show below that this selection operator is well-defined under standard conditions, which are also used to analyze the variational regularization method.


Following up on variational modeling as described in Section \ref{sec:varmod}, we define a generic variational regularization operator as follows.

\begin{definition}[Variational Regularization]\label{def:varreg}
Let $\fidelfct:\range \times \range \rightarrow \R_+ \cup \{\infty\}$ be continuous with $F(f,f)=0$ for all $f \in \oprange_\fidelfct(K)$ and $\regfct:\domain \times \regdomain \rightarrow \Rinf$ be proper, lower semi-continuous and convex functionals, and let $K \in \linbound{\domain}{\range}$. Then the potentially set-valued operator $\regop:\range \times \regdomain \rightrightarrows \domain$ defined as 
\begin{align}  
\regoparg{\noisy} := \argmin_{u \in \domain} \left\{ \fidelity{Ku}{\noisy} + \regfctarg{u} \right\}\label{eq:varreg}
\end{align}
is said to be a \emph{variational regularization}, for fixed regularization parameter(s) $\regparambold \in \regdomain$.
\end{definition}

\begin{remark}
We want to highlight that for convex $J$ and $F$ that is convex in its first argument any $\regsol \in \regoparg{\noisy}$ can equivalently be characterized via the optimality condition of \eqref{eq:varreg}, i.e.
\begin{align}
- K^\ast \partial_x \fidelity{K\regsol}{\noisy} \in \partial \regfctarg{\regsol}\label{eq:varregopt}
\end{align}
for all $\regsol \in \regoparg{\noisy}$.
\end{remark}


\subsection{Analysis of Variational Regularization}

In the following we will discuss the basic analysis of variational regularization methods, again we try to give a rather general perspective that covers most of the results in literature (but due to its generality does not simply reproduce them). Since we focus on the nonlinear regularization we will make the assumption that $\range$ is a separable Hilbert space. A first key issue is the existence of minimizers, which of course depends strongly on the choice of the regularization functional $\regfct$ and possibly also the operator $K$ and the fidelity $F$. As usual the key issues are lower semicontinuity and compactness in some topology. The latter is always obtained by coercivity in a Banach space norm, which is concluded from the boundedness of the fidelity and in particular the regularization functional. Consequently, the type of compactness is always weak or weak-star, since it is derived from the Banach-Alaoglu theorem (cf. \cite{rudin}).

A natural assumption to make for an existence proof is the following:
\begin{assumption} \label{existenceassumption}
Let $\domain = Z^*$ for some normed space $Z$ and let the weak-star topology on $\domain$ be metrizable on bounded sets. Assume moreover
\begin{itemize}
\item $K=L^*$ for a bounded linear operator $L:\range \rightarrow Z$.

\item $\regfct = H^*$ for some proper functional $H: Z \rightarrow \R \cup \{+\infty\}$ and $\regfct$ is nonnegative.

\item $\fidelfct$ is a proper, nonnegative, convex functional and for every $g \in \range$ there exists $u$ with  
$$ \fidelfct(Ku,g) + J(u,\regparambold) < \infty. $$

\item For each $g \in \range$ and $\regparambold \in { A}$, there exists a constant $c=c(a,b,\Vert g\Vert)$ depending monotonically non-decreasing on all arguments such that 
$$ \Vert u \Vert_{\domain} \leq c \qquad \text{if } \fidelity{Ku}g \leq a, \quad\regfctarg{u} \leq b. $$
\end{itemize}
\end{assumption}
Note that the above assumptions on $K$ and $\fidelfct$ are reminiscent of the setup used by \cite{bredies2013inverse} and later by \cite{brinkmann2017bias}. An alternative setup is to use a compactness assumption on $K$ or some condition on the range of $K$.
Moreover, the assumption on $\regfct$ to be the polar of a proper functional implies convexity, which is predominant in most approaches in regularization theory.  With these assumptions we can first verify well-posedness of the selection operator.

\begin{lemma} \label{selectlemma}
Let Assumption \ref{existenceassumption} be satisfied. Then for every $f \in \oprange_\fidelfct(K)$ the selection operator ${\cal S}$ is well-defined by \eqref{eq:variationalselect} for every $\regparambold \in A$.
\end{lemma}
\begin{proof}
If $f \in \oprange_\fidelfct(K)$ then there exists a minimizer $u^*$ of $\fidelfct(Ku,f)$ with $J(u^*,\regparambold) < \infty$. Since the minimization in the definition of ${\cal S}$ can be restricted to the set 
of $u$ such that $\fidelfct(Ku,f) =\fidelfct(Ku^*,f)=:a$, we obtain an upper bound on the fidelity. On this nonempty set we look for $u$ with $J(u) \leq J(u^*)=:b$. Thus, for the set of such $u$, the norm in $\domain$ is bounded due to  Assumption \ref{existenceassumption} and for each minimizing sequence there exists a weak-star convergent subsequence $u_n$ (we can use the metric version of the Banach-Alaoglu theorem due to the assumption of metrizability on bounded sets). 
Moreover, from our assumptions above it is straight-forward to see that $J(\cdot,\regparambold)$ is sequentially weak-star lower semicontinuous and $\fidelfct(\cdot,f)$ is weakly lower semicontinuous. From our assumption on  $K$ being the adjoint of $L$ we see that it is continuous from the weak-star topology of $\domain$ to the weak topology of $\range$, since for $g \in \range$, because
$$ \langle K u_n , g \rangle = \langle u_n , Lg \rangle $$
and $Lg \in Z$.
As a consequence, the full functional $\fidelfct(\cdot,f)+J(\cdot,\regparambold)$ is weak-star lower semicontinuous. Hence, the weak-star limit of $u_n$ is a minimizer, i.e. ${\cal S}$ is not empty. 
\end{proof}

The next step is to verify well-definedness of the regularization operator:
\begin{theorem}
Let Assumption  \ref{existenceassumption} be satisfied. Then for every $f \in \range$ the variational regularization model has a minimizer in $\domain$ for every $\regparambold \in A$, i.e., the regularization operator $R$ is well-defined by \eqref{eq:varreg}. Moreover, $R(f,\regparambold)$ is a convex set. 
\end{theorem}
\begin{proof}
In order to obtain an a-priori bound we use the assumption that there exists $\tilde u$ with 
$$ a:= \fidelfct(K\tilde u,f) + J(\tilde u,\regparambold) < \infty.$$ 
Hence, we can restrict the minimization to those $u$ with functional value less or equal $a$. Setting $b=a$ and using the nonnegativity of both terms we obtain the boundedness of the norm on this subset due to  Assumption \ref{existenceassumption}. The remaining weak star compactness and lower semicontinuous arguments to verify the existence of a minimizer are analogous to the proof of Lemma \ref{selectlemma}. The convexity of $R(f,\regparambold)$ follows from the convexity of the set of minimizers of a convex functional. 
\end{proof}

In order to verify the generalized stability as well as the convergence of the variational regularization, a further condition on $\fidelfct$ with respect to the second variable is needed. There are several options, the easiest one being satisfied by standard examples such as squared norms is continuity. 
\begin{theorem} \label{theorem:stability}
Let Assumption  \ref{existenceassumption} be satisfied and let $\fidelfct$ be continuous with respect to the second variable. Then for $\regparambold \in A$ and every sequence $f_n \rightarrow f \in \range$ there exists a subsequence $u_{n_k} \in R(f_{n_k},\regparambold)$ converging to an element $u^* \in   R(f,\regparambold)$
in the weak star topology.
\end{theorem}
\begin{proof}
By definition of the regularization operator we find for 
 $u_n \in R(f_n,\regparambold)$ that for any $u \in \domain$
$$ \fidelfct(Ku_n,f_n) + J(u_n,\regparambold) \leq  \fidelfct(Ku ,f_n) + J(u,\regparambold) .$$
Due to the convergence of $f_n$ and the continuity of $\fidelfct$ in the second argument the right-hand side in the last estimate is uniformly bounded by some constant $a$, which again provides uniform bounds for both terms on the left-hand side. Consequently
$$ \Vert u_n \Vert \leq c(a,a,\Vert f_n\Vert).$$ 
The boundedness of  $\Vert f_n\Vert$ and monotone dependence of $c$ yields a uniform bound on $\Vert u_n \Vert$, thus a weakly converging subsequence.  Using lower semicontinuity arguments as in the results above and the continuity of $\fidelfct$ with respect to the second variable we see that for the limit $u^*$ the inequality
\begin{align*}
 \fidelfct(Ku^* ,f) + J(u^*,\regparambold) & \leq \lim \inf \fidelfct(Ku_{n_k},f_{n_k}) + J(u_{n_k},\regparambold) \leq \lim \fidelfct(Ku ,f_{n_k}) + J(u,\regparambold) \\ & = \fidelfct(Ku ,f ) + J(u,\regparambold)  .
 \end{align*}
 Hence $u^* \in R(f,\regparambold).$
\end{proof}

As mentioned earlier, the type of convergence in Theorem \ref{theorem:stability} corresponds exactly to the type of stability in the Kuratowski limit superior. 
We finally provide a comment on the convergence of the regularization method only, the proof is very analogous to the stability result, an a-priori bound is obtained by the estimate 
$$ 
\fidelfct(Ku^\regparambold,f^\delta) + J(u^\regparambold,\regparambold) \leq
\fidelfct(Ku^\dagger,f^\delta) + J(u^\dagger,\regparambold) \leq \delta +  J(u^\dagger,\regparambold) $$ for $u^{\regparambold} \in R(f^\delta,\regparambold)$ and any element  $u^\dagger \in {\cal S}(f, \regparambold)$. Depending on the specific dependence on $\regparambold$ some condition on the interplay of the noise level and the limit of $\regparambold$ is needed, to pass to the limit in 
$$ J(u^\regparambold,\regparambold)  \leq \delta +  J(u^\dagger,\regparambold) .$$
An abstract condition as $\regparambold$ converges to $\regparambold^*$ outside $A$ is 
$$ \lim_{\regparambold \rightarrow \regparambold^\ast} \frac{\delta}{J(u^\dagger,\regparambold)} = 0 , $$
then 
$$ \limsup_{\regparambold \rightarrow \regparambold^\ast} \frac{J(u^\regparambold,\regparambold)}{J(u^\dagger,\regparambold)} \leq 1. $$
In the standard case $J(u,\regparambold) = \alpha J(u)$ the condition is simply $\frac{\delta}\alpha \rightarrow 0$. Hence, for such parameter choices, variational regularization methods define indeed convergent regularization operators.

\subsection{Error Estimates}
When it comes to the solution of ill-posed, inverse problems, an important question to address is the question of how errors in the measurement data are being propagated in the regularization process; in particular, convergence with respect to the noise level $\delta$ and the rate of convergence are of major interest. Following up on Definition \ref{def:dconvergence}, we look into $D$-convergence in the case of $D$ being a Bregman distance.

In order to derive error estimates, we restrict ourselves to the following smoothness-class $\mathcal{M}_\nu$. Given some unknown ground truth solution $\minsol \in \select(\clean, \regparambold)$, we ensure $\minsol \in \oprange(\regoparg{\cdot})$, i.e. we have to ensure that there exists data $\rangedata$ such that $\minsol \in \regoparg{\rangedata}$ is a solution of the corresponding variational regularization problem.
\begin{definition}[(Variational) Range condition]
An element $\minsol \in \select(\clean, \regparambold)$ for $\clean \in \oprange_\fidelfct(K)$ satisfies the \emph{range condition} if $\minsol \in \oprange(\regoparg{\cdot})$. If $K \in \linbound{\domain}{\range}$, $\fidelfct$ is convex and Fr\'{e}chet-differentiable w.r.t. its first argument, and $\regfctarg{\cdot}$ is proper, convex and l.s.c., then this is equivalent to
\begin{align}
\exists \, p^\dagger \in \partial \regfctarg{\minsol}\, , \exists \, \rangedata \in \range: \qquad p^\dagger = - K^\ast \partial_x \fidelity{K\minsol}{\rangedata} \, .\tag{RC}\label{eq:rc}
\end{align}
\end{definition}
From now on we assume $K \in \linbound{\domain}{\range}$, convexity and Fr\'{e}chet-differentiability of $\fidelfct$ in its first argument, and properness, convexity and lower semi-continuity of $\regoparg{\cdot}$ for the remainder of this section, which will allow us to use an appropriate optimality condition. 

Let us sketch the basic idea in the case of a quadratic fidelity $\fidelfct(f,g) = \frac{1}2 \Vert f-g\Vert^2$ with some norm in a Hilbert space and $J(u,\regparambold) = \alpha J(u)$. The optimality condition \eqref{eq:basicvariational} is given by
$$ K^* (Ku^\alpha - f^\delta) + \alpha p^\alpha = 0, \qquad p^\alpha \in \partial J(u^\alpha). $$
In order to satisfy the range condition for $u^\dagger$ we need to assume the existence $f_\alpha^\dagger$ such that $p^\dagger \in \partial J(u^\dagger)$. In order to satisfy the range condition for $u^\dagger$ we need to assume the existence $f_\alpha^\dagger$ such that $p^\dagger \in \partial J(u^\dagger)$ and
$$ K^* (Ku^\dagger - f_\alpha^\dagger) + \alpha p^\dagger = 0.$$
We see that this equation implies the condition $p^\dagger = K^*v$ for some $v$ (noticing $K^* (Ku^\dagger -f^\dagger) = 0$). On the other hand if this condition is satisfied we can construct $f_\alpha^\dagger = \clean - \alpha v$, i.e., $p^\dagger = K^*v$ is equivalent to the range condition \eqref{eq:rc}. An error estimate can then be obtained by subtracting both optimality conditions
$$K^* K (u^\alpha - u^\dagger) + \alpha (p^\alpha -p^\dagger)= K^* (f^\delta - f_\alpha^\dagger).
$$ 
Taking a duality product with $u^\alpha - u^\dagger$ yields 
$$ \Vert K(u^\alpha - u^\dagger) \Vert^2 + \alpha D_J^{p^\alpha}(u^\dagger,u^\alpha) + \alpha
D_J^{p^\dagger}(u^\alpha,u^\dagger) = 
\langle K(u^\alpha - u^\dagger),f^\delta - f_\alpha^\dagger \rangle.  $$
Applying Young's inequality on the right-hand side and inserting the special form of $f_\alpha^\dagger$ then immediately yields an error estimate (cf. \cite{burger2016bregman}). Note that we obtain an upper bound on the residual as well as the symmetric Bregman distance
\begin{equation}
\symbreg{\minsol}{\regsol} = \bregdis{p^\alpha}{u^\dagger}{u^\alpha} +  
\bregdis{p^\dagger}{u^\alpha}{u^\dagger}.
\end{equation}
For further interpretations of the error estimates see \cite{burgerosher,burger2007error,resmerita06,burger2016bregman}).

We now want to show that \eqref{eq:rc} coincides with the well-known source condition (cf. \cite{chavent1997regularization,burgerosher}) for a certain class of fidelity functionals. Before we proceed, we have to define this source condition first.
\begin{definition}[Source condition]
An element $\minsol \in \select(\clean, \regparambold)$ for $\clean \in \oprange_K(\fidelfct)$ satisfies the \emph{source condition} if
\begin{align*}
\mathcal{R}(K^\ast) \cap \partial \regfctarg{\minsol} \neq \emptyset \, .
\end{align*}
This is equivalent to
\begin{align}
\exists \, p^\dagger \in \partial \regfctarg{\minsol}, \, \exists \, v \in \range^\ast \setminus \{ 0 \}: \qquad p^\dagger = K^\ast v \, . \tag{SC}\label{eq:sc}
\end{align}
\end{definition}
\begin{remark}\label{rem:sc}
For scalar regularization parameters $\regparambold = \regparam$ and regularization functionals of the form $\regfctarg[\regparam]{u} = \regparam \regfct_1(u)$ the source condition for $\regparam = 1$ can be written as $K^\ast v \in \partial \regfct_1(\minsol) = \partial \regfctarg[1]{\minsol}$. Every other potential source condition $K^\ast v_\regparam \in \partial \regfctarg[\regparam]{\minsol}$ can be expressed in terms of $v$ via the relation $v_\regparam = \regparam v$.
\end{remark}
It is obvious that \eqref{eq:rc} implies \eqref{eq:sc}. However, we want to go one step further and show that \eqref{eq:rc} and \eqref{eq:sc} are even equivalent conditions for fidelity functionals $\fidelity{Ku}{\noisy} := G(Ku - \noisy)$, where $G$ is a Legendre functional. Legendre functionals are defined as follows.
\begin{definition}[{\cite[Definition 5.2]{bauschke2001essential}}]\label{def:legendre}
Let $G:\range \rightarrow \Rinf$ be a proper, convex and l.s.c. functional. We say that $G$ is 
\begin{itemize}
\item \emph{essentially smooth}, if $\partial G$ is both locally bounded and single-valued on its domain.
\item \emph{essentially strictly convex}, if $(\partial G)^{-1}$ is locally bounded on its domain and $G$ is strictly convex on every convex subset of $\dom(\partial G)$.
\item \emph{Legendre}, if $G$ is both essentially smooth and essentially strictly convex.
\end{itemize}
\end{definition}
\noindent Now we show that \eqref{eq:rc} and \eqref{eq:sc} are equivalent for $G$ being a Legendre functional.
\begin{theorem}\label{thm:rcscequivalence}
Let $\range$ be reflexive, and suppose $\fidelity{\clean}{\noisy} := G(\clean - \noisy)$ for any $\clean, \noisy \in \range$, where $G:\range \rightarrow \Rinf$ is a Legendre functional. Then \eqref{eq:rc} and \eqref{eq:sc} are equivalent conditions.
\begin{proof}
"$\Rightarrow$": Condition \eqref{eq:rc} trivially implies \eqref{eq:sc} if we define $v := - \partial_x \fidelity{K\minsol}{\rangedata} = - G^\prime(K\minsol - \rangedata)$.\\
"$\Leftarrow$": The source condition \eqref{eq:sc} can be written as
\begin{align*}
0 &= p^\dagger - K^\ast v \, ,\\
\Leftrightarrow 0 &= p^\dagger + K^\ast G^\prime( (G^\ast)^\prime( - v ) ) \, ,
\end{align*}
where $G^\ast :\range^\ast \rightarrow \Rinf$ denotes the convex conjugate of $G$. Note that $G^\ast$ is also a Legendre functional since $\range$ is reflexive (see \cite[Corollary 5.5]{bauschke2001essential}), and that the last equality is valid for all $v \in \dom(G)$ due to \cite[Theorem 5.9]{bauschke2001essential}. Hence, if we define 
\begin{align*}
\rangedata := K\minsol - (G^\ast)^\prime(-v)
\end{align*}
we ensure that the range condition \eqref{eq:rc} is satisfied.
\end{proof}
\end{theorem}


The range condition \eqref{eq:rc} allows us to derive error estimates in a Bregman distance setting for these very generic variational regularization methods. The following lemma builds the basis by estimating Bregman distances between $\regsol$ and $\minsol$ in terms of differences of the data fidelities. 

\begin{lemma}\label{lem:rcequality}
Let \eqref{eq:rc} be satisfied. Then we observe
\begin{align}
\begin{split}
&\bregdis[\fidelity{K \cdot}{\noisy}]{}{\minsol}{\regsol} + \bregdis[\fidelity{K \cdot}{\scdata}]{}{\regsol}{\minsol} + \symbreg{\minsol}{\regsol} \\
{} = {} &\fidelity{K\minsol}{\noisy} - \fidelity{K\minsol}{\scdata} + \fidelity{K\regsol}{\scdata} - \fidelity{K\regsol}{\noisy}
\end{split}\label{eq:errorreform}
\end{align}
for every $\regsol \in \regoparg{\noisy}$.
\begin{proof}
Computing the optimality condition \eqref{eq:varregopt} of \eqref{eq:varreg} and subtracting $p^\dagger \in \partial \regfctarg{\minsol}$ from both sides of the equality yields
\begin{align*}
p_\regparambold - p^\dagger = - K^\ast \partial_x \fidelity{K\regsol}{\noisy} - p^\dagger \, , 
\end{align*}
for any $p_\regparambold \in \partial \regfctarg{\regsol}$. Taking a duality product with $\regsol - \minsol$ then yields
\begin{align*}
\symbreg{\regsol}{\minsol} = \underbrace{\langle K^\ast \partial_x \fidelity{K\regsol}{\noisy}, \minsol - \regsol \rangle}_{= \fidelity{K\minsol}{\noisy} - \fidelity{K\regsol}{\noisy} - \bregdis[{\fidelity{K\cdot}{\noisy}}]{}{\minsol}{\regsol}} - \langle p^\dagger, \regsol - \minsol \rangle \, .
\end{align*}
Hence, we conclude
\begin{align}
\bregdis[{\fidelity{K\cdot}{\noisy}}]{}{\minsol}{\regsol} + \symbreg{\regsol}{\minsol} = \fidelity{K\minsol}{\noisy} - \fidelity{K\regsol}{\noisy} - \langle p^\dagger, \regsol - \minsol \rangle \, .\label{eq:errestcmp1}
\end{align}
If we now choose $p^\dagger = -K^\ast \partial_x \fidelity{K\minsol}{\rangedata}$ -- which is possible since \eqref{eq:rc} holds true -- we obtain the equality
\begin{align}
\begin{split}
- \langle p^\dagger, \regsol - \minsol \rangle &= \langle K^\ast \partial_x \fidelity{K\minsol}{\rangedata}, \regsol - \minsol \rangle \\
&= \fidelity{K\regsol}{\rangedata} - \fidelity{K\minsol}{\rangedata} - \bregdis[{\fidelity{K\cdot}{\rangedata}}]{}{\regsol}{\minsol}
\end{split}\, .\label{eq:errestcmp2}
\end{align}
Inserting \eqref{eq:errestcmp2} into \eqref{eq:errestcmp1} then yields \eqref{eq:errorreform}.
\end{proof}
\end{lemma}

Before we proceed, we make the following observation for data fidelities $\fidelfct$ that are also Bregman distances.
\begin{corollary}\label{cor:bregmanofbregman}
Let $\fidelfct:\range \times \range \rightarrow \R$ be a Bregman distance, i.e.
\begin{align*}
\fidelity{\clean}{\rangedata} = G(\clean) - G(\rangedata) - \langle G^\prime(\rangedata), \clean - \rangedata \rangle \geq 0 \, , 
\end{align*}
for all $\clean, \rangedata \in \range$, and some functional $G:\range \rightarrow \R$. Then we already observe
\begin{align*}
\bregdis[\fidelity{\cdot}{\rangedata}]{}{\clean}{\noisy} = \fidelity{\clean}{\noisy} \, ,
\end{align*}
for all $\clean, \rangedata, \noisy \in \range$.
\end{corollary}
\begin{proof}
We simply compute
\begin{align*}
\bregdis[\fidelity{\cdot}{\rangedata}]{}{\clean}{\noisy} {} = {} &\fidelity{\clean}{\rangedata} - \fidelity{\noisy}{\rangedata} - \langle \partial_x \fidelity{\noisy}{\rangedata}, \clean - \noisy \rangle\\
{} = {} &G(\clean) - G(\rangedata) - \langle G^\prime(\rangedata), \clean - \rangedata \rangle \\
&-G(\noisy) + G(\rangedata) + \langle G^\prime(\rangedata), \noisy - \rangedata \rangle\\
&- \langle G^\prime(\noisy) - G^\prime(\rangedata), \clean - \noisy \rangle \, ,\\
{} = {} &G(\clean) - G(\noisy) - \langle G^\prime(\rangedata), \clean - \noisy \rangle \\
&- \langle G^\prime(\noisy) - G^\prime(\rangedata), \clean - \noisy \rangle \, , \\
{} = {} &\bregdis[G]{}{\clean}{\noisy} = \fidelity{\clean}{\noisy} \, ,
\end{align*}
and hence, prove the result.
\end{proof}
As a consequence, Lemma \ref{lem:rcequality} reads as follows for data fidelities that are also Bregman distances.
\begin{lemma}
Let the assumptions of Lemma \ref{lem:rcequality} and Corollary \ref{cor:bregmanofbregman} hold true. Then we have
\begin{align*}
\symbreg{\minsol}{\regsol} = \langle G^\prime(\rangedata) - G^\prime(K\minsol) - ( G^\prime(\noisy) - G^\prime(K\regsol) ), K\minsol - K\regsol \rangle \, .
\end{align*}
\begin{proof}
From Corollary \ref{cor:bregmanofbregman} we know that 
\begin{align*}
\bregdis[\fidelity{K \cdot}{\noisy}]{}{\minsol}{\regsol} = \fidelity{K\minsol}{K\regsol} \qquad \text{and} \qquad
\bregdis[\fidelity{K \cdot}{\rangedata}]{}{\regsol}{\minsol} = \fidelity{K\regsol}{K\minsol} \, .
\end{align*}
Hence, we observe
\begin{align*}
\begin{split}
\bregdis[\fidelity{K \cdot}{\noisy}]{}{\minsol}{\regsol} + \bregdis[\fidelity{K \cdot}{\rangedata}]{}{\regsol}{\minsol} &= \fidelity{K\minsol}{K\regsol} + \fidelity{K\regsol}{K\minsol}\\
&= \symbreg[G]{K\minsol}{K\regsol}\\
&= \langle G^\prime(K\minsol) - G^\prime(K\regsol), K\minsol - K\regsol \rangle
\end{split} \, .
\end{align*}
We also discover
\begin{align*}
&\fidelity{K\minsol}{\noisy} - \fidelity{K\minsol}{\scdata} + \fidelity{K\regsol}{\scdata} - \fidelity{K\regsol}{\noisy} \\
\begin{split}
{} = {} &G(K\minsol) - G(\noisy) - \langle G^\prime(\noisy), K\minsol - \noisy \rangle\\
 &- \left( G(K\minsol) - G(\rangedata) - \langle G^\prime(\rangedata), K\minsol - \rangedata \rangle \right)\\
 &+ G(K\regsol) - G(\rangedata) - \langle G^\prime(\rangedata), K\regsol - \rangedata \rangle\\
 &- \left( G(K\regsol) - G(\noisy) - \langle G^\prime(\noisy), K\regsol - \noisy \rangle \right)
\end{split}\, ,\\
 {} = {} &\langle G^\prime(\noisy), K\regsol - K\minsol \rangle + \langle G^\prime(\rangedata), K\minsol - K\regsol \rangle \, ,\\
 {} = {} &\langle G^\prime(\noisy) - G^\prime(\rangedata), K\regsol - K\minsol \rangle \, .
\end{align*}
Combining these two equalities with \eqref{eq:errorreform} yields the desired result.
\end{proof}
\end{lemma}

\begin{example}
We can use \eqref{eq:errorreform} to derive the same error estimates presented in \cite{burger2007error} for the choice $\fidelity{Ku}{\noisy} = \frac{1}{2}\| Ku - \noisy \|_{\hilbert}^2$, where $\hilbert$ is a Hilbert space. In this case we observe 
\begin{align*}
\bregdis[{\fidelity{K\cdot}{\noisy}}]{}{\minsol}{\regsol} = \bregdis[\fidelity{K \cdot}{\scdata}]{}{\regsol}{\minsol} = \frac{1}{2} \| K(\minsol - \regsol) \|_{\hilbert}^2 \, .
\end{align*}
Hence, Equation \eqref{eq:errorreform} reads as
\begin{align*}
\begin{split}
&\| K(\minsol - \regsol) \|_{\hilbert}^2 + \symbreg{\minsol}{\regsol} + \frac{1}{2}\| K\regsol - \noisy \|_{\hilbert}^2\\
{} = {} &\frac{1}{2} \| K\minsol - \noisy \|_{\hilbert}^2 + \frac{1}{2} \| K\regsol - \rangedata \|_{\hilbert}^2 - \frac{1}{2} \| K\minsol - \rangedata \|_{\hilbert}^2
\end{split} \, .
\end{align*}
If we make use of the estimate $\frac{1}{2} \| K\minsol - \noisy \|_{\hilbert}^2 \leq \frac{1}{2} \| \clean - \noisy \|_{\hilbert}^2 \leq \delta$, the previous equality transforms into the inequality
\begin{align*}
\begin{split}
&\| K(\minsol - \regsol) \|_{\hilbert}^2 + \symbreg{\minsol}{\regsol} + \frac{1}{2}\| K\regsol - \noisy \|_{\hilbert}^2\\
{} \leq {} &\delta + \frac{1}{2} \| K(\regsol - \minsol) + (K\minsol - \rangedata) \|_{\hilbert}^2 - \frac{1}{2} \| K\minsol - \rangedata \|_{\hilbert}^2\\
{} \leq {} &\delta + \| K(\minsol - \regsol) \|_{\hilbert}^2 + \| K\minsol - \rangedata \|_{\hilbert}^2 - \frac{1}{2} \| K\minsol - \rangedata \|_{\hilbert}^2
\end{split}\, .
\end{align*}
Subtracting $\| K(\minsol - \regsol) \|_{\hilbert}^2$ on both sides of the inequality then yields the error estimate
\begin{align}
\symbreg{\minsol}{\regsol} + \frac{1}{2}\| K\regsol - \noisy \|_{\hilbert}^2 \leq \delta + \frac{1}{2} \| K\minsol - \rangedata \|_{\hilbert}^2 \, .\label{eq:sqhilbertfidelerrest}
\end{align}
We want to emphasize that the constant $\frac{1}{2} \| K\minsol - \rangedata \|_{\hilbert}^2$ on the right-hand-side of the inequality does depend on the choice of $\alpha$. From Remark \ref{rem:sc} and the proof of Theorem \ref{thm:rcscequivalence} it follows that if we consider regularizations of the form $\regfctarg{u} = \alpha \regfct(u)$, the source condition \eqref{eq:sc} and the range condition \eqref{eq:rc} are linked via the relation $\rangedata = K\minsol + \regparam v$, where $v$ is the source condition element for $\regparam = 1$, i.e. $K^\ast v \in \partial \regfctarg[1]{\minsol} = \partial \regfct(\minsol)$. In this setting, the error estimate \eqref{eq:sqhilbertfidelerrest} then reads as
\begin{align*}
\symbreg[\regfct]{\minsol}{\regsol} + \frac{1}{2\regparam}\| K\regsol - \noisy \|_{\hilbert}^2 \leq \frac{\delta}{\regparam} + \frac{\regparam}{2} \| v \|_{\hilbert}^2
\end{align*}
Hence, choosing $\regparam(\delta) = \sqrt{2 \delta} / \| v \|_{\hilbert}$ then yields $\symbreg[\regfct]{\minsol}{\regsol} = \mathcal{O}(\sqrt{\delta})$.
\end{example}


There are various routes and generalizations that can be taken from these types of estimates, e.g. to weaker source conditions with the concepts of approximate or variational source conditions (cf. \cite{Schusterbuch,flemming2010new,flemming2013variational,flemmingexistence,Hohage2016}), improved estimates for stronger conditions (cf. \cite{Resmerita2005,grasmair2013variational}), or large noise that is not necessarily in $\range$ (cf. \cite{burger2016large}). In special cases such as $\ell^1$-regularization improved results can be obtained, due to the effective finite-dimensionality this case is on the borderline to being well-posed (cf. 
\cite{grasmair2011necessary,grasmair2011linear,burger2013convergence,flemming2015l1,flemming2016unified,flemming2017injectivity}).
Recently also converse results could be obtained (cf. \cite{2017arXiv171201499F,Hohage2016}).

\subsection{Variational Eigenvalue Problems}\label{sec:vareigprob}

The standard tool for the analysis of linear regularization methods is singular value decomposition.  In the case of nonlinear regularization no analogue of singular values and singular vectors was known for a long time. A generalization for nonlinear variational methods was made in \cite{benning2013ground}, which we discuss in the following. We generalize singular vectors as eigenvectors of the variational regularization operator $\regop$ as defined in Definition \ref{def:varreg}, i.e. we look for functions $\ul$ that satisfy 
\begin{align}
\lambda \ul \in \regoparg{\sigma K\ul} \, , \label{eq:eigenvalue}
\end{align}
for constants $\lambda, \sigma \in [0, \infty)$, typically $\sigma =1$.   For simplicity we focus on the case where $\regparambold = \regparam$ is a scalar, and $\fidelity{Ku}{\noisy} = G(Ku - \noisy)$, where $G$ is a Legendre functional for the remainder of this section. If we consider the optimality condition \eqref{eq:varregopt} of \eqref{eq:varreg} we immediately observe that any $\ul$ satisfying \eqref{eq:eigenvalue} also has to satisfy
\begin{align}
-K^\ast G^\prime\left( (\lambda - \sigma) K\ul \right) \in \partial \regfctarg[\regparam]{\lambda \ul} \, . \label{eq:eigprev}
\end{align}
We now assume that both $G^\prime$ and $\partial \regfct$ are homogeneous in the sense that they satisfy $G^\prime(c u) = s_1(c) G^\prime(u)$ and $\partial \regfctarg{c u} = s_2(c, \regparam) \partial \regfct(u)$ for constants $c \in \R$ and functions $s_1, s_2:\R \rightarrow \R$. Then \eqref{eq:eigprev} simplifies to
\begin{align}
-\frac{s_1(\lambda - \sigma)}{s_2(\lambda/\sigma, \regparam)} K^\ast G^\prime\left( K\ul \right) \in \partial \regfct(\sigma \ul) \, . \label{eq:eig}
\end{align}
Equation \eqref{eq:eig} paves the way for the following definition of generalized singular vectors.
\begin{definition}[Generalized singular system]
Let $\{ u_\sigma, v_\sigma, \sigma \}$ satisfy
\begin{align}
Ku_\sigma = \sigma v_\sigma \quad \text{and} \quad K^\ast G^\prime\left(v_\sigma \right) \in \partial \regfct(\sigma u_\sigma)\label{eq:singvec}
\end{align}
for $\sigma > 0$. Then $\{ u_\sigma, v_\sigma, \sigma \}$ is called a \emph{generalized singular system}.
\end{definition}
\begin{remark}
In case of $G(v) = \frac{1}{2}\| v \|_{L^2(\Sigma)}^2$ and $\regfctarg[\regparam]{u} = \frac{\regparam}{2} \| u \|_{L^2(\Omega)}^2$ this definition is consistent with the classical singular vector theory for compact operators.
\end{remark}
\begin{example}
Suppose $G(v) = \frac{1}{2}\| v \|_{L^2([0, 1])^2}$, $K:\bv([0, 1]) \rightarrow L^2([0, 1])$ is the embedding operator and $\regfctarg[\regparam]{u} = \regparam \tvast(u)$, where $\tvast$ denotes the (one-dimensional) total variation with Dirichlet-zero boundary conditions. It has been shown in \cite{benning2013ground} that Haar Wavelets are generalized singular vectors of $\tvast$. Precisely, we have $v_{\sigma_n, k} = \sigma_n u_{\sigma_n, k} \in \partial \tvast(\sigma_n u_{\sigma_n, k}) = \partial \tvast(u_{\sigma_n, k})$ for $\sigma_n := 2^{-\frac{n + 4}{2}}$ and $u_{\sigma_n, k}$ defined as
\begin{align*}
u_{\sigma_n, k}(x) := 2^{\frac{n}{2}} \Psi(2^n x - j) \quad \text{with} \quad \Psi(x) := \begin{cases}
1 & x \in \left[0, \frac{1}{2} \right) \\
-1 & x \in \left[\frac{1}{2}, 1\right) \\
0 & \text{else}
\end{cases} \, .
\end{align*}
The singular value $\sigma_n$ is determined via \eqref{eq:singvec}. The dual singular vector $v_{\sigma_n, k}$ has to satisfy $v_{\sigma_n, k} \in \partial \tvast(\sigma_n u_{\sigma_n, k})= \tvast(u_{\sigma_n, k})$. If we make use of $v_{\sigma_n, k} = u_{\sigma_n, k} / \sigma_n$ and take a dual product with $u_{\sigma_n, k}$, we immediately observe $\sigma_n = \| u_{\sigma_n, k} \|_{L^2([0, 1])}^2 / \tvast(u_{\sigma_n, k})$. In Figure \ref{fig:singvecexample} we see the Haar wavelet $u_{\sigma_1, 1/2}$ and its scaled version $v_{\sigma_1, 1/2} = \sigma_1 u_{\sigma_1, 1/2} = 2^{-\frac{5}{2}} u_{\sigma_1, 1/2}$.
\begin{figure}[!t]
\begin{center}
\subfloat[$u_{\sigma_1, 1/2}$]{
\begin{tikzpicture}[scale=0.48,
declare function={ 
    haar(\x) = and(\x >= 0, \x < 1/2) * (+1) + and(\x >= 1/2, \x < 1) * (-1);
}
]
\begin{axis}[%
            domain = 0:1,
            samples = 250,
			ymin = -(2*sqrt(2)+0.5),
			ymax = (2*sqrt(2)+0.5),
            xlabel = {$x$},
            ylabel = {$u_\sigma(x)$},
            ]
            \addplot[line width=0.5mm,  blue] {(2^(1/2))*haar((2^1)*(x - 0.25))};                        
        \end{axis}
\end{tikzpicture}
}\hspace{0.3cm}
\subfloat[$v_{\sigma_1, 1/2}$]{
\begin{tikzpicture}[scale=0.48,
declare function={ 
    haar(\x) = and(\x >= 0, \x < 1/2) * (+1) + and(\x >= 1/2, \x < 1) * (-1);
}
]
\begin{axis}[%
            domain = 0:1,
            samples = 250,
			ymin = -(2*sqrt(2)+0.5),
			ymax = (2*sqrt(2)+0.5),
            xlabel = {$x$},
            ylabel = {$v_\sigma(x)$},
            ]
            \addplot[line width=0.5mm, blue] {sqrt(2)*(2^(1/2))*haar((2^1)*(x - 0.25))};                        
        \end{axis}
\end{tikzpicture}
}
\end{center}
\caption{The Haar wavelet $u_{\sigma_1, 1/2}$ and its scaled version $v_{\sigma_1, 1/2}$. In \cite{benning2013ground} it has been shown that together with $\sigma_1 = 2^{-\frac{5}{2}}$ they form a generalized singular system in the sense of \eqref{eq:singvec} with $K$ being the identity in $L^2$.}\label{fig:singvecexample}
\end{figure}
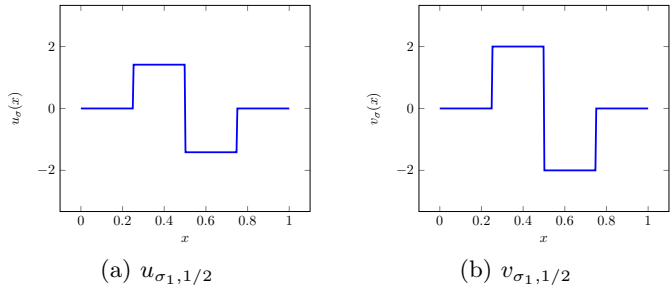
\end{example}
\noindent The generalized singular system is defined so that \eqref{eq:eig} and \eqref{eq:singvec} coincide for 
\begin{align}
s_1(\lambda - \sigma) = -s_2(\lambda/\sigma, \regparam) \, .\label{eq:singveceigcond}
\end{align}
Hence, if we choose $\regparam$ and $\lambda$ such that \eqref{eq:singveceigcond} holds true, we already know that \eqref{eq:eigenvalue} is satisfied for these particular choices of $\lambda$ and $\regparam$.  
\begin{example}
For $G(v) = \frac{1}{2}\| v \|_{L^2(\Sigma)}^2$ and $\regfctarg[\regparam]{u} = \frac{\regparam}{2} \| u \|_{L^2(\Omega)}^2$, with $\Sigma$ and $\Omega$ being domains in $\R^{d_1}$ respectively $\R^{d_2}$ we observe $s_1(x) = x$ and $s_2(x, \regparam) = \regparam x$. Hence, \eqref{eq:singveceigcond} simplifies to $\sigma - \lambda = (\regparam \lambda) / \sigma$. Solving for $\lambda$ then yields
\begin{align*}
\lambda = \frac{\sigma}{\sigma^2 + \regparam} \, ,
\end{align*}
which perfectly coincides with the singular value decomposition representation of Tikhonov regularization.
\end{example}
\begin{example}
For $G(v) = \frac{1}{2}\| v \|_{L^2(\Sigma)}^2$ and $\regfctarg{u} = \regparam \tv(u)$ we have $s_1(x) = x$ and $s_2(x, \regparam) = \regparam$. Consequently, \eqref{eq:singveceigcond} solved for $\lambda$ reads as
\begin{align*}
\lambda = \frac{1 - \regparam}{\sigma} \, .
\end{align*}
This eigenvalue of this particular regularization operator is consistent with classical singular value theory in the sense that it satisfies $\lim_{\regparam \downarrow 0} \lambda = 1/\sigma$.
\end{example}

An interesting observation from the previous examples is that $\regparam > 0$ automatically implies $\lambda < 1/\sigma$ (unless $u_\sigma \in \ker(\regfct)$). This implies that there always is a systematic error when it comes to recovering singular vectors with variational regularization methods that have quadratic fidelity. This is also true for input data that is not given in terms of a singular vector, see \cite[Theorem 7]{benning2013ground}. In the next section we see that iterative regularization methods can overcome this systematic reconstruction bias.

\section{Iterative Regularization Methods}\label{sec:iterreg}

Iterative regularization is based on a different paradigm then variational regularization and based on the simple observation that most iterative procedures can be applied in a robust fashion to ill-posed problems. The standard example in a Hilbert space is the Landweber iteration (cf. \cite{landweber1951iteration})
$$ u^k = u^{k-1} - \tau K^* (Ku^{k-1} -f^\delta),$$ 
which only applies the continuous operators $K$ and $K^*$. Let us mention again at this point that with standard initial values such as $u^0=0$ the iterates satisfy a range condition $u^k \in \oprange(K^*)$. At an abstract level we construct an iteration procedure
\begin{equation}  \label{eq:iterationprocedure}
u^k = R_I(f^\delta,v^{k-1},\regparambold),
\end{equation}
with some iteration operator $R_I$ and a collection of variables $v^{k-1}$ summarizing the information used about the first $k-1$ steps. In this case the parameter set ${\regparambold}$ will contain the iteration index as well as auxiliary parameters such as the step size $\tau$. In the simplest case of a one-step method like the Landweber iteration we simply have $v^{k-1}=u^{k-1}$, for multistep methods the variable $v^{k-1}$ could be a collection of several previous iterations. As we shall see in the methods below $v^{k-1}$ could also collect some auxiliary variables.  

For such methods one observes a so-called {\em semi-convergence} phenomenon. In the case of exact data $f \in \oprange(K)$ the method is converging, while in the case of noisy data it seems to approximate the exact solution for an initial phase of the iteration and then starts to diverge. This behaviour naturally leads to the idea of achieving a regularizing effect by stopping the iterations early. A standard approach is the so-called discrepancy principle, which monitors the residual during the iteration and compares it with the noise level. Since the exact solution could lead to a residual at this level there is no particular reason to iterate further once the residual is at the size of the noise level:
 
\begin{definition}[Morozov's discrepancy principle]\label{def:morozov}
Let $\clean$ and $\noisy$ satisfy $\fidelity{\clean}{\noisy} \leq \delta$. If we choose $\eta \geq 1$ and $k^\ast := k^\ast(\delta, f^\delta)$ such that
\begin{align*}
\fidelity{Ku^{k^\ast}}{\noisy} \leq \eta \delta < \fidelity{Ku^k}{\noisy}
\end{align*}
is satisfied for $u^{k^\ast} \in R_I(\noisy, v^{k^\ast - 1},\regparambold)$ and $u^k \in R_I(\noisy, v^{k - 1},,\regparambold)$ for all $k < k^\ast$, then $u^k$ is said to satisfy \emph{Morozov's discrepancy principle}.
\end{definition}

Given a stopping rule to determine $k^*(\delta,f^\delta)$ such as the discrepancy principle we can define the full regularizaton operator:
\begin{equation}
 R(f^\delta,\regparambold) = u^{k_*(\delta,f^\delta)},
\end{equation}
where for $k=1,\ldots,k_*(\delta,f^\delta)$ the iterates $u^k$ are determined by \eqref{eq:iterationprocedure} with some fixed initial value $v^0$ including $u^0$. 

The semiconvergence behaviour of such a method is then the standard convergence of a nonlinear regularization method, in particular for consistency we need $u^k \wlim u^\dagger$ as $k \rightarrow \infty$ in the case of clean data $\clean \in \oprange_\fidelfct(K)$ and $\minsol \in \select(\clean, \regparambold)$ . A standard tool used to prove the convergence of an iterative regularization method is to find some error measure to the true solution that is decreasing until the stopping index is reached. For the methods below constructed from a regularization functional $J$ we will see that this is the case for the Bregman distance, i.e.,
\begin{itemize}
\item $ \bregdis{p^{k + 1}}{u^\dagger}{u^{k + 1}} \leq \bregdis{p^k}{u^\dagger}{u^k}$, for $u^k \in \regop_I(\noisy, v^{k - 1}, \regparambold)$ and all $k \leq k^\ast-1$,
\item $ \lim_{\delta \rightarrow 0} \bregdis{p^k}{u^k_\delta}{u^k} = 0$, for $u^k_\delta \in \regop_I(\noisy, v^{k - 1}_\delta, \regparambold)$ and $u^k \in \regop_I(\clean, v^{k - 1}, \regparambold)$.
\end{itemize}
With some further effort one can then conclude the convergence of the regularization method in this sense: 
\begin{align*}
\lim_{\delta \rightarrow 0} \sup \left\{ \left. \bregdis{p^{k^\ast(\delta, \noisy)}}{\minsol}{\regoparg{\noisy}} \, \right| \, \noisy \in \range, \fidelity{\clean}{\noisy} \leq \delta \right\} = 0 \, ,
\end{align*}
for $\regoparg{\noisy} = u^{k^\ast(\delta, \noisy)}$ (cf. \cite{osher2005iterative,Schusterbuch}).

As in the case of Banach spaces such as BV there is no immediate analogue of simple iterative procedures in Hilbert  spaces one often resorts to define an iteration operator $R_I$ by solving a variational problem. This approach will be detailed in the next sections.

\subsection{Bregman Iteration}
The concept of Bregman iteration -- also known as proximal minimization algorithm -- introduces an iteration into the variational regularization framework by replacing the regularization function $\regfctarg{u}$ with the corresponding generalized Bregman distance $\bregdis{p}{u}{v}$, for $v \in \domain$ and $p \in \partial \regfct(v)$. For the choice $\regfctarg{u} = \frac{\regparambold}{2} \| u \|_{\hilbert}^2$ it is also known as iterated Tikhonov regularization, which dates back to the works of Kryanev \cite{KRYANEV197424}, further analyzed e.g. in \cite{groetsch1977sequential,thomas1979approximation}  . The extension to more general choices of Bregman distances has first been proposed by Censor and Zenios in \cite{censor1992proximal}, shortly followed by Teboulle in \cite{teboulle1992entropic}, and has since been subject to extensive research \cite{eckstein1993nonlinear,kiwiel1997proximal}. Notably, it has been extended to generalized Bregman distances that allow for subdifferentiable instead of differentiable functionals in \cite{osher2005iterative}. Note that in such cases there is no one-to-one relation between $u^{k-1}$ and its subgradient $p^{k-1}$, hence we set $v^{k-1}=(u^{k-1},p^{k-1})$.  With a set-valued iteration operator, the Bregman iteration can be written as
\begin{align*}
u^k \in R_I(\noisy, v^{k - 1},\regparambold) &= \argmin_{u \in \domain} \left\{ \fidelity{Ku}{\noisy} + \bregdis{p^{k - 1}}{u}{u^{k - 1}} \right\} \, ,\\
p^k &= p^{k - 1} - K^\ast \partial_x \fidelity{K u^k}{\noisy} \, , \vphantom{\argmin_{u \in \domain} \left\{ \fidelity{Ku}{\noisy} + \bregdis{p^{k - 1}}{u}{u^{k - 1}} \right\}}
\end{align*}
for $p^0 \in \partial \regfctarg{u^0}$. The entire method is summarized in Algorithm \ref{alg:bregiter}.
\begin{algorithm}[!t]
\caption{Bregman iteration}
\label{alg:bregiter}
\begin{algorithmic}
\State{Initialize $\regparambold \in \regdomain$, $f^\delta \in \range$, $u^0 \in \domain$ and $p^0$ with $p^0 \in \partial \regfctarg{u^0}$}
\For{$k = 1, \ldots, k^\ast$}
\State{Compute $\regop_I(\noisy, v^{k - 1}, \regparambold) = \argmin_{u \in \domain} \left\{ \fidelity{Ku}{f^\delta} + \bregdis{p^{k - 1}}{u}{u^{k - 1}} \right\}$}
\State{Pick $u^k \in \regop_I(\noisy, v^{k - 1}, \regparambold) \vphantom{\argmin_{u \in \domain} \left\{ \fidelity{Ku}{\noisy} + \bregdis{p^{k - 1}}{u}{u^{k - 1}} \right\}}$}
\State{Update $p^k = p^{k - 1} - K^\ast \partial_x \fidelity{K u^k}{f^\delta} \vphantom{\argmin_{u \in \domain} \left\{ \fidelity{Ku}{\noisy} + \bregdis{p^{k - 1}}{u}{u^{k - 1}} \right\}}$}
\State{Set $v^k = (u^k, p^k) \vphantom{\argmin_{u \in \domain} \left\{ \fidelity{Ku}{\noisy} + \bregdis{p^{k - 1}}{u}{u^{k - 1}} \right\}}$}
\EndFor\\
\Return{$u^{k^\ast}$, $p^{k^\ast}$}
\end{algorithmic}
\end{algorithm}
\begin{remark}\label{rem:primdualsumformula}
The update for the subgradient can also be written as
\begin{align}
p^k = p^0 - \sum_{n = 1}^k K^\ast \partial_x \fidelity{Ku^n}{\noisy} .\label{eq:bregiterdualupdate}
\end{align}
Hence, we can rewrite the primal update to
\begin{align}
\regop_I(\noisy, \{ u^n \}_{n = 1}^{k - 1}, p^0, \regparambold) = \argmin_{u \in \domain} \left\{ \fidelity{Ku}{\noisy} + \regfctarg{u} - \left\langle p^0 - \sum_{n = 1}^{k - 1} K^\ast \partial_x \fidelity{Ku^n}{\noisy}, u \right\rangle \right\} \, .\label{eq:bregiterprimalupdate}
\end{align}
\end{remark}
In the following we want to recall (or derive) a few important properties of Algorithm \ref{alg:bregiter}. We start with a trivial monotonic decrease of the data fidelity.
\begin{corollary}[Monotonic decrease of the data fidelity]\label{cor:mondecbreg}
Suppose $u^0$ satisfies $\fidelity{Ku^0}{\noisy} < \infty$. Then the iterates of Algorithm \ref{alg:bregiter} satisfy
\begin{align*}
\fidelity{K u^{k + 1}}{\noisy} + \bregdis{p^k}{u^{k + 1}}{u^k} &\leq \fidelity{K u^k}{\noisy} \, ,
\intertext{and}
\lim_{k \rightarrow \infty} \bregdis{p^k}{u^{k + 1}}{u^k} &= 0 \, ,
\end{align*}
for $u^k \in \regopit(\noisy, v^{k - 1}, \regparambold)$ and all $k \in \N$.
\begin{proof}
The first statement follows trivially from the convexity of $\fidelfct$ (in its first argument) and $\regfct$, and the fact that $u^{k + 1}$ is a minimizer of $\energy(u) := \fidelity{K u}{\noisy} + \bregdis{p^k}{u}{u^k}$. The first statement then implies
\begin{align*}
\sum_{k = 0}^{N - 1} \bregdis{p^k}{u^{k + 1}}{u^k} &\leq \fidelity{K u^0}{\noisy} - \fidelity{K u^N}{\noisy}\\
&\leq \fidelity{K u^0}{\noisy} < \infty \, .
\end{align*}
Taking the limit $N \rightarrow \infty$ then yields the second statement.
\end{proof}
\end{corollary}
If we want to show that the Bregman iteration is a convergent regularization method in the sense of Definition \ref{def:dconvergence}, a first step towards this result would be the following monotonicity lemma.
\begin{lemma}[Fej\'{e}r monotonicity of Algorithm \ref{alg:bregiter}]\label{lem:fejer}
Let $\clean \in \oprange_\fidelfct(K)$, $\minsol \in \select(\clean, \regparambold)$ and let $\noisy \in \range$ with $\fidelity{\clean}{\noisy} \leq \delta$. We further assume that the iterates of Algorithm \ref{alg:bregiter} satisfy Definition \ref{def:morozov} for $\eta = 1$. Then the iterates also satisfy the strict Fej\'{e}r monotonicity 
\begin{align*}
\bregdis{p^k}{\minsol}{u^k} < \bregdis{p^{k - 1}}{\minsol}{u^{k - 1}} \, ,
\end{align*}
for $u^k \in \regop_I(\noisy, v^{k - 1}, \regparambold)$ and all $k < k^\ast$.
\begin{proof}
Through straight-forward computations we obtain
\begin{align*}
\bregdis{p^k}{\minsol}{u^k} - \bregdis{p^{k - 1}}{\minsol}{u^{k - 1}} {} = {} &\underbrace{-\bregdis{p^{k - 1}}{u^k}{u^{k - 1}}}_{< 0}\\
&- \langle p^k - p^{k - 1}, \minsol - u^k \rangle\\
&\leq \langle K^\ast \partial_x \fidelity{K u^k}{f^\delta}, \minsol - u^k \rangle\\
&\leq \left( \delta - \fidelity{K u^k}{f^\delta} \right)\\
&< 0
\end{align*}
for $k < k^\ast$, where we have made use of the convexity of $F$ in its first argument, and $\fidelity{K\minsol}{\noisy} \leq \fidelity{\clean}{\noisy} \leq \delta$.
\end{proof}
\end{lemma}
\begin{corollary}\label{cor:sumfidelbound}
Let $\clean \in \oprange_\fidelfct(K)$ and $\minsol \in \select(\clean, \regparambold)$. Then the iterates of Algorithm \ref{alg:bregiter} satisfy
\begin{align}
\sum_{k = 0}^\infty \fidelity{Ku^k}{\clean} < \infty \label{eq:fidelbound}
\end{align}
for $\delta = 0$ (and, thus, $\noisy = \clean$) and $u^0$ (with $p^0 \in \partial \regfctarg{u^0}$) chosen s.t. $\bregdis{p^0}{\minsol}{u^0} < \infty$.
\begin{proof}
For $\delta = 0$ we conclude
\begin{align*}
\fidelity{Ku^k}{\clean} \leq \bregdis{p^{k - 1}}{\minsol}{u^{k - 1}} - \bregdis{p^k}{\minsol}{u^k}
\end{align*}
from Lemma \ref{lem:fejer}. Summing up from $k = 0$ to some $k = k^\ast$ therefore yields
\begin{align*}
\sum_{k = 0}^{k^\ast} \fidelity{Ku^k}{\clean} \leq \bregdis{p^0}{\minsol}{u^0} - \bregdis{p^{k^\ast}}{\minsol}{u^{k^\ast}} \leq \bregdis{p^0}{\minsol}{u^0} < \infty \, .
\end{align*}
Taking the limit $k^\ast \rightarrow \infty$ yields the assertion.
\end{proof}
\end{corollary}
\begin{remark}\label{rem:oplimit}
Given the continuity of $\fidelfct$ and $K \in \linbound{\domain}{\range}$, Equation \eqref{eq:fidelbound} already implies
\begin{align}
Ku^k \wlim[\range] f \, , \label{eq:oplimit}
\end{align}
if $\tau_{\range}$ is an appropriate topology in $\range$ related to $\fidelfct$.
\end{remark}


\begin{lemma}\label{lem:bregiterselectsol}
Suppose that after a finite number of iterations the $k^\ast$-th iterate of Algorithm \ref{alg:bregiter} satisfies $Ku^{k^\ast} = \clean$, for $u^{k^\ast} = \regoparg{\clean}$, $\clean \in \oprange_K(\fidelfct)$ and $p^0 \in \oprange(K^\ast)$. Then $u^{k^\ast} \in \select(\clean, \regparambold)$.
\begin{proof}
We know $\bregdis{p^{k^\ast}}{u}{u^{k^\ast}} \geq 0$ for all $u \in \domain$ and $u^{k^\ast} = \regoparg{\clean}$, since $\regfct$ is convex; this in particular holds true for any $\hat{u} \in \{ u \, | \, Ku = \clean \}$. Hence, we observe
\begin{align*}
J(u^{k^\ast}) &\leq J(\hat{u}) - \left\langle p^{k^\ast}, \hat{u} - u^{k^\ast} \right\rangle\\
&= J(\hat{u}) - \langle p^0, \hat{u} - u^{k^\ast} \rangle + \sum_{n = 1}^{k^\ast} \left\langle K^\ast \partial_x \fidelity{Ku^n}{\clean} , \hat{u} - u^{k^\ast} \right\rangle \, ,\\
&\leq J(\hat{u}) - \langle q^0, \underbrace{K\hat{u} - Ku^{k^\ast}}_{= 0} \rangle + \sum_{n = 1}^{k^\ast} \left\langle \partial_x \fidelity{Ku^n}{\clean} , \vphantom{K^\ast \partial_x \fidelity{Ku^n}{\clean} , \hat{u} - u^{k^\ast}}\right. \underbrace{K\hat{u} - Ku^{k^\ast}}_{= 0} \left. \vphantom{K^\ast \partial_x \fidelity{Ku^n}{\clean} , \hat{u} - u^{k^\ast}} \right\rangle \, ,\\
&= J(\hat{u}) \, ,
\end{align*}
for the substitution $p^0 := K^\ast q^0$, possible due to $p^0 \in \oprange(K^\ast)$. Here we have made use of Equation \eqref{eq:bregiterdualupdate}. Consequently, we conclude $u^{k^\ast} \in \select(\clean, \regparambold)$.
\end{proof}
\end{lemma}

In the limiting case $k^\ast \rightarrow \infty$ the selection is not as clear, one cannot prove in general that the limit is minimizing $J$. To make this more apparent consider the case of a least squares fidelity
$\fidelfct(Ku,f) = \frac{1}2 \Vert Ku - f\Vert^2$ and initial value being a minimizer of the regularization, i.e. $p^0 = 0$. Then the estimate as in the last proof (at arbitrary index $m$) with $\hat u = u^\dagger$ becomes
$$J(u^{m}) \leq J(u^\dagger) - \sum_{n = 1}^m \left\langle K^\ast (f^\delta - Ku^n), \hat{u} - u^m  \right\rangle =  J(\hat{u}) - \sum_{n = 1}^m \left\langle f^\delta - Ku^n, f - K u^m  \right\rangle.$$
Using Young's inequality and monotonicity of the residual ($\Vert Ku^m - f^\delta \Vert \leq \Vert Ku^n -f^\delta \Vert$) we conclude 
$$ J(u^{m}) \leq J(u^\dagger) + \frac{3}2 \sum_{n = 1}^m \Vert Ku^n -\noisy \Vert^2 +  \frac{m}2 \delta^2. $$
Summing the estimate in the proof of the Fejer monotonicity we further find
$$  \sum_{n = 1}^m \Vert Ku^n -\noisy \Vert^2 \leq D_J^{p^0}(u^\dagger,u^0) = J(u^\dagger).$$ 
Thus, we find
$$ J(u^{k^\ast}) \leq \frac{5}2 J(u^\dagger) +  \frac{k^\ast}2 \delta^2. $$
Since for the discrepancy principle one can show that 
$k^\ast \delta^2 \rightarrow 0$ in the limit $\delta \rightarrow 0$ (cf. \cite{osher2005iterative}) the limit of the regularization has a functional value $J$ bounded by $\frac{5}2 J(u^\dagger)$. With a more fine argument on Young's inequality this upper bound can be decreased to $2 J(u^\dagger)$, but not to $J(u^\dagger)$. On the other hand this might be advantageous, since an estimate of $J(u^{k^\ast})$ smaller than $J(u^\dagger)$ might mean a bias depending on $J$, since the value of the regularization functional is actually underestimated. This means e.g. in the case of total variation that the contrast is underestimated by variational methods, which is improved by iterative regularization (cf. \cite{osher2005iterative}). 
To conclude this section, we show numerical results of Bregman-iterative regularization in the context of deconvolution, which demonstrates the effect on total variation regularization. 


\begin{example}\label{exm:deconvolution}
We consider the inverse problem of the convolution operation, i.e. $Ku = \clean$ with
\begin{align}
(Ku)(y) := \int_{\R^2} u(x) h(x - y) \, dx \, ,\label{eq:convolution}
\end{align}
which is therefore also known as \emph{deconvolution}. Here, $h$ denotes the convolution kernel that we assume to know a-priori. Since we cannot expect to know $\clean$ but just $\noisy$ with $\fidelity{\clean}{\noisy} \leq \delta$, we need to approximate the inverse problem solution through regularization. In Figure \ref{fig:pixelbregiter} we can see selected iterates of Algorithm \ref{alg:bregiter} for a single parameter $\regparam = 1/4$, the data fidelity term $\fidelity{Ku}{\noisy} = \frac{1}{2}\| Ku - \noisy \|_{L^2(\R^2)}^2$, and the regularization functional $\regfctarg[\regparam]{u} = \regparam \tv(u)$. The data $\noisy = \clean + n$ is the sum of $\clean$, created via a discretized version of the exact forward model \eqref{eq:convolution}, and noise $n \in \mathcal{N}(0, 0.05)$. For the particular example used here, the fidelity-noise-bound is $\fidelity{\clean}{\noisy} = 5.95$. The inner variational regularization method is solved via the primal-dual hybrid gradient method (PDHGM), see \cite{zhu2008efficient,pock2009algorithm,Esser:GeneralFramework,ChambollePock,chambolle2016introduction}. We clearly observe the inverse scale-space nature of the Bregman iteration. The first iterate only contains features at a very coarse scale, and then more and more features at finer and finer scales are introduced throughout the course of the iteration.
\end{example}

\begin{figure}[!t]
\begin{center}
\subfloat[Original $\minsol$]{\includegraphics[width=0.24\textwidth]{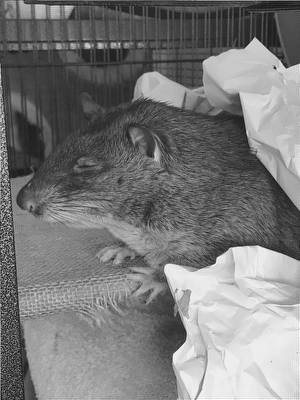}\label{subfig:pixelbregiter1}}\vspace{0.05cm}
\subfloat[Blurred \& noisy $\noisy$]{\begin{tikzpicture}[      
        every node/.style={anchor=south west,inner sep=0pt},
        x=1mm, y=1mm,
      ]   
     \node (fig1) at (0,0)
       {\includegraphics[width=0.24\textwidth]{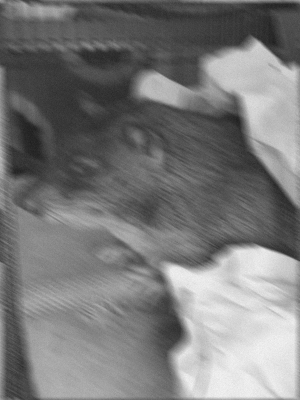}};
     \node (fig2) at (0,0)
       {\includegraphics[width=0.1\textwidth]{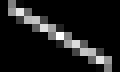}};  
\end{tikzpicture}\label{subfig:pixelbregiter2}}\vspace{0.05cm}
\subfloat[1st iterate]{\includegraphics[width=0.24\textwidth]{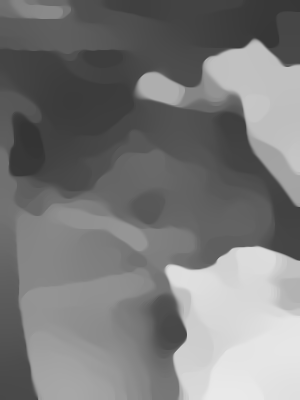}\label{subfig:pixelbregiter3}}\vspace{0.05cm}
\subfloat[3rd iterate]{\includegraphics[width=0.24\textwidth]{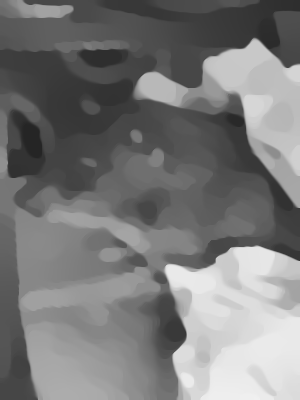}}\vspace{0.05cm}\\
\subfloat[6th iterate]{\includegraphics[width=0.24\textwidth]{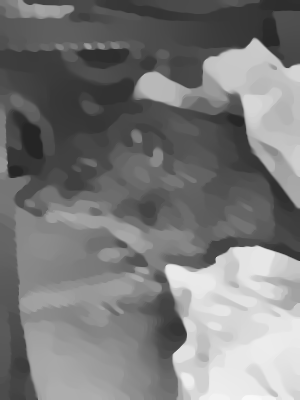}}\vspace{0.05cm}
\subfloat[20th iterate]{\includegraphics[width=0.24\textwidth]{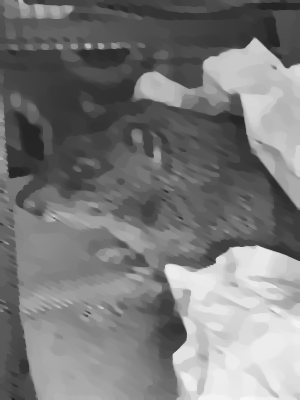}}\vspace{0.05cm}
\subfloat[55th iterate]{\includegraphics[width=0.24\textwidth]{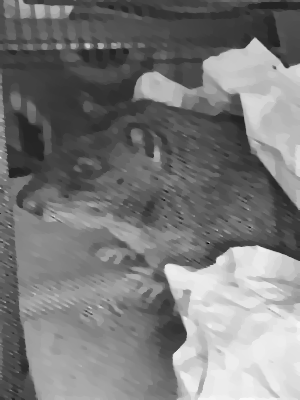}}\vspace{0.05cm}
\subfloat[96th iterate]{\includegraphics[width=0.24\textwidth]{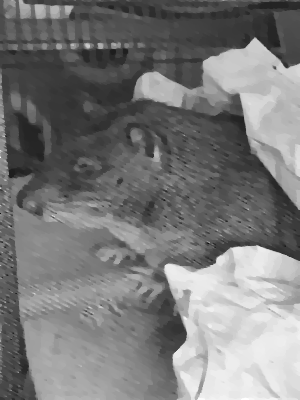}\label{subfig:pixelbregiter4}}
\end{center}
\caption{Figure \ref{subfig:pixelbregiter1} shows an image $\minsol \in \R^{400 \times 300}$ of Pixel, the Gambian pouched rat. In Figure \ref{subfig:pixelbregiter2} we see a degraded and noisy version $\noisy \in \R^{400 \times 300}$ of the original image $\minsol$. The degradation stems from a discretized version of the convolution (see \eqref{eq:convolution}) with periodic boundary conditions and the convolution kernel depicted in the bottom left corner of Figure \ref{subfig:pixelbregiter2}. Figure \ref{subfig:pixelbregiter3} - \ref{subfig:pixelbregiter4} show different iterates of Algorithm \ref{alg:bregiter} for $\fidelity{Ku}{\noisy} = \frac{1}{2}\| Ku - \noisy \|_{L^2(\R^2)}^2$, $\regfctarg[\regparam]{u} = \regparam \tv(u)$ and $\regparam = 1/4$. The 96th iterate visualized in Figure \ref{subfig:pixelbregiter4} is the first that violates Definition \ref{def:morozov}, for $\eta = 1$ and $\delta = 5.95$.}\label{fig:pixelbregiter}
\end{figure}

\subsubsection*{Debiasing generalized Eigenfunctions}\label{subsubsec:debiaseigenfcts}
We want to continue the analysis of the generalized Eigenvalue problem introduced in Section \ref{sec:vareigprob}. We have figured out that there is always as systematic bias of variational regularization methods for $\fidelity{Ku}{\noisy} = \frac{1}{2} \| Ku - \noisy \|_{L^2(\Sigma)}^2$, i.e. $\lambda < 1$ in \eqref{eq:eigenvalue} for $\noisy = v_\sigma$. It has been shown in \cite{Benning:PhdThesis} that this systematic bias can be corrected with the help of Bregman iterations in case of scalar $\regparambold = \regparam$ and $\regfct$ with $\partial \regfctarg[\regparam]{c u} = \regparam \partial \regfct(u)$. Assume that $\regparam$ is chosen such that $u^k = 0$ for all $k < k^\ast - 1$, and $u^{k^\ast - 1} = \frac{1 - \regparam}{\sigma} u_\sigma$ for some $k^\ast \in \N$. Then we can easily conclude from \eqref{eq:varregopt} and \eqref{eq:bregiterprimalupdate} that $u^{k^\ast}$ has to satisfy 
\begin{align*}
\frac{1}{\regparam} \left( \lambda - \frac{1 - \regparam}{\sigma} \right) K^\ast K u^{k^\ast} \in \partial \regfct(\lambda u^{k^\ast}) \, .
\end{align*}
We easily calculate that the above equation simplifies to the singular vector condition \eqref{eq:singvec} for the choice $\lambda = 1/\sigma$. Consequently, $u^{k^\ast} = \regoparg[\regparam]{\noisy} = u_\sigma / \sigma$, and we have corrected for the bias of the previous iterate.\\

\begin{figure}[!t]
\begin{center}
\includegraphics[width=0.5\textwidth]{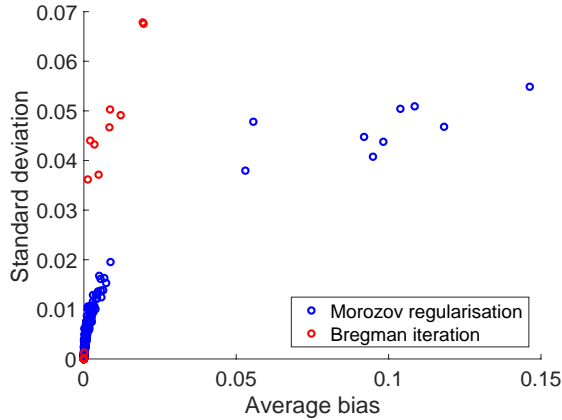}
\end{center}
\caption{We see \eqref{eq:stdandbias} for the compressed sensing toy example in Section \ref{subsubsec:debiaseigenfcts}. The blue circles represent the standard deviation and average absolute bias values for all coefficients recovered with \eqref{eq:biasmorozov}. The red circles show the same quantities for all coefficients recovered with \eqref{eq:biasbregman}. It becomes evident that for this example the average bias is significantly reduced, while the standard deviation of the reconstructed coefficients is comparable.}\label{fig:biasreduction}
\end{figure}

The previous computations demonstrate that Bregman iterations correct for the systematic bias of variational regularization reconstructions of generalized singular vectors in case of one-homogeneous regularization functionals $\regfct$. However, the phenomenon is not limited to singular vectors. The following numerical toy example shows that the average reconstruction bias can be significantly reduced with the help of Bregman iterations. Suppose the following setup. Our forward model $K \in \R^{m \times n}$, for $m = 128$ and $n = 512$, is a matrix with its entries drawn randomly from $\mathcal{N}(0, 1)$. We define a sparse vector $\minsol \in \R^n$ with nine non-zero entries, drawn randomly from $\mathcal{N}(0, 1)$, and set $\clean = K\minsol$. Subsequently, we create one hundred instances of noisy data via $\noisy_j := \clean + n_j$, for $n \in \mathcal{N}(0, 0.5)$ and $j \in \{1, \ldots, 100\}$. We now compute reconstructions for each of the one hundred instances with the following two regularization methods:
\begin{align}
\regop_{\text{Morozov}}(\noisy_j, \delta_j) = \argmin_{u \in \R^n} \left\{ \| u \|_1 \quad \text{subject to} \quad \| Ku - \noisy_j \|_2 \leq \delta_j \right\} \, ,\label{eq:biasmorozov}
\end{align}
and
\begin{align}
\regop_{\text{Bregman}}(\noisy_j, \{u^n_j \}_{n = 1}^{k - 1}, \regparam) = \argmin_{u \in \R^n} \left\{ \frac{1}{2} \left\| Ku - \left( k \noisy - \sum_{n = 1}^{k - 1} Ku^n_j \right) \right\|_2^2 + \regparam \| u \|_1 \right\} \, ,\label{eq:biasbregman}
\end{align}
for $u^{\text{Morozov}}_j \in \regop_{\text{Morozov}}(\noisy_j, \delta_j)$, $\delta_j := \frac{1}{2} \| K\minsol - \noisy_j \|_2^2$, $u^{\text{Bregman}}_j \in \regop_{\text{Bregman}}(\noisy_j, \{u^n_j \}_{n = 1}^{k^\ast - 1}, \regparam)$, and $k^\ast$ chosen according to Definition \ref{def:morozov} for $\eta = 1$ and $\delta = \delta_j$. We then compute the average absolute bias and the standard deviation of the reconstructions, i.e. we compute
\begin{align}
\left| \minsol - \frac{1}{100} \sum_{j = 1}^{100} \hat{u}_j \right| \quad \text{and} \quad \sqrt{\frac{1}{99} \sum_{j = 1}^{100} \left(\hat{u}_j - \frac{1}{100} \sum_{j = 1}^{100} \hat{u}_j \right)^2}\label{eq:stdandbias}
\end{align}
for $\hat{u}_j \in \{ u^{\text{Morozov}}_j, u^{\text{Bregman}}_j \}$. Both average absolute bias and standard deviation are visualized for each of the $n = 512$ coefficients in Figure \ref{fig:biasreduction}. We clearly observe that with similar standard deviation, the average absolute bias is significantly reduced by the Bregman iteration in comparison to the Morozov regularization model.     

\subsection{Linearized Bregman Iteration}\label{subsec:linbregiter}

As the name suggests, the linearized Bregman iteration can be derived from Algorithm \ref{alg:bregiter} by replacing the term $\fidelity{Ku^k}{\noisy}$ with its linearization
\begin{align*}
\fidelity{Ku^k}{\noisy} \approx \fidelity{Ku^{k - 1}}{\noisy} + \langle \partial_x \fidelity{Ku^{k - 1}}{\noisy} , Ku^k - Ku^{k - 1} \rangle \, .
\end{align*}
Hence, if we replace $\fidelity{Ku^k}{\noisy}$ in Algorithm \ref{alg:bregiter} with this linearization multiplied by some constant $\tau > 0$, we obtain 
\begin{align*}
\regopitarg{\noisy, v^{k - 1}} {} = {} &\argmin_{u \in \domain} \left\{ \tau \left( \fidelity{Ku^{k - 1}}{\noisy} + \langle \partial_x \fidelity{Ku^{k - 1}}{\noisy} , Ku^k - Ku^{k - 1} \rangle \right) \right.\nonumber\\
&\left. \quad + \bregdis[\regfct(\cdot, \regparam)]{p^{k - 1}}{u}{u^{k - 1}} \right\} \nonumber \, ,\\
{} = {} &\argmin_{u \in \domain} \left\{ \tau \langle \partial_x \fidelity{Ku^{k - 1}}{\noisy} , Ku^k - Ku^{k - 1} \rangle + \bregdis[\regfct(\cdot, \regparam)]{p^{k - 1}}{u}{u^{k - 1}} \right\} \, ,\\
u^k {} \in {} &\regopitarg{\noisy, v^{k - 1}} \nonumber\\
p^k {} = {} &p^{k - 1} - \tau K^\ast \partial_x \fidelity{Ku^{k - 1}}{\noisy} \, ,
\end{align*}
for $\regparambold = (\tau, \regparam)$ and $v^{k - 1} := (u^k, p^k)$ for all $k \in \N$. These equations are summarized in Algorithm \ref{alg:linbregiter}.

The linearized Bregman iteration is a generalization of the Landweber regularization \cite{landweber1951iteration} for the choices $\fidelity{Ku}{\noisy} = \frac{1}{2} \| Ku - \noisy \|_{L^2(\Sigma)}^2$ and $\regfct(u, \regparam) = \frac{\regparam}{2} \| u \|_{L^2(\Omega)}^2$, for some signal domains $\Omega$ and $\Sigma$. It is also a generalization of the mirror descent algorithm proposed in \cite{nemirovskii1983problem}, where $\regfct(\cdot, \regparam)$ is a Legendre functional in the sense of Definition \ref{def:legendre}. This connection for convex, differentiable $\fidelfct$ and strongly-convex and differentiable $\regfct(\cdot, \regparam)$ was made in \cite{beck2003mirror}. The extension to subdifferentiable convex $\regfct(\cdot, \regparam)$ was first proposed in \cite{darbon2007fast} and has since been studied extensively \cite{yin,cai1,cai2,yin2}.

\begin{algorithm}[!ht]
\caption{Linearized Bregman iteration}
\label{alg:linbregiter}
\begin{algorithmic}
\State{Initialize $u^0 \in \domain$, $p^0$ with $p^0 \in \partial \regfctarg[\regparam]{u^0}$, $r^0 = K^\ast \partial_x \fidelity{K u^0}{\noisy}$, $\regparambold = (\regparam, \tau) \in \regdomain$}
\While{stopping criterion is not satisfied}
\State{Compute $\regopitarg{\noisy, v^{k - 1}} = \argmin_{u \in \domain} \left\{ \tau \langle r^{k - 1}, u \rangle + \bregdis[\regfct(\cdot, \regparam)]{p^{k - 1}}{u}{u^{k - 1}} \right\}$}
\State{Pick $u^k \in \regopitarg{\noisy, v^{k - 1}} \vphantom{\argmin_{u \in \domain} \left\{ \tau \langle r^{k - 1}, u \rangle + \bregdis{p^{k - 1}}{u}{u^{k - 1}} \right\}}$}
\State{Update $p^k = p^{k - 1} - \tau \, r^{k - 1} \vphantom{\argmin_{u \in \domain} \left\{ \tau \langle r^{k - 1}, u \rangle + \bregdis{p^{k - 1}}{u}{u^{k - 1}} \right\}}$}
\State{Compute $r^k = K^\ast \partial_x \fidelity{K u^k}{\noisy} \vphantom{\argmin_{u \in \domain} \left\{ \tau \langle r^{k - 1}, u \rangle + \bregdis{p^{k - 1}}{u}{u^{k - 1}} \right\}}$}
\State{Set $v^k = (u^k, p^k) \vphantom{\argmin_{u \in \domain} \left\{ \tau \langle r^{k - 1}, u \rangle + \bregdis{p^{k - 1}}{u}{u^{k - 1}} \right\}}$}
\EndWhile\\
\Return $u^{k^\ast}$, $p^{k^\ast}$
\end{algorithmic}
\end{algorithm}

\noindent Similar to Remark \ref{rem:primdualsumformula}, we can rewrite the dual update of the linearized Bregman iteration as
\begin{align}
p^k = p^0 - \sum_{n = 0}^{k - 1} K^\ast \partial_x \fidelity{Ku^n}{\noisy} \, ,\label{eq:linbregiterdualupdate}
\end{align}
and the primal update as
\begin{align}
\regopit(\noisy, \{ u^n \}_{n = 0}^{k - 1}, p^0, \regparam) = \argmin_{u \in \domain} \left\{ \regfct(u, \regparam) - \left\langle p^0 - \sum_{n = 0}^{k - 1} K^\ast \partial_x \fidelity{Ku^n}{\noisy}, u \right\rangle \right\} \, .\label{eq:linbregiterprimalupdate}
\end{align}

In order to carry out a convergence analysis similar to the analysis for the standard Bregman iteration, we define the surrogate functional
\begin{align}
\surrogatearg{u} := \regfct(u, \regparam) - \tau \fidelity{Ku}{\noisy} \, .
\end{align}
We further assume for the remainder of this section that $\regfct$ and $\tau$ are chosen such that $\surrogate$ is convex. In practice, this requires strong convexity properties of $\regfct$, which can simply be established by adding a sufficiently strongly convex functional to the original choice of $\regfct$.
\begin{example}
Let $K \in \linbound{L^2(\Omega)}{L^2(\Sigma)}$ and $\fidelity{Ku}{\noisy} = \frac{1}{2}\|Ku - \noisy\|_{L^2(\Sigma)}^2$, for domains $\Omega \subset \R^n$ and $\Sigma \subset \R^m$, and let $\regfct_1$ be a proper, l.s.c. and convex functional. Then the functional $\regfct_\tau(u, \regparam) := \regfctarg[\regparam]{u} - \frac{\tau}{2}\|Ku - \noisy\|_{L^2(\Sigma)}^2$ is convex for the choices
\begin{align*}
\regfctarg[\regparam]{u} := \frac{1}{2} \| u \|^2_{L^2(\Omega)} + \regfct_1(u) \quad \text{and} \quad \tau < \frac{1}{\| K \|_{\linbound{L^2(\Omega)}{L^2(\Sigma)}}^2} \, .
\end{align*}
\end{example}

Similar to the Bregman iteration analysis, we start with a statement about the monotonic decrease of the data fidelity.
\begin{corollary}[Monotonic decrease of the data fidelity]\label{cor:mondeclinbreg}
Suppose $u^0$ satisfies $\fidelity{Ku^0}{\noisy} < \infty$. Then the iterates of Algorithm \ref{alg:linbregiter} satisfy 
\begin{align}
\fidelity{Ku^{k + 1}}{\noisy} + \frac{1}{\tau} \bregdis[\surrogatearg{\cdot}]{q^k}{u^{k + 1}}{u^k} \leq \fidelity{Ku^{k + 1}}{\noisy} \label{eq:mondeclinbreg}
\end{align}
and
\begin{align*}
\lim_{k \rightarrow \infty} \bregdis[\surrogatearg{\cdot}]{q^k}{u^{k + 1}}{u^k} = 0 \, ,
\end{align*}
for $u^k \in \regopitarg{\noisy, v^{k - 1}}$ and $q^k \in \partial \surrogatearg{u^k}$.
\begin{proof}
First of all we highlight that $\langle r^{k - 1}, u^k - u^{k - 1} \rangle = \langle K^\ast \partial_x \fidelity{Ku^{k - 1}}{\noisy}, u^k - u^{k - 1} \rangle$ can be written as
\begin{align*}
\langle K^\ast \partial_x \fidelity{Ku^{k - 1}}{\noisy}, u^k - u^{k - 1} \rangle = \fidelity{Ku^k}{\noisy} - \fidelity{Ku^{k - 1}}{\noisy} - \bregdis[\fidelity{K \cdot}{\noisy}]{}{u^k}{u^{k - 1}} \, ,
\end{align*}
for all $k \in \N$. Hence, the (primal) update of the linearized Bregman iteration can be rewritten to
\begin{align*}
\regopitarg{\noisy, v^{k - 1}} = \argmin_{u \in \domain} \left\{ \tau \left( \fidelity{Ku^k}{\noisy} - \fidelity{Ku^{k - 1}}{\noisy} \right) + \bregdis[\surrogatearg{\cdot}]{q^{k - 1}}{u^k}{u^{k - 1}} \right\} \, ,
\end{align*}
for $q^{k - 1} = p^{k - 1} - \tau K^\ast \partial_x \fidelity{Ku^{k - 1}}{\noisy} \in \partial \surrogatearg{u^{k - 1}}$ and $p^{k - 1} \in \partial \regfct(u^{k - 1}, \regparam)$. Hence, we conclude
\begin{align*}
&\tau \left( \fidelity{Ku^k}{\noisy} - \fidelity{Ku^{k - 1}}{\noisy} \right) + \bregdis[\surrogatearg{\cdot}]{q^{k - 1}}{u^k}{u^{k - 1}} \\
{} \leq {} &\underbrace{\tau \left( \fidelity{Ku^{k - 1}}{\noisy} - \fidelity{Ku^{k - 1}}{\noisy} \right) }_{= 0} + \underbrace{\bregdis[\surrogatearg{\cdot}]{q^{k - 1}}{u^{k - 1}}{u^{k - 1}}}_{= 0} \, ,
\end{align*}
and thus, Equation \eqref{eq:mondeclinbreg}. In the same fashion as in the proof of Corollary \ref{cor:mondecbreg} we further conclude $\lim_{k \rightarrow \infty} \bregdis[\surrogatearg{\cdot}]{q^k}{u^{k + 1}}{u^k} = 0$.
\end{proof}
\end{corollary}

As in the case of the standard Bregman iteration, the linearized Bregman iteration also satisfies Fe\'{j}er monotonicity in case the discrepancy principle is not violated.

\begin{lemma}[Fej\'{e}r monotonicity of Algorithm \ref{alg:linbregiter}]
Let $\clean \in \oprange_\fidelfct(K)$, $\minsol \in \select(\clean, \regparam)$ and let $\noisy \in \range$ with $\fidelity{\clean}{\noisy} \leq \delta$. We further assume that the iterates of Algorithm \ref{alg:linbregiter} satisfy Definition \ref{def:morozov} for $\eta = 1$. Then the iterates also satisfy the strict Fej\'{e}r monotonicity 
\begin{align*}
\bregdis[\surrogatearg{\cdot}]{q^k}{\minsol}{u^k} < \bregdis[\surrogatearg{\cdot}]{q^{k - 1}}{\minsol}{u^{k - 1}} \, ,
\end{align*}
for all $u^k \in \regopitarg{\noisy, v^{k - 1}}$ and $q^k \in \partial \surrogatearg{u^k}$, for all $k \leq k^\ast$.
\begin{proof}
Through straight-forward computations we obtain
\begin{align*}
\bregdis[\surrogatearg{\cdot}]{q^k}{\minsol}{u^k} - \bregdis[\surrogatearg{\cdot}]{q^{k - 1}}{\minsol}{u^{k - 1}} {} = {} &\underbrace{-\bregdis[\surrogatearg{\cdot}]{q^{k - 1}}{u^k}{u^{k - 1}}}_{< 0}\\
&- \langle q^k - q^{k - 1}, \minsol - u^k \rangle\\
&\leq \left\langle p^k - p^{k - 1} - \tau K^\ast (\partial_x \fidelity{Ku^k}{\noisy} - \partial_x \fidelity{Ku^{k - 1}}{\noisy}), \minsol - u^k\right\rangle \\
&= \tau \langle K^\ast \partial_x \fidelity{Ku^k}{\noisy}, \minsol - u^k \rangle\\
&\leq \tau \left( \delta - \fidelity{Ku^k}{f^\delta} \right)\\
&< 0
\end{align*}
for $k \leq k^\ast$, where we have made use of the convexity of $F$ in its first argument, and $F(Ku^\dagger, \noisy) \leq \delta$.
\end{proof}
\end{lemma}
\noindent In analogy to Corollary \ref{cor:sumfidelbound}, we can show the same result for the linearized Bregman iteration
\begin{corollary}
Let $\clean \in \oprange_\fidelfct(K)$ and $\minsol \in \select(\clean, \regparam)$. Then the iterates of Algorithm \ref{alg:linbregiter} satisfy \eqref{eq:fidelbound}, for $\delta = 0$ (and thus $\noisy = \clean$) and $u^0$ (with $q^0 \in \partial \surrogatearg{u^0}$) chosen such that $\bregdis[\surrogatearg{\cdot}]{q^0}{\minsol}{u^0} < \infty$.
\begin{proof}
The proof follows the exact same steps as the proof of Corollary \ref{cor:sumfidelbound}.
\end{proof}
\end{corollary}
\noindent As in the case of Bregman iteration, Remark \ref{rem:oplimit} follows from this result.

\begin{figure}[!t]
\begin{center}
\subfloat[Original $\minsol$]{\includegraphics[width=0.24\textwidth]{Images/Pixelsmall.png}\label{subfig:pixellinbregiter1}}\vspace{0.05cm}
\subfloat[Blurred \& noisy $\noisy$]{\begin{tikzpicture}[      
        every node/.style={anchor=south west,inner sep=0pt},
        x=1mm, y=1mm,
      ]   
     \node (fig1) at (0,0)
       {\includegraphics[width=0.24\textwidth]{Images/Pixelblurrednoisy.png}};
     \node (fig2) at (0,0)
       {\includegraphics[width=0.1\textwidth]{Images/convkernel.png}};  
\end{tikzpicture}\label{subfig:pixellinbregiter2}}\vspace{0.05cm}
\subfloat[1st iterate]{\includegraphics[width=0.24\textwidth]{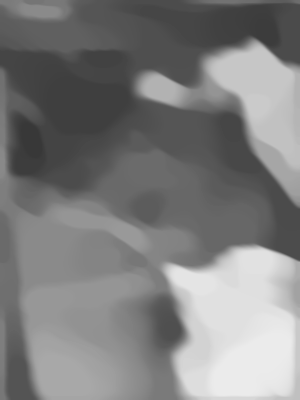}\label{subfig:pixellinbregiter3}}\vspace{0.05cm}
\subfloat[3rd iterate]{\includegraphics[width=0.24\textwidth]{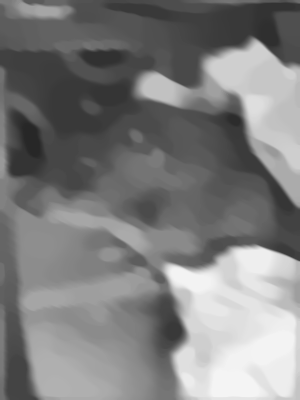}}\vspace{0.05cm}\\
\subfloat[6th iterate]{\includegraphics[width=0.24\textwidth]{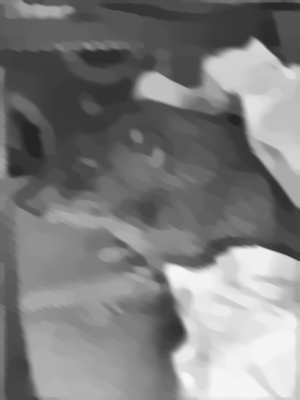}}\vspace{0.05cm}
\subfloat[20th iterate]{\includegraphics[width=0.24\textwidth]{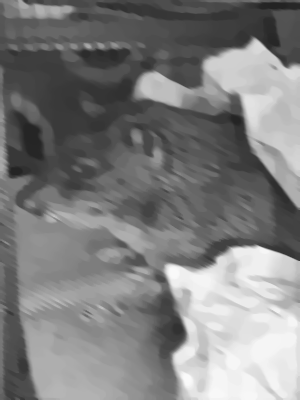}}\vspace{0.05cm}
\subfloat[70th iterate]{\includegraphics[width=0.24\textwidth]{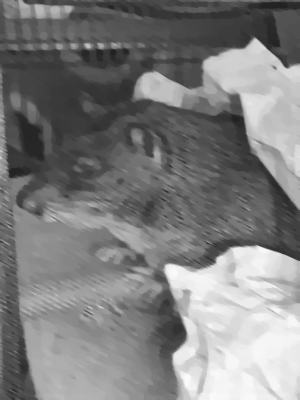}}\vspace{0.05cm}
\subfloat[128th iterate]{\includegraphics[width=0.24\textwidth]{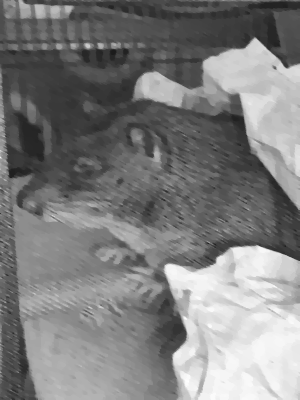}\label{subfig:pixellinbregiter4}}
\end{center}
\caption{Figure \ref{subfig:pixellinbregiter1} shows the image $\minsol \in \R^{400 \times 300}$ of Pixel, the Gambian pouched rat, originally introduced in Figure \ref{subfig:pixelbregiter1}. In Figure \ref{subfig:pixellinbregiter2} we see the same degraded and noisy version $\noisy \in \R^{400 \times 300}$ together with the convolution kernel $h$ as shown in Figure \ref{subfig:pixelbregiter2}. Figure \ref{subfig:pixellinbregiter3} - \ref{subfig:pixellinbregiter4} show different iterates of Algorithm \ref{alg:linbregiter} for $\fidelity{Ku}{\noisy} = \frac{1}{2}\| Ku - \noisy \|_{L^2(\R^2)}^2$, $\regfct(u, \regparam) = \frac{1}{2} \| u \|_{L^2(\R^2)}^2 + \regparam \tv(u)$ and $\regparam = 1/4$. The 128th iterate visualized in Figure \ref{subfig:pixellinbregiter4} is the first that violates Definition \ref{def:morozov}, for $\delta = 5.95$.}\label{fig:pixellinbregiter}
\end{figure}

The following result guarantees converge to a solution in $\select(\clean, \regparam)$ in case $Ku^{k^\ast} = \clean$ is satisfied after an finite number $k^\ast$ of iterations of Algorithm \ref{alg:linbregiter}.
\begin{lemma}\label{lem:linbregiterselectsol}
Suppose that after a finite number of iterations the $k^\ast$-th iterate of Algorithm \ref{alg:linbregiter} satisfies $Ku^{k^\ast} = \clean$, for $u^{k^\ast} \in \regoparg{\clean}$, $\clean \in \oprange_K(\fidelfct)$ and $p^0 \in \oprange(K^\ast)$. Then $u^{k^\ast} \in \select(\clean, \regparam)$.
\begin{proof}
The proof is almost identical to the proof of Lemma \ref{lem:bregiterselectsol}; the only difference is that we use \eqref{eq:linbregiterdualupdate} instead of \eqref{eq:bregiterdualupdate}.
\end{proof}
\end{lemma}

\begin{remark}
Note that the statements of Lemma \ref{lem:bregiterselectsol} and Lemma \ref{lem:linbregiterselectsol} look identical, but one needs to remember that the underlying functionals $\regfct$ will most likely not be. This is due to the fact that for the linearized Bregman iteration additional terms have to be added in order to also make $\surrogate$ convex.
\end{remark}

We conclude this section with numerical results for the same deconvolution example introduced in Example \ref{exm:deconvolution}. We observe that with the same choice of regularization parameter and the same initialization, Algorithm \ref{alg:linbregiter} requires more iterations in order to converge to a solution that violates the discrepancy principle with the same error bound. On the other hand, the variational subproblems are computationally cheaper to solve compared to the standard Bregman iteration case, at least with the (accelerated) PDHGM used for this example.

\subsection{Coupled and Modified Bregman Iterations}

The Bregman iteration (as well as its linearized variant) leave some freedom for modification, an obvious one comes with respect to the choice of the subgradient $p^{k-1}$. The update from the optimality condition is of course the obvious one and particularly suited for a convergence proof. However, one may use different ways to determine a subgradient $p^k$ from $u^k$. As an example one may solve some variational problem
$$ p^k \in \argmin_p \{ H(p,p^{k-1}) ~|~ p \in \partial J(u^{k},{\bf \alpha}) \} , $$
with some convex functional $H$. In the case of $\ell^1$ minimization one might choose $H(p,p^{k-1})=\Vert p \Vert_2$, which yields the minimal subgradient, i.e. choosing again sign$_0$ in the case of a multivalued sign. 

Another option for choosing subgradients has been investigated in \cite{Moeller:ColorBregmanTV} when one solves joint reconstruction problems for multiple unknowns $u_1,\ldots,u_M$. There a coupled Bregman iteration was proposed and analyzed, which is based on choosing a new subgradient for the Bregman iteration in the $i$-th image $u_i$ from a linear combination of the subgradients in the other channels. In this way a joint subgradient for all the channels is approximated, which means a structural joint sparsity in the case of the $\ell^1$-norm or joint edge information in the total variation case. 
In \cite{rasch2017joint} an infimal convolution version of the coupled Bregman iteration has been investigated for an application to PET-MR imaging.

\section{Bias and Scales}

The previous arguments related to eigenfunctions demonstrate that bias and scale are closely related (at least when interpreting scale in terms of eigenfunctions and eigenvalues). The bias of variational regularization methods is larger on small scale features. Thus, debiasing and multiscale aspects in regularization methods appear closely related as it has been worked out very recently. We discuss those ideas in the following. 

\subsection{Inverse Scale space}
For regularization functionals of the form $\regfctarg[\regparambold]{u} = \regparam \regfct_1(u)$ we can write the dual Bregman iteration update as
\begin{align*}
\frac{p^k - p^{k - 1}}{\Delta t} = - K^\ast \partial_x \fidelity{Ku^k}{\noisy}
\end{align*}
for $\Delta t := 1/\regparam$ and $p^k \in \partial \regfct_1(u^k)$, for all $k \in \N$. Thus, taking the limit $\alpha \rightarrow \infty$, respectively $\Delta t \rightarrow 0$, yields the following time-continuous formulation of the Bregman iteration, also known as the \emph{inverse scale space flow} \cite{burger2005nonlinear,burger2006nonlinear,burger2007inverse},
\begin{align}
\partial_t p(t) = -K^\ast \partial_x \fidelity{Ku(t)}{\noisy} \, ,\label{eq:invscalspaceflow}
\end{align}
for $p(t) \in \partial \regfct_1(u(t))$.

For many typical choices of regularization functionals $\regfct_1$ it is difficult to numerically compute solutions of \eqref{eq:invscalspaceflow}, with the $\ell^1$ norm and in general any polyhedral regularization functional being the exception (cf. \cite{ourpaper,moel12,moeller2013multiscale}). Nevertheless, \eqref{eq:invscalspaceflow} is very useful to study theoretical properties of iterative regularizations in the limiting case.

Unsurprisingly, it is straight-forward to carry out an Eigenanalysis similar to the one discussed in Section \ref{sec:vareigprob} and Section \ref{subsubsec:debiaseigenfcts} for the regularization operator $\regoparg[t]{\noisy} = u(t)$ with $u(t)$ satisfying \eqref{eq:invscalspaceflow} in case of $\fidelity{Ku}{\noisy} = G(Ku - \noisy)$. The following result is a generalization of \cite[Theorem 9]{benning2013ground}.
\begin{theorem}
Let $(u_\sigma, v_\sigma)$ be a pair of generalized singular vectors with singular value $\sigma$, $\clean = v_\sigma$ and suppose $\regfct_1$ is (absolutely) one-homogeneous, i.e. $\regfct_1(c u) = |c| \regfct_1(u)$ for all $c \in \R$. Then $0 \in \regoparg[t]{v_\sigma}$ for $0 \leq t < t_\ast$ and
\begin{align*}
\frac{1}{\sigma} u_\sigma \in \regoparg[t]{v_\sigma}
\end{align*}
for $t \geq t_\ast = 1$.
\begin{proof}
Firstly, we verify $0 \in \regoparg[t]{v_\sigma}$ for $0 \in [0, t_\ast )$. From \eqref{eq:singvec} and the absolute one-homogeneity of $\regfct_1$ we observe $\regfct_1(u_\sigma) = \langle G^\prime(v_\sigma), Ku_\sigma \rangle$. We further see from the definition of the subdifferential that $t \leq 1 =  \langle G^\prime(v_\sigma), Ku_\sigma \rangle / \regfct_1(u_\sigma)$ implies $p(t) := t K^\ast G^\prime(v_\sigma) \in \partial \regfct_1(0)$. Since $\partial_t p(t) = K^\ast G^\prime(v_\sigma)$ and $p(0) = 0$, we have shown that $u(t) = 0$ is a solution of \eqref{eq:invscalspaceflow}.

\noindent For $t \geq t_\ast$ a continuous extension of $p(t)$ is 
\begin{align*}
p(t) = p(t_\ast) + (t_\ast - t) K^\ast G^\prime(Ku(t) - v_\sigma) \, .
\end{align*}
We immediately see that $u(t) = u_\sigma / \sigma$ is a solution for $t \geq t_\ast$, since $p(t_\ast) = t_\ast K^\ast G^\prime(v_\sigma) \in 
\partial \regfct_1(u_\sigma / \sigma)$ and $\partial_t p(t) = 0$. 
\end{proof}
\end{theorem}
Hence, the inverse scale space reconstruction also has no bias (for input data $v_\sigma$ satisfying \eqref{eq:singvec}), compared to the variational regularization method.

A similar result can be derived even in the case of noisy data $\noisy = v_\sigma + n$, where $n$ is an error term that satisfies the specific source condition 
\begin{align*}
\mu K^\ast G^\prime(v_\sigma) + \eta K^\ast n \in \partial \regfct(\sigma u_\sigma) \, ,
\end{align*}
for constants $\mu$ and $\eta$. For more details we refer to \cite[Theorem 10]{benning2013ground}.

In the following we briefly want to discuss reconstruction guarantees for linear combinations of multiple singular vectors. Precisely we ask for what times can we guarantee
\begin{align*}
\frac{\gamma_j}{\sigma_j} u_{\sigma_j} \in \regop\left(\sum_{j = 1}^n \gamma_j v_{\sigma_j}, t\right) \, ,
\end{align*}
for coefficients $\{ \gamma_j \}_{j \in \N}$. Due to the nonlinearity of $\regfct_1$, we can in general not expect such a decomposition. If we restrict ourselves to the following two conditions, however, such a result can be guaranteed (see \cite[Theorem 3.14]{schmidt2016inverse}). The first condition is $K$-orthogonality of the singular vectors, i.e.
\begin{align}
\langle Ku_{\sigma_i}, Ku_{\sigma_j} \rangle = \begin{cases}
1 & i = j\\ 0 & i \neq j
\end{cases} \, ,\tag{OC}\label{eq:oc}
\end{align}
for $i, j \in \{1, \ldots, n\}$. The second condition is the so-called \eqref{eq:sub0}-condition, which reads as follows:
\begin{definition}[{\cite[Definition 3.1]{schmidt2016inverse}}]
Let $(u_{\sigma_1}, u_{\sigma_2}, \ldots, u_{\sigma_n})$ be an ordered set of primal singular vectors of $\regfct_1$ with corresponding dual singular vectors $(v_{\sigma_1}, v_{\sigma_2}, \ldots, v_{\sigma_n})$ and singular values $(\sigma_1, \sigma_2, \ldots, \sigma_n)$. Then the singular vectors satisfy the \emph{(SUB0) condition} if
\begin{align}
\sum_{j = 1}^k K^\ast G^\prime( v_{\sigma_j} ) \in \partial \regfct_1(0) \, , \tag{SUB0}\label{eq:sub0}
\end{align}
for all $k \in \{1, \ldots, n\}$.
\end{definition}
\noindent Given \eqref{eq:oc} and \eqref{eq:sub0}, we can guarantee the following decomposition result, which is a direct generalization of \cite[Theorem 3.14]{schmidt2016inverse}.
\begin{theorem}
Let $(u_{\sigma_1}, u_{\sigma_2}, \ldots, u_{\sigma_n})$, $(v_{\sigma_1}, v_{\sigma_2}, \ldots, v_{\sigma_n})$ and $(\sigma_1, \sigma_2, \ldots, \sigma_n)$ be a system of ordered singular vectors, for which the $v_j$'s are normalized, and for which \eqref{eq:oc} and \eqref{eq:sub0} are satisfied. Then, for data $\clean = \sum_{j = 1}^n \gamma_j v_{\sigma_j}$ with positive coefficients $(\gamma_1, \ldots, \gamma_n)$ we have $u(t) \in \regop(\clean, t)$, with
\begin{align*}
u(t) = \begin{cases}
0 & 0 \leq t \leq t_1\\
\sum_{j = 1}^k \frac{\gamma_j}{\sigma_j} u_{\sigma_j} & t_k \leq t < t_{k + 1}, \, \text{for all $k = 1, \ldots, n - 1$}\\
\sum_{j = 1}^n \frac{\gamma_j}{\sigma_j} u_{\sigma_j} & t_n \leq t
\end{cases} \, ,
\end{align*}
where $t_k = \gamma_k$ and $t_k < t_{k + 1}$ for all $k \in \{ 1, \ldots, n\}$.
\end{theorem}

We refer to \cite{benning2013ground} for more information on individual generalized singular vectors and the inverse scale space flow. For more theoretical results and analytical as well as numerical examples of ordered sets of singular vectors that satisfy \eqref{eq:oc} and \eqref{eq:sub0}, we refer to \cite{schmidt2016inverse}.

\subsection{Two-Step Debiasing} 

While Bregman iterations and inverse scale space methods perform debiasing in an iterative fashion (and effectively change the variational model), one may also consider two-step procedures that first solve the original variational model and then perform a second step to reduce the bias (cf. \cite{deledalle2015debiasing,deledalle2017clear}). The first and simplest case where this idea was brought up is regularization with the $\ell^1$-norm, where a so-called {\em refitting} strategy (cf. \cite{lederer2013trust}) is quite natural. After the variational problem
\begin{equation}
\regsol[\regparam]_\delta \in \argmin \fidelity{Ku}{\noisy} + \alpha \Vert u \Vert_{\ell^1} 
\end{equation}
is solved, the second step simply consists in minimizing $\fidelity{Ku}{\noisy}$ over the set of all $u$ sharing the support of $u_\alpha^\delta$.  Since this procedure throws away information about the sign of the entries of $u$, one can further improve to define the regularization operator via
\begin{equation}
	\regoparg[\regparam]{\noisy}= \argmin \{ \fidelity{Ku}{\noisy}~|   ~\text{sign}_0(u_i) = \text{sign}_0((\regsol[\regparam]_\delta)_i), \forall~i \}.
\end{equation}
where sign$_0(u_i)$ is the single-valued sign (i.e. zero for $u_i=0$).
Since the sign corresponds to a subgradient of the $\ell^1$-norm, we can reinterpret the debiased regularization operator in a variational way: We minimize the fidelity subject to the constraint of $u$ sharing a subgradient with $\regsol[\regparam]_\delta$. This is a key observation towards a generalization for arbitrary convex regularizations as noticed in (cf. \cite{brinkmann2017bias}). The general debiasing problem can be rephrased as a two step procedure
\begin{equation}
\regsol[\regparam]_\delta	\in \argmin \fidelity{Ku}{\noisy} + \regparam \regfct(u)
\end{equation}
followed by 
\begin{equation}
	\regoparg[\regparam]{\noisy}= \argmin \{ \fidelity{Ku}{\noisy}~|   ~p \in \partial \regfct(u) \cap \partial \regfct(\regsol[\regparam]_\delta) \}.
\end{equation}
For computational purposes the arbitrary choice of the subgradient $p \in \partial \regfct(\regsol[\regparam]_\delta)$ is not suitable, but we can indeed use the subgradient from the first step. Noticing that for differentiable fidelities, the optimality condition reads
\begin{equation}
	p^\alpha_\delta = - \frac{1}\alpha K^\ast \partial \fidelity{Ku}{\noisy} \in \partial \regfct(\regsol[\regparam]_\delta)
\end{equation}
we can use the debiasing procedure
\begin{equation}
	\regoparg[\regparam]{\noisy}= \argmin \{ \fidelity{Ku}{\noisy}~|   ~p^\alpha_\delta \in \partial \regfct(u)  \}.
\end{equation}
The condition $p^\alpha_\delta \in \partial \regfct(u) $ can be reformulated as a vanishing Bregman distance between $u$ and $u_\alpha^\delta$, thus we observe some relations to the Bregman iteration. The second step can be interpreted as a Bregman iteration step in the limit of the regularization parameter to infinity. We refer to \cite{brinkmann2017bias} for a detailed analysis of this debiasing approach.

The effect of the debiasing is illustrated for the simple case of total variation denoising, i.e. the solution of 
\begin{equation}
 \regoparg[\regparam]{\noisy} = \argmin_{u \in BV(\Omega)} \left( \frac{1}2 \Vert u - \noisy\Vert_{L^2(\Omega)}^2  + \alpha |u|_{BV}\right).
\end{equation}
Figure \ref{fig:debiasing1} compares the solution of the variational problem in (c) with the one obtained in the two-step debiasing procedure (d) and the Bregman iteration (e). Both methods reduce the contrast loss of the TV regularization (which is difficult to see in the image, but becomes more apparent in the small background buildings). Overall the Bregman iteration seems to restore more of the small details like the grass structure however. Figure \ref{fig:debiasing2} demonstrates the debiasing effect for increasing regularization parameter, where the variational model destroys more and more details. In particular for larger $\alpha$ one observes the effect of restoring smaller structures apparently contained in the subgradient but not the primal variable of the variational model.
\setlength{\tabcolsep}{0.2mm}
\renewcommand{\arraystretch}{0.1}
\begin{figure} 
 \begin{tabular}{B{0.04\textwidth}B{0.17\textwidth}B{0.17\textwidth}B{0.17\textwidth}B{0.17\textwidth}B{0.17\textwidth}}

   \rotatebox[origin=c]{90}{isotropic TV} &
   \includegraphics[width=0.17\textwidth]{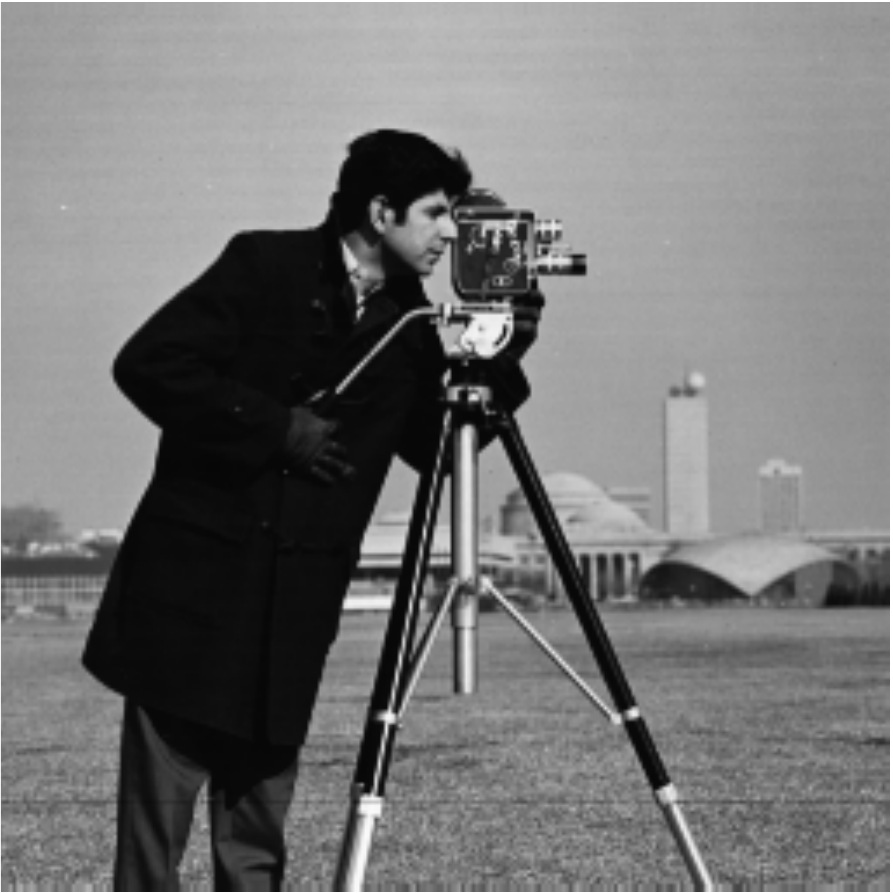} &
   \includegraphics[width=0.17\textwidth]{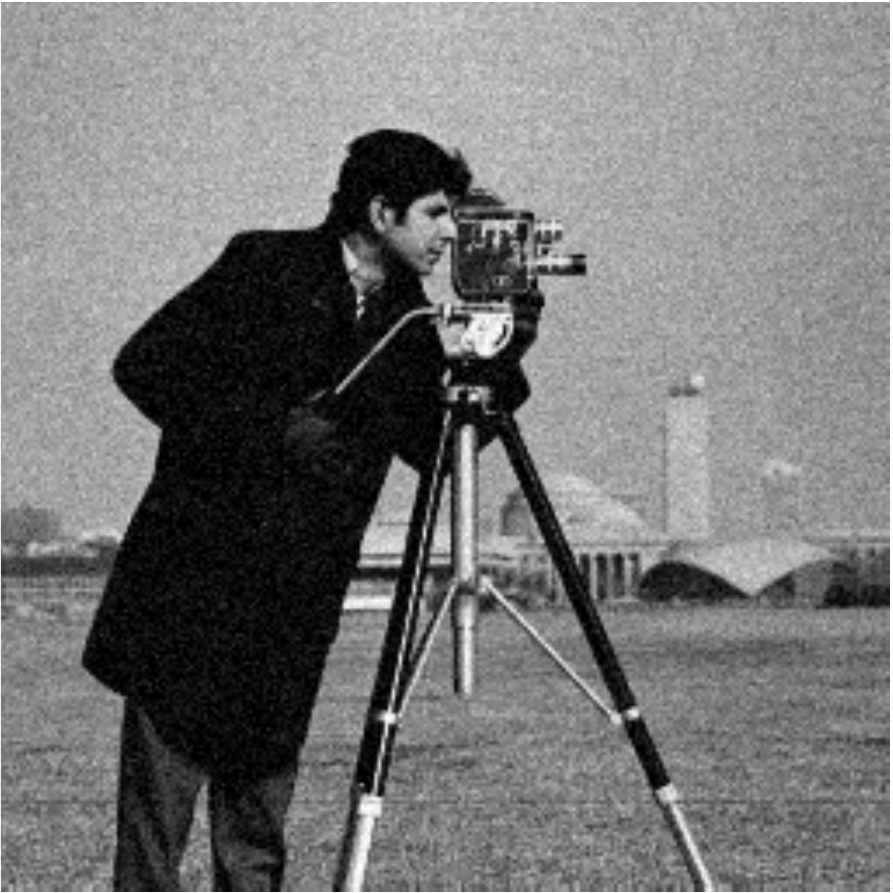} & 
   \includegraphics[width=0.17\textwidth]{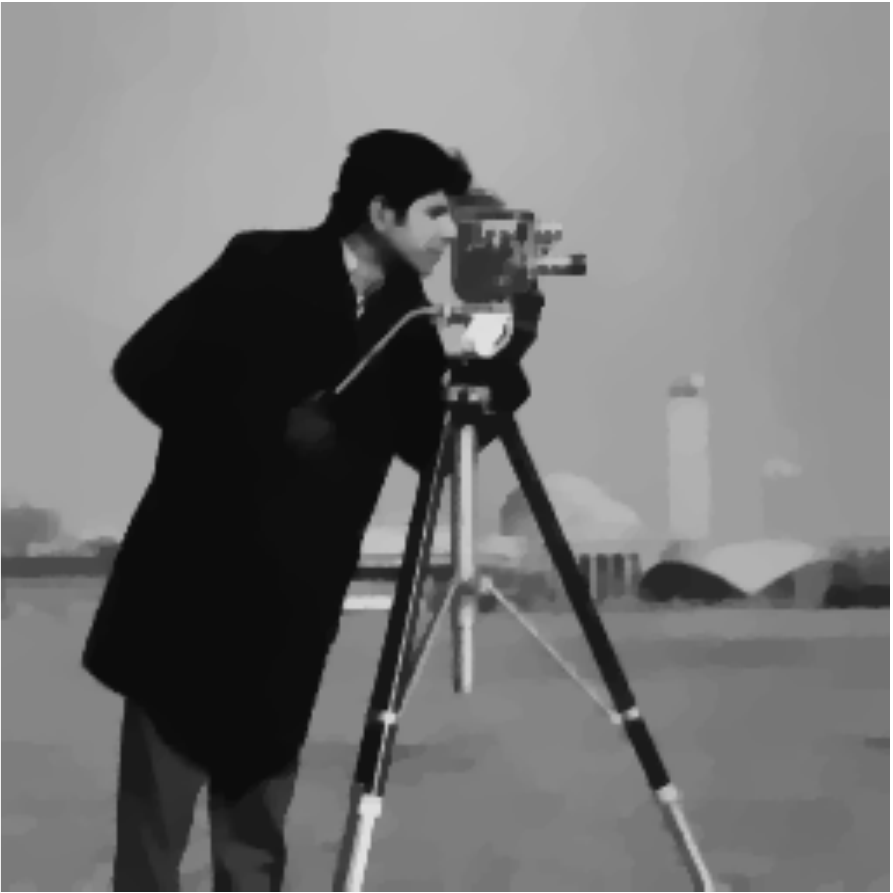} &
   \includegraphics[width=0.17\textwidth]{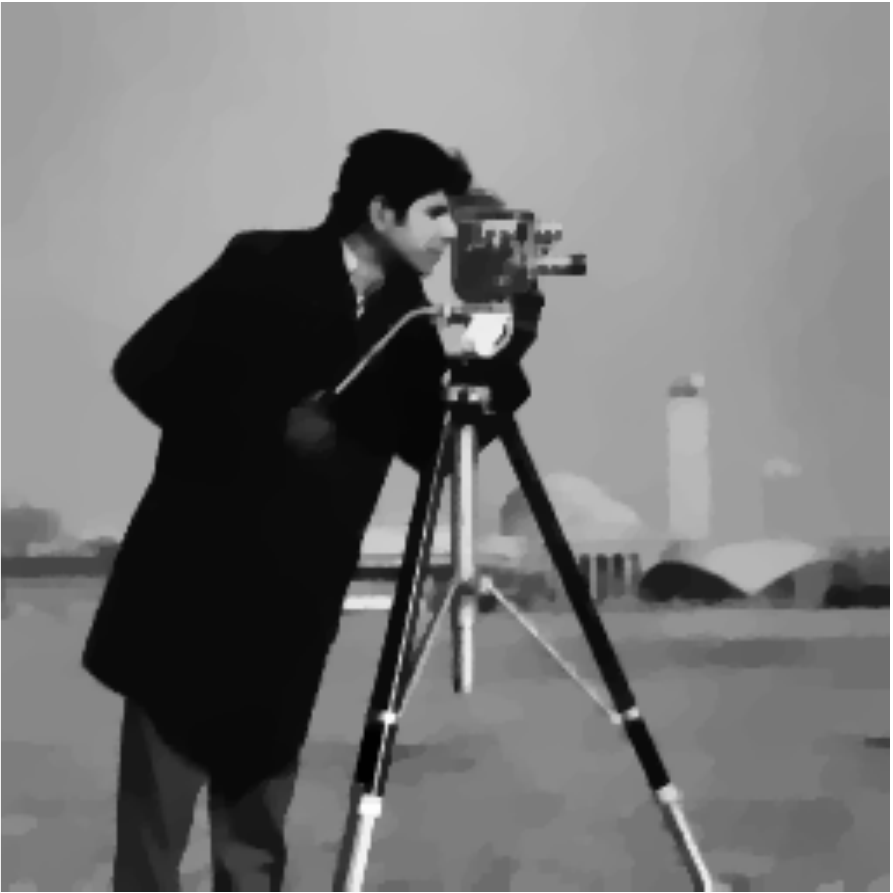} &
   \includegraphics[width=0.17\textwidth]{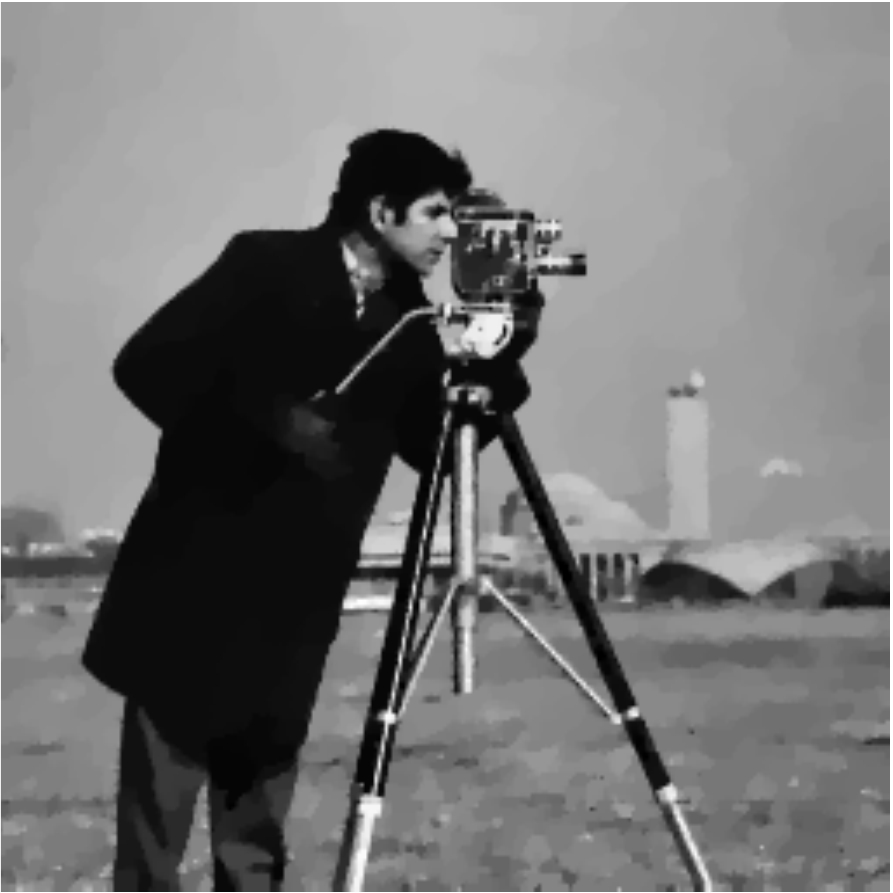}\\ \\[1mm]
   
   & (a) Original & (b) Noisy & (c) TV & (d) Debiasing & (e) Bregman its \\

\end{tabular} 
\caption{Camera man (256x256): Comparison of TV denoising for $\alpha = 0.1$, with the two-step debiasing, and Bregman iterations ($\alpha = 0.5$ and 7 Bregman iterations). \label{fig:debiasing1}}
\end{figure}
\begin{figure}
\begin{tabular}{B{0.04\textwidth}B{0.3\textwidth}B{0.3\textwidth}B{0.3\textwidth}}

   \rotatebox[origin=c]{90}{TV} &
   \includegraphics[width=0.3\textwidth]{Images/cameraman_tv_1} &
   \includegraphics[width=0.3\textwidth]{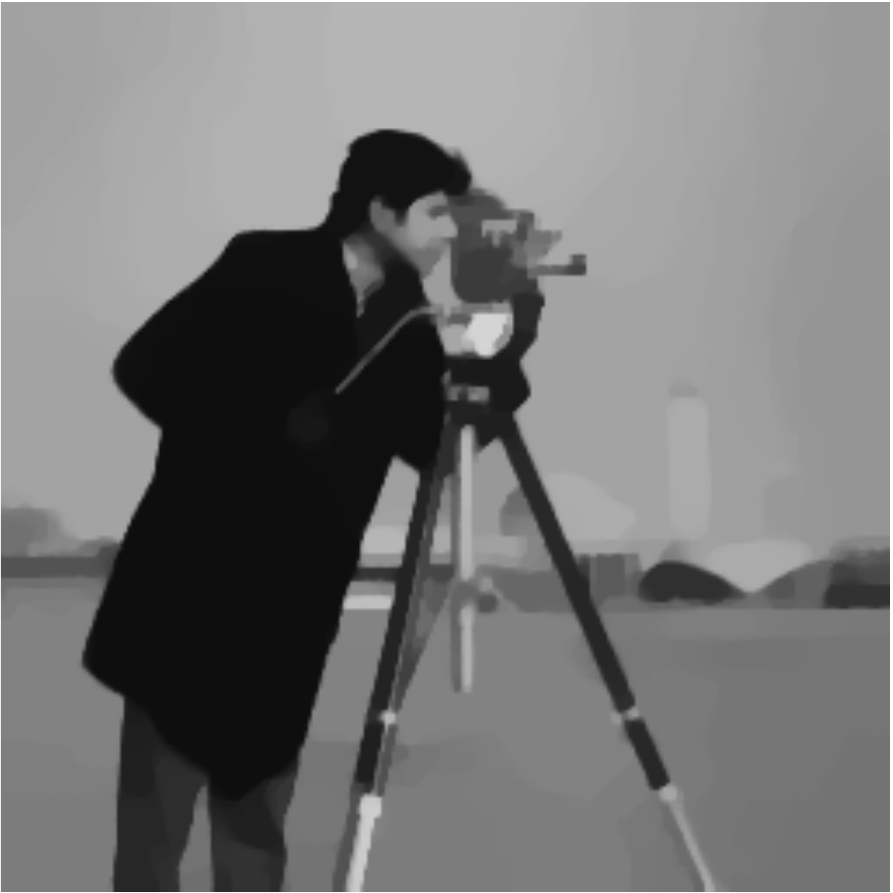} & 
   \includegraphics[width=0.3\textwidth]{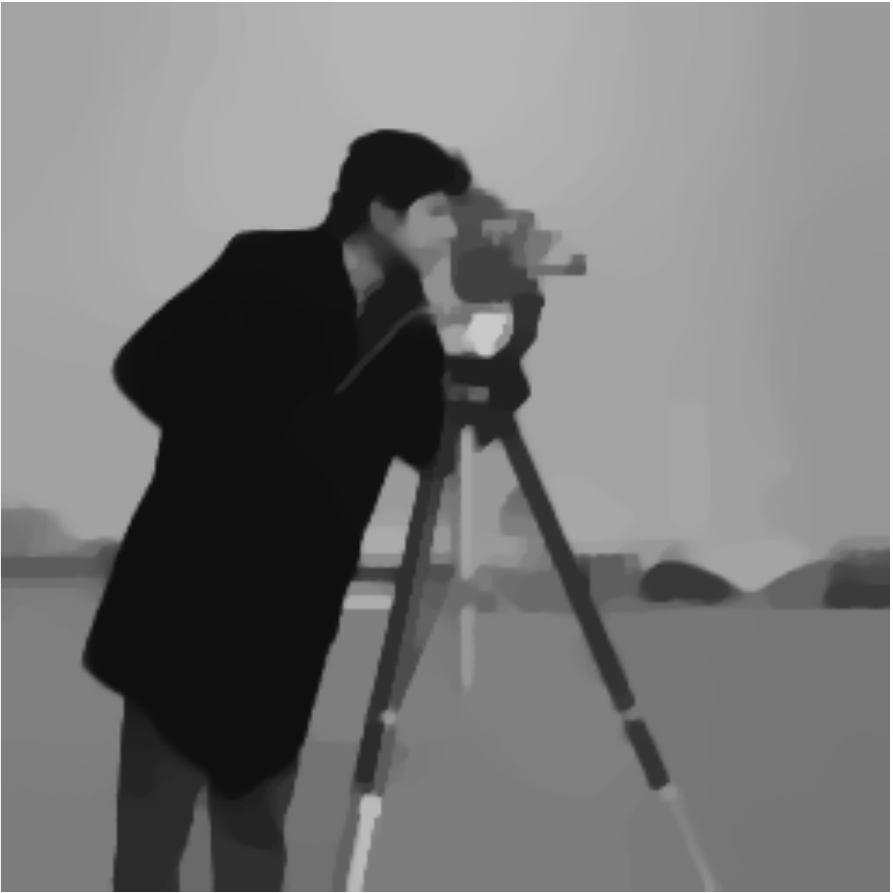} \\
   
   \rotatebox[origin=c]{90}{Debiasing} &
   \includegraphics[width=0.3\textwidth]{Images/cameraman_debiased_1} &
   \includegraphics[width=0.3\textwidth]{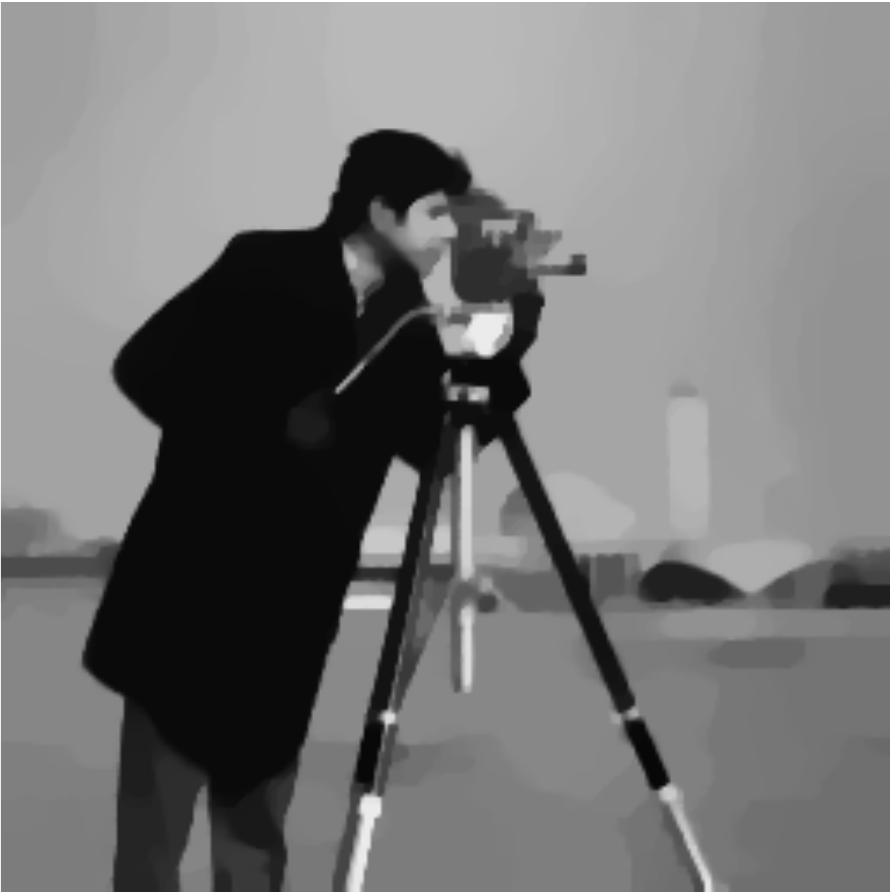} & 
   \includegraphics[width=0.3\textwidth]{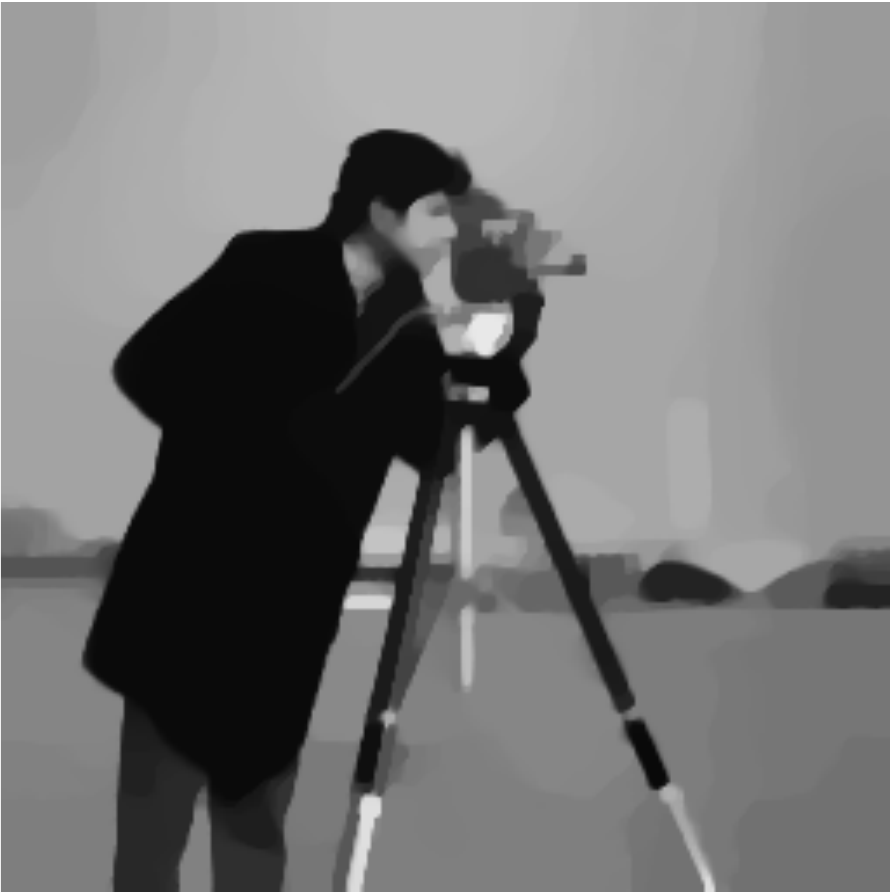} \\ \\[1mm]
   
   & (a) $\alpha = 0.1$ & (b) $\alpha = 0.2$& (c) $\alpha = 0.3$ \\
\end{tabular}
\caption{Camera man TV denoising and debiasing for different values of the regularization parameter. \label{fig:debiasing2}}
\end{figure}

\subsection{Nonlinear Spectral Transform}\label{subsec:nonlinspectrafo}
The iterative regularization methods presented in Section \ref{sec:iterreg} can easily be extended to nonlinear spectral decomposition methods via the following trivial observation. Every iterate $u^k \in \regopitarg{\noisy, v^{k - 1}}$ can be represented as the sum of the differences of two subsequent iterates, i.e.
\begin{align*}
u^k = u^0 + \sum_{n = 1}^n u^n - u^{n - 1} \, . 
\end{align*}
If we define $\varphi^0 := u^0$ and $\varphi^n := u^n - u^{n - 1}$ for $n > 1$, and equip the sum with coefficients $\{ c^n \}_{n = 0}^k$, we can write $u^k$ as
\begin{align*}
u^k = \sum_{n = 0}^k c^n \varphi^n \, . 
\end{align*}
In the following we are going to motivate why such a decomposition is useful for localizing individual scales if the underlying regularization functional is (absolutely) one-homogeneous and where we have a scalar parameter $\regparam$. Following up on the bias correction example for generalized singular vectors in Section \ref{subsubsec:debiaseigenfcts}, we know that for the Bregman iteration $\regopitarg[\regparam]{\clean, v^{k - 1}}$ with $\clean = v_\sigma$ we observe
\begin{align*}
u^k = \begin{cases}
0 & k < k^\ast \\ \frac{k^\ast - \regparam}{\sigma} u_\sigma & k = k^\ast\\ \frac{1}{\sigma} u_\sigma & k \geq k^\ast + 1
\end{cases} \, .
\end{align*}
Replacing $\clean = v_\sigma$ with $\clean = \sigma v_\sigma = Ku_\sigma$ therefore yields  
\begin{align*}
u^k = \begin{cases}
0 & k < k^\ast \\ \left( k^\ast - \frac{\regparam}{\sigma} \right) u_\sigma & k = k^\ast\\ u_\sigma & k \geq k^\ast + 1
\end{cases} \, ,
\end{align*}
and consequently we observe
\begin{align*}
\varphi^n = \begin{cases}
0 & n \not\in \{ k^\ast, k^\ast + 1\} \\ \left( k^\ast - \frac{\regparam}{\sigma} \right) u_\sigma & n = k^\ast\\ \left( 1 + \frac{\regparam}{\sigma} - k^\ast \right) u_\sigma & n =k^\ast + 1
\end{cases} \, .
\end{align*}
The last equation implies that if the input datum is given in terms of the forward model applied to a (primal) singular vector, this primal singular vector is localized in only two components $\varphi^{k^\ast}$ and $\varphi^{k^\ast + 1}$. The index $k^\ast$ depends on the choice of $\regparam$ and on the singular value $\sigma$. Hence, singular vectors with different scales, respectively different values of $\sigma$, will be localized in $\varphi^{\hat{k}}$ and $\varphi^{\hat{k} + 1}$ for $\hat{k} \neq k^\ast$. This is visualized in Figure \ref{fig:singvecspectrum} and Figure \ref{fig:bee}. It is therefore fair to call $\{ \varphi^n \}_{n = 1}^k$ a spectrum and the individual $\varphi^n$, for $n \in \{1, \ldots, k\}$, the spectral components.

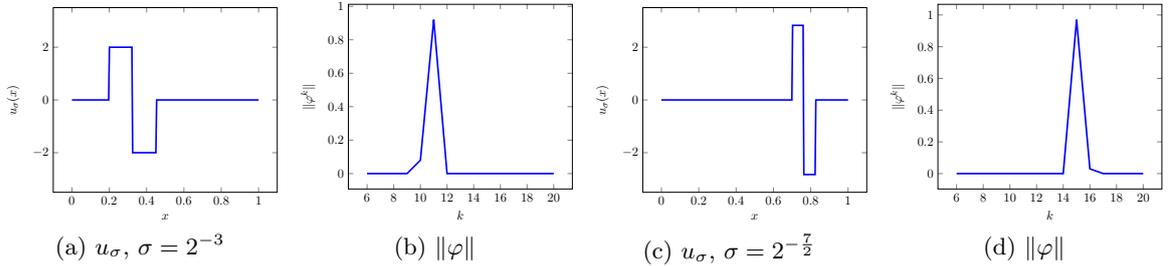
\begin{figure}[!t]
\begin{center}
\subfloat[$u_\sigma$, $\sigma = 2^{-3}$]{
\begin{tikzpicture}[scale=0.43,
declare function={ 
    haar(\x) = and(\x >= 0, \x < 1/2) * (+1) + and(\x >= 1/2, \x < 1) * (-1);
}
]
\begin{axis}[%
            domain = 0:1,
            samples = 250,
			ymin = -3.5,
			ymax = 3.5,
            xlabel = {$x$},
            ylabel = {$u_\sigma(x)$},
            ]
            \addplot[line width=0.5mm, blue] {(2^(2/2))*haar((2^2)*(x - 0.2))};                       
        \end{axis}
\end{tikzpicture}
}
\subfloat[$\| \varphi \|$]{
\begin{tikzpicture}[scale=0.43]
\begin{axis}[%
            domain = 6:17,
            samples = 9,
            xlabel = {$k$},
            ylabel = {$\| \varphi^k \|$},
            ]
            \addplot[line width=0.5mm, blue, forget plot]
            table[row sep=crcr]{
6 0\\
7 0\\ 
8 0\\
9 0\\ 			            
10 0.08\\
11 0.92\\
12 0\\
13 0\\
14 0\\
15 0\\
16 0\\
17 0\\
18 0\\
19 0\\
20 0\\
            }; 
        \end{axis}
\end{tikzpicture}
}
\subfloat[$u_\sigma$, $\sigma = 2^{-\frac{7}{2}}$]{
\begin{tikzpicture}[scale=0.43,
declare function={ 
    haar(\x) = and(\x >= 0, \x < 1/2) * (+1) + and(\x >= 1/2, \x < 1) * (-1);
}
]
\begin{axis}[%
            domain = 0:1,
            samples = 250,
			ymin = -3.5,
			ymax = 3.5,
            xlabel = {$x$},
            ylabel = {$u_\sigma(x)$},
            ]
            \addplot[line width=0.5mm, blue] {(2^(3/2))*haar((2^3)*(x - 0.7))}; 
        \end{axis}
\end{tikzpicture}
}
\subfloat[$\| \varphi \|$]{
\begin{tikzpicture}[scale=0.43]
\begin{axis}[%
            domain = 6:17,
            samples = 9,
            xlabel = {$k$},
            ylabel = {$\| \varphi^k \|$},
            ]
            \addplot[line width=0.5mm, blue, forget plot]
            table[row sep=crcr]{
6 0\\
7 0\\ 
8 0\\
9 0\\ 			            
10 0\\
11 0\\
12 0\\
13 0\\
14 0\\
15 0.97\\
16 0.03\\
17 0\\
18 0\\
19 0\\
20 0\\
            };
        \end{axis}
\end{tikzpicture}
}
\end{center}
\caption{We see two singular vectors of $\regfct = \tvast$ with different $\sigma$-values, and excerpts of their corresponding (analytically computed) spectra, for $\regparam = 1.24$. We clearly observe that both vectors are located at different positions of the spectrum. Hence, both singular vectors could be isolated from a sum of the two by applying a band-pass filter to the spectrum.}\label{fig:singvecspectrum}
\end{figure}

\noindent Consequently, the operator $\mathcal{S}:\range \times \domain^k \times \R^k \times \regdomain \rightarrow \domain$ with 
\begin{align*}
\mathcal{S}(\clean, (u_n)_{n=0}^k, (c_n)_{n = 0}^k, \regparam) := \sum_{n = 0}^k c_n \varphi^n \quad \text{with} \quad \varphi^n := \begin{cases}
u^n - u^{n - 1} & n > 1\\
u^0 & n = 1
\end{cases} \, ,
\end{align*}
for $u^n \in \regoparg[\regparam]{\clean, v^{n - 1}}$ can be seen as a spectral transform of the input signal $\noisy$. For $K:\bv(\Omega) \rightarrow L^2(\Omega)$ this type of spectral transform is a discretization of the inverse-scale-space based spectral transform defined in \cite{burger2015spectral,burger2016spectral}. For $K:\bv(\Omega) \rightarrow L^2(\Omega)$ and $\regfctarg[\regparam]{u} = \regparam \tv(u)$, the idea of generalized spectral transforms goes back to \cite{gilboa2014nonlinear,gilboa2014total}. For a detailed overview on this form of nonlinear spectral transform we refer to \cite{gilboa2016nonlinear}. Another interesting recent extension is the spectral transform in the context of image segmentation \cite{zeune2017multiscale}.

\begin{figure}[!t]
\begin{center}
\includegraphics[width=\textwidth]{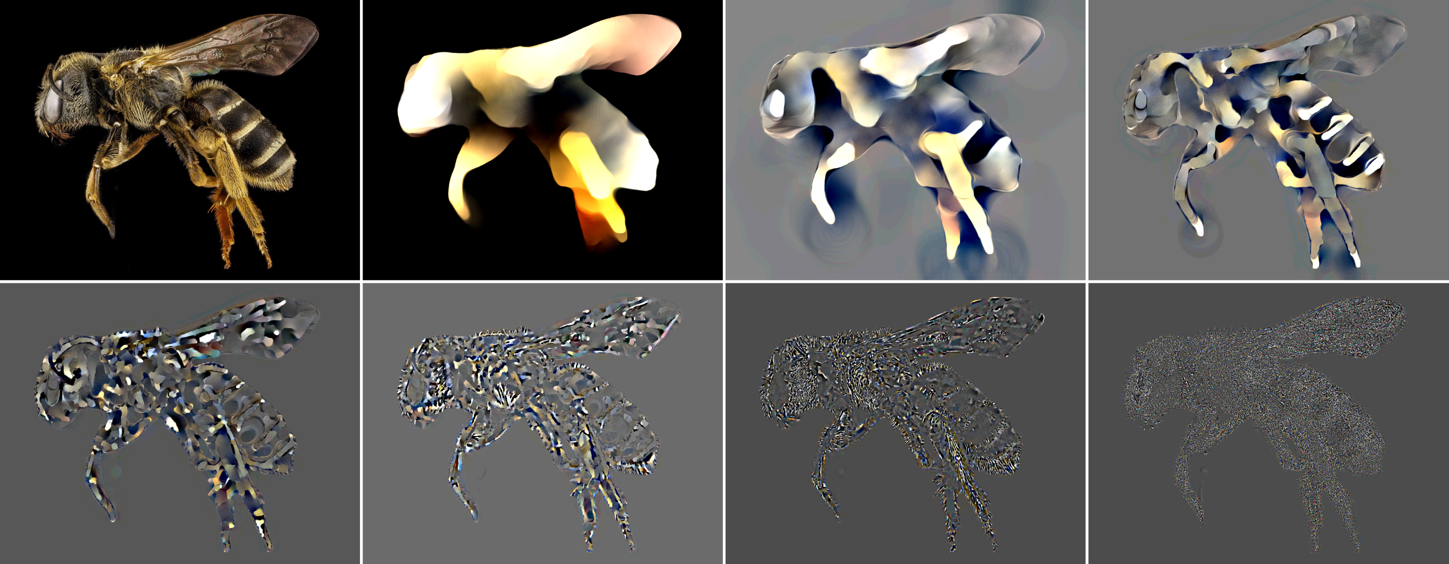}
\end{center}
\caption{Spectral decomposition of the image of a bee. From \cite{Benning2017}.}\label{fig:bee}
\end{figure}


\section{Applications}

Obviously modern regularization methods have found applications in all kind of inverse problems and pushed forward the state of the art. As some examples let us mention TV/TGV Bregman iterations for superresolution (cf. \cite{marquina2008image}), PET reconstruction (cf. \cite{muller2011reconstruction,muller2013advanced}) or STED microscopy (cf. \cite{BSB09a,primaldualbregman}), as well as TGV reconstructions in MR (cf. \cite{knoll2011second}). Providing an overview of the various approaches for well-known imaging modalities would by far exceed the scope and size of this survey. Hence, in the following we provide some novel examples of applications, which are actually driven by advances in regularization techniques.



\subsection{Velocity-Encoded Magnetic Resonance Imaging}\label{sec:velmri}

Magnetic Resonance Imaging (MRI) is an imaging technique that allows to visualize the chemical composition of humans/animals or materials. MRI scanners utilize strong magnetic fields and radio waves to excite subatomic particles such as protons that subsequently emit radio frequency signals which can be measured with the radio frequency coils that initially excited those radio waves, see for example \cite{paul1993principles}. MRI is often used to measure contrast in tissue. However, due to shear, endless possibilities of radio-frequency pulse sequence design and programming of the gradient coils, MRI is a versatile imaging tool with capabilities beyond imaging contrast in tissue. A potential, more sophisticated application is phase-encoded magnetic resonance velocity imaging, which in medical imaging is used to study the distribution and variation in blood flow \cite{gatehouse2005applications}. In the physical sciences, it is being used to study the rheology of complex fluids \cite{callaghan1999rheo}, liquids and gases flowing through packed beds \cite{sederman1998structure,holland2010reducing}, granular flows \cite{holland2008spatially} and multi-phase turbulence flows \cite{tayler2012exploring}. The main advantage of MRI over other modalities when it comes to studying flow is that it is possible to image flows non-invasively. However, the main drawback of the technique is the acquisition time of the measurement. 

\begin{figure}[!h]
\begin{center}
\includegraphics[width=0.3\textwidth]{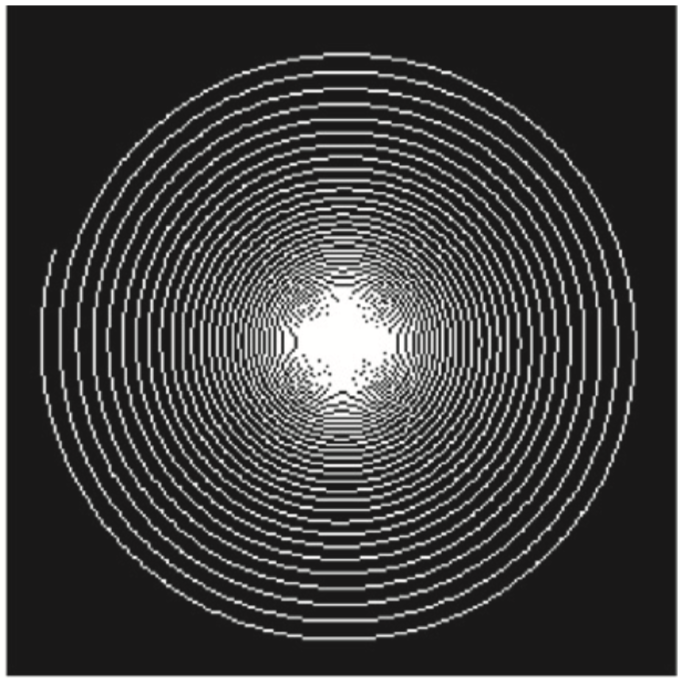}
\end{center}
\caption{A simulated spiral on a cartesian grid. From \cite{benning2014phase}.}\label{fig:spiral}
\end{figure}

In \cite{lustig2007sparse} the idea of sub-sampling in the spatial data domain has been exploited to overcome this limitation and to speed up the MRI acquisition process. Due to fewer measurements compared to unknowns that are being recovered, some form of regularization needs to be integrated into the reconstruction process. Sparsity-promoting variational regularization methods are suitable candidates and most prominently, total variation regularization has successfully been deployed to increase the temporal resolution of MRI acquisitions. Since measurement noise in MRI data can be modeled as being normally distributed, a standard variational regularization approach is
\begin{align}
\regoparg{\noisy} = \argmin_{u \in \domain} \left\{ \frac{1}{2} \| \mathcal{F}u - \noisy \|_2^2 + \regfctarg{u} \right\} \, ,\label{eq:varregmri}
\end{align}
where $\mathcal{F}$ is the operator
\begin{align*}
(\mathcal{F}u)(t^k) := (2\pi)^{-\frac{n}{2}} \int_{\R^n} u(x) \exp\left(-i \int_{t^{k - 1}}^{t^k} x(t) \cdot g(t) \, dt \right) \, dx \, ,
\end{align*}
and $n \in \{2, 3\}$ denotes the dimension of the signal and $g:[0, T] \rightarrow \R^n$ represents the function that controls the gradient coils of the MRI machine. We observe that $\mathcal{F}$ is almost identical to the Fourier transform sampled at discrete locations, if we can approximate $\int_{t^{k - 1}}^{t^k} x(t) \cdot g(t) dt \approx x \cdot \int_{t^{k - 1}}^{t^k} g(t) \, dt$. This can be achieved by adequate programming of the gradient coils. However, $\int_{t^{k - 1}}^{t^k} x(t) \cdot g(t) dt$ can be approximated more generally via the Taylor series
\begin{align*}
\int_{t^{k - 1}}^{t^k} x(t) \cdot g(t) dt \approx \sum_{r = 0}^\infty \frac{x^{(r)}(t^{k - 1})}{r !} \cdot \int_{t^{k - 1}}^{t^k} g(t) t^r \, dt \, , 
\end{align*}
and with clever programming of $g$, other moments such as velocity or acceleration can be encoded. In the following, we assume that the radio-frequency pulse sequence and the gradient coils are programmed such that we first encode the velocity in the z-direction, i.e. for $x = (x_1, x_2, x_3)$ and $g(t) = (g_1(t), g_2(t), g_3(t))$ we have
\begin{align*}
\int_{0}^{t_0} x(t) \cdot g(t) dt \approx \underbrace{x_3^\prime(0)}_{=: v_z} \int_{0}^{t_0} g_3(t) \, t \, dt \, ,
\end{align*}
in the interval $[0, t_0]$, and then perform the spatial encoding such that
\begin{align*}
\int_{t_{k - 1}}^{t^k} x(t) \cdot g(t) dt \approx \left( \begin{array}{c}
x_1(t_{k - 1}) \\ x_2(t_{k - 1}) 
\end{array} \right) \cdot \int_{t^{k - 1}}^{t^k} \left( \begin{array}{c}
g_1(t) \\ g_2(t) 
\end{array} \right) \, dt
\end{align*}
holds true for $t_0 < t_1 < \ldots < t_m = T$. Then, with $x = (x_1(t_{k - 1}), x_2(t_{k - 1}))$ and $g = (g_1, g_2)$ as an abuse of notation, $\mathcal{F}$ reads as
\begin{align}
(\mathcal{F}(u, v_z))(t^k) = \frac{1}{2\pi} \int_{\R^2} u(x) \exp(- i \sigma v_z(x)) \exp\left(-i x \cdot \int_{t^{k - 1}}^{t^k} g(t) \, dt \right) \, dx \, ,\label{eq:velmriforward}
\end{align}
for some constant $\sigma$. In order to avoid non-linearity of the forward model, we couple $u$ and $v_z$ by simply defining $w := u \exp(-i \sigma v_z)$. Then the forward model $\mathcal{F}$ simply reduces to the (sub-sampled) Fourier transform.

\begin{figure}[!t]
\begin{center}
\subfloat[Fully-sampled $u$]{\includegraphics[trim={3cm 1cm 2.2cm 1cm},clip,width=0.32\textwidth]{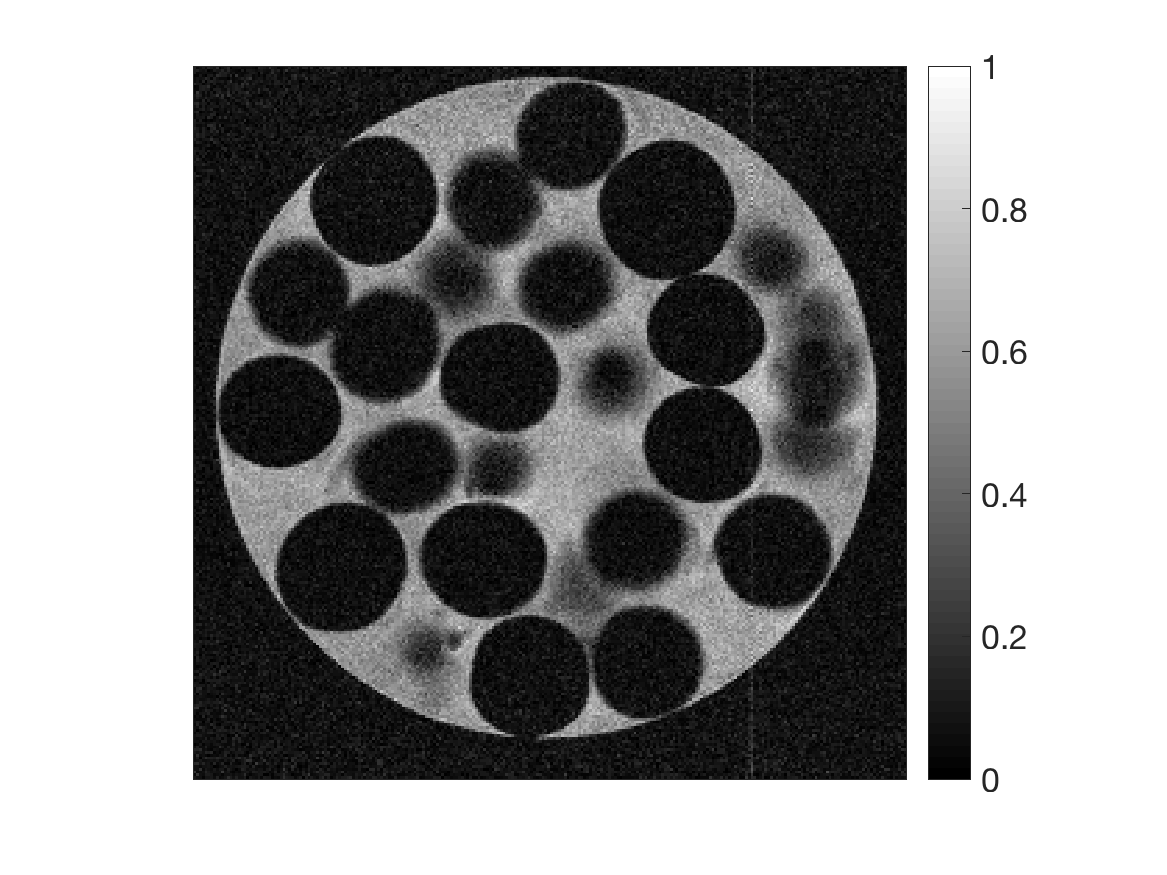}\label{subfig:velmri11}}
\subfloat[Zero-filled $u$]{\includegraphics[trim={3cm 1cm 2.2cm 1cm},clip,width=0.32\textwidth]{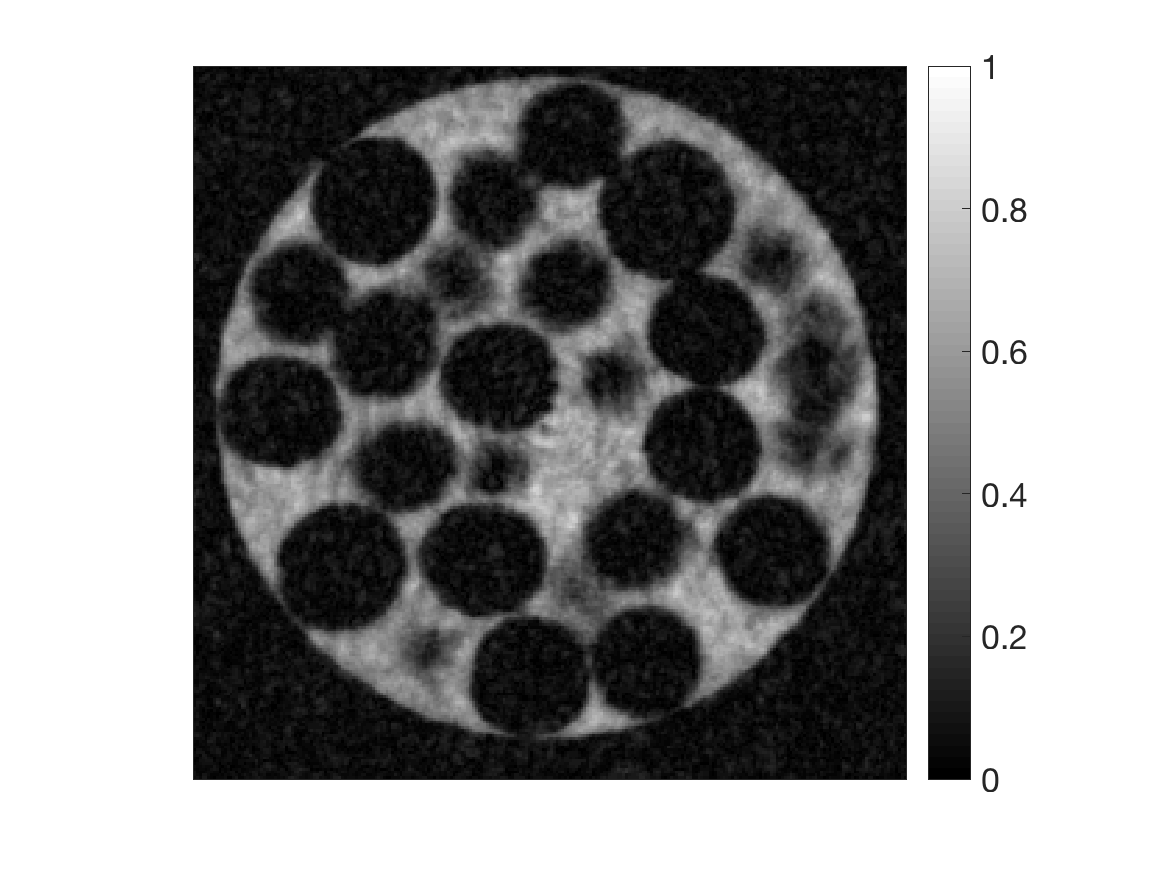}\label{subfig:velmri12}}
\subfloat[TGV-based $u$]{\includegraphics[trim={3cm 1cm 2.2cm 1cm},clip,width=0.32\textwidth]{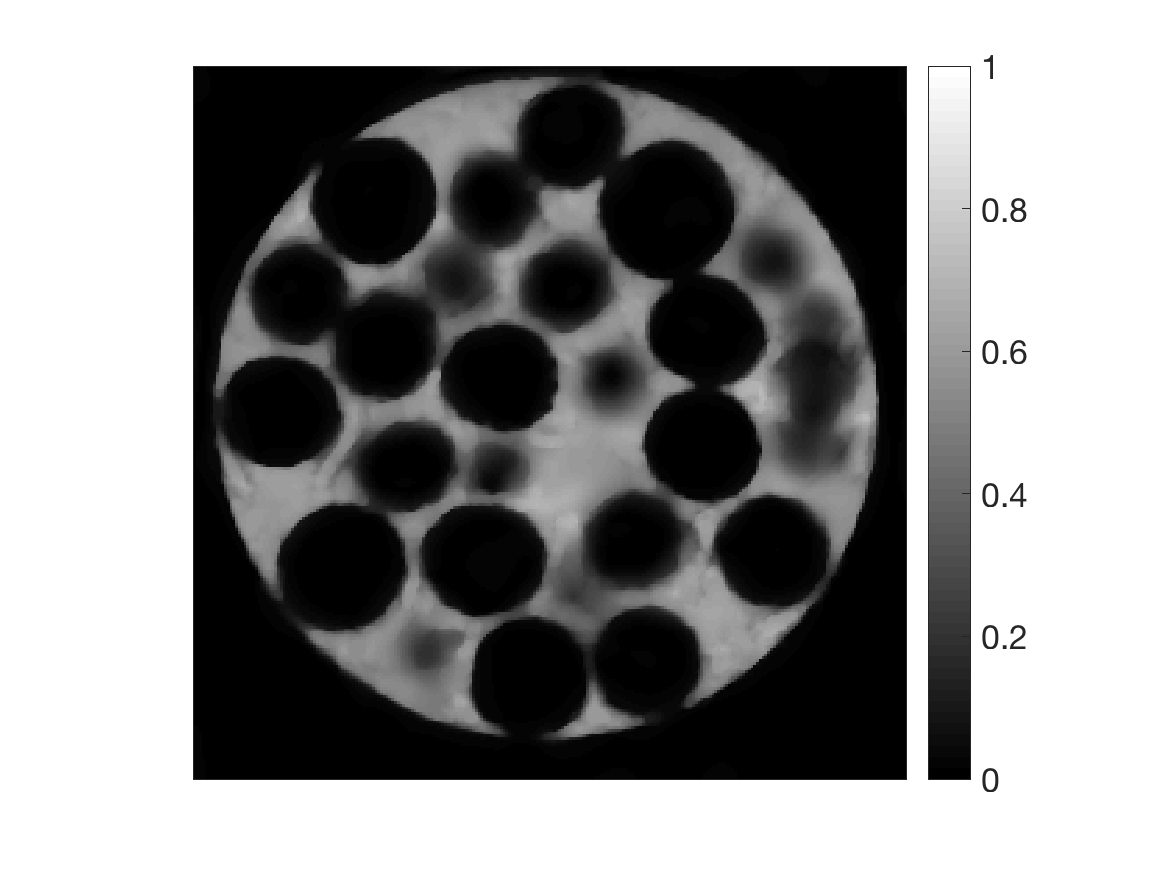}\label{subfig:velmri13}}\\
\end{center}
\caption{Magnitude images of the velocity dataset used in \cite{benning2014phase}, courtesy of Andrew J. Sederman. Figure \ref{subfig:velmri11} shows the magnitude image derived from applying the inverse of the Fourier transform to the fully-sampled Fourier data of the velocity-encoded MRI measurement and subsequently taking the modulus. In Figure \ref{subfig:velmri12} we see the magnitude image that we obtain if we set all Fourier samples to zero that are not part of the spiral visualized in Figure \ref{fig:spiral}, and subsequently proceed as with the fully-sampled data. Finally, Figure \ref{subfig:velmri13} shows the magnitude reconstructions from the $\text{TGV}_\beta^2$-based variational regularization reconstruction \eqref{eq:varregmri}.}
\label{fig:velmri1}
\end{figure}

\noindent In \cite{benning2014phase}, three choices for regularization functionals have been investigated: assuming $\regparambold = (\regparam, \beta)$, we have
\begin{align}
\regfctarg[\alpha, \beta]{u} = \alpha \begin{cases} \text{TV}(u) \\
\text{TGV}_\beta^2(u) \\
\sum_{j = 1}^\infty |\langle u, \varphi_j \rangle |
\end{cases} \, . \label{eq:mriregfcts}
\end{align}
Here $\{ \varphi_j \}_{j \in \Z}$ denotes a wavelet basis. In Figure \ref{fig:velmri1} we see computational solutions of \eqref{eq:varregmri} for the choice $\regfctarg[\regparam, \beta]{u} = \text{TGV}_\beta^2(u)$, a spiral sub-sampling strategy on a cartesian grid, see Figure \ref{fig:spiral}, and the parameter choices $\regparam = 0.1$ and $\beta = 3$. Those results have again been computed with the PDHGM. Subsequently, $v_z$ has been extracted as the principle value of the reconstruction $w \in \regoparg[\regparam, \beta]{\noisy}$. The reconstructed z-velocity $v_z$ is subsequently unwrapped via by solving the linear system
\begin{align*}
\Delta \hat{v}_z = \cos(v_z) \Delta \sin(v_z) - \sin(v_z) \Delta \cos(v_z) 
\end{align*}
for $\hat{v}_z$. Here $\Delta$ denotes the Laplace operator. The unwrapped reconstructed velocity $\hat{v}_z$ is visualized in Figure \ref{fig:velmri2}.

\begin{figure}[!ht]
\begin{center}
\subfloat[Fully-sampled]{\includegraphics[trim={3cm 1cm 2.2cm 1cm},clip,width=0.32\textwidth]{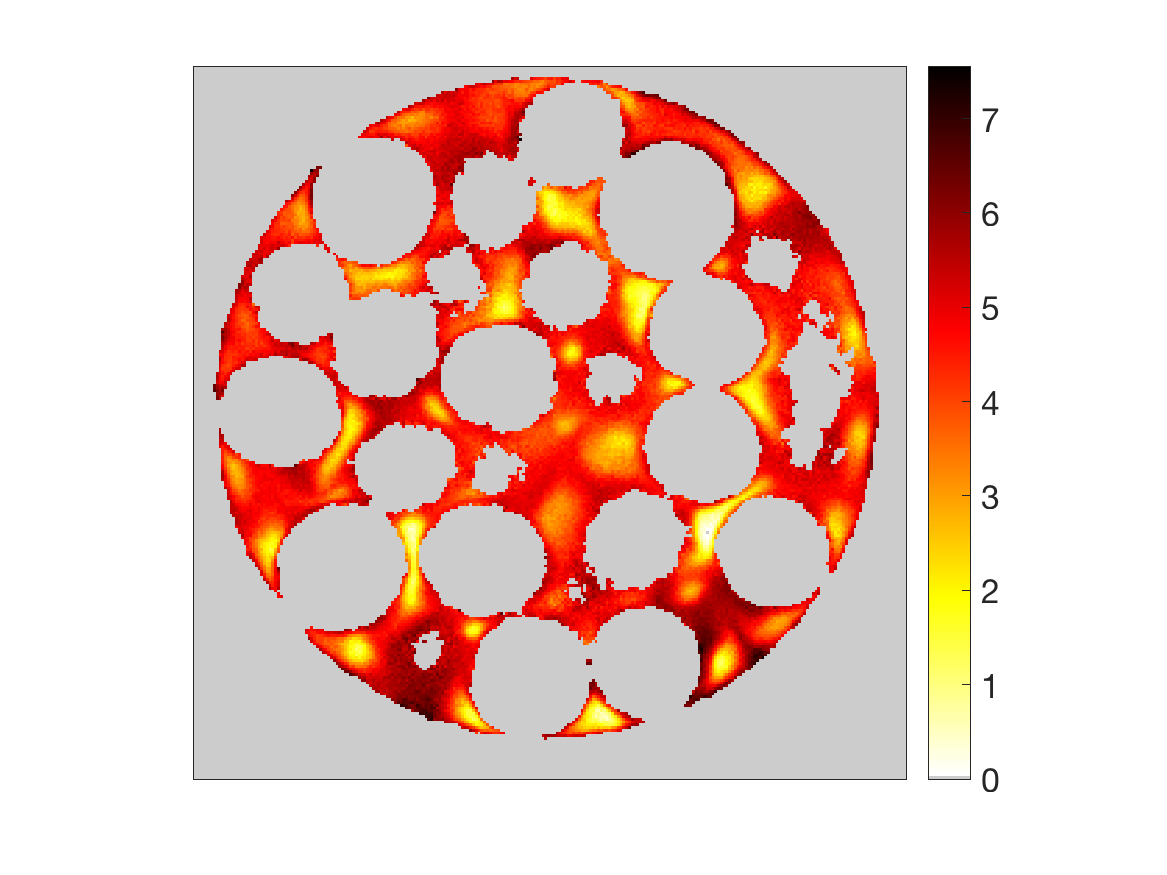}}
\subfloat[Zero-filled]{\includegraphics[trim={3cm 1cm 2.2cm 1cm},clip,width=0.32\textwidth]{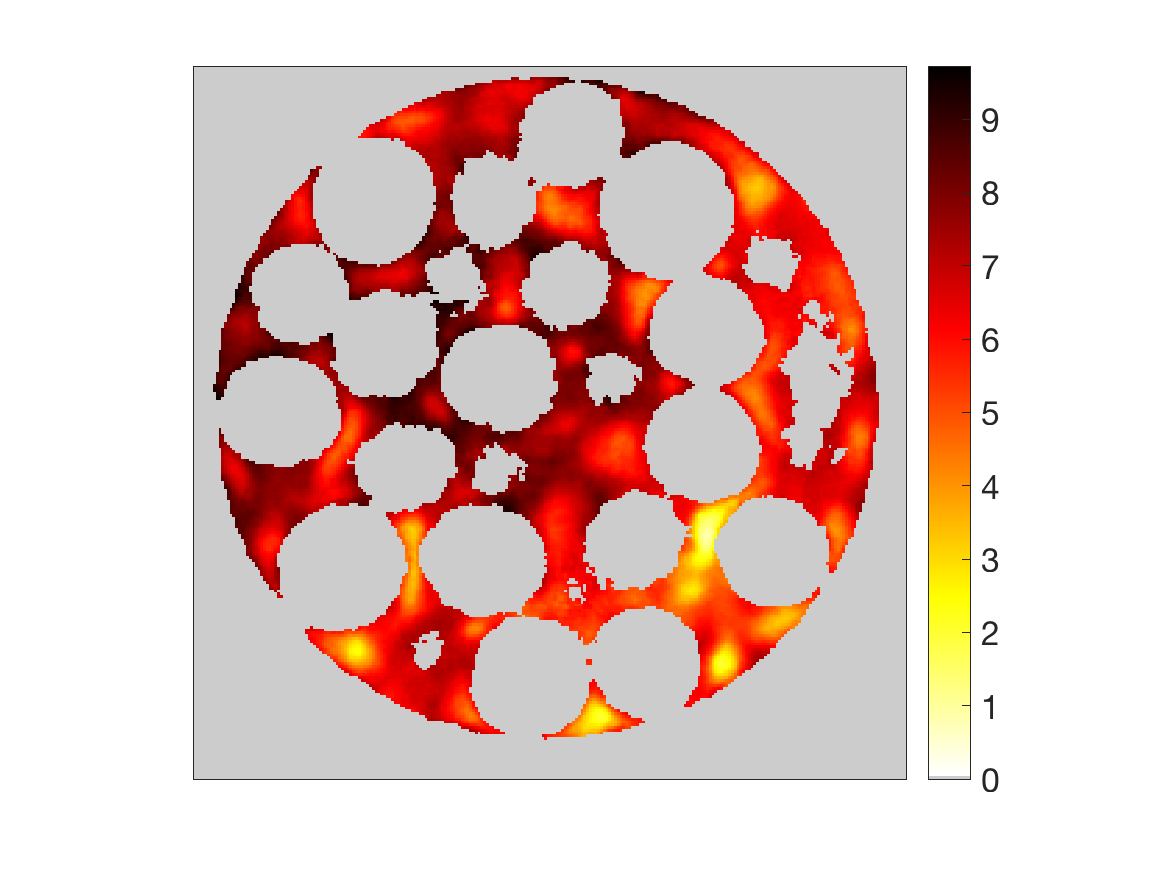}}
\subfloat[TGV-based]{\includegraphics[trim={3cm 1cm 2.2cm 1cm},clip,width=0.32\textwidth]{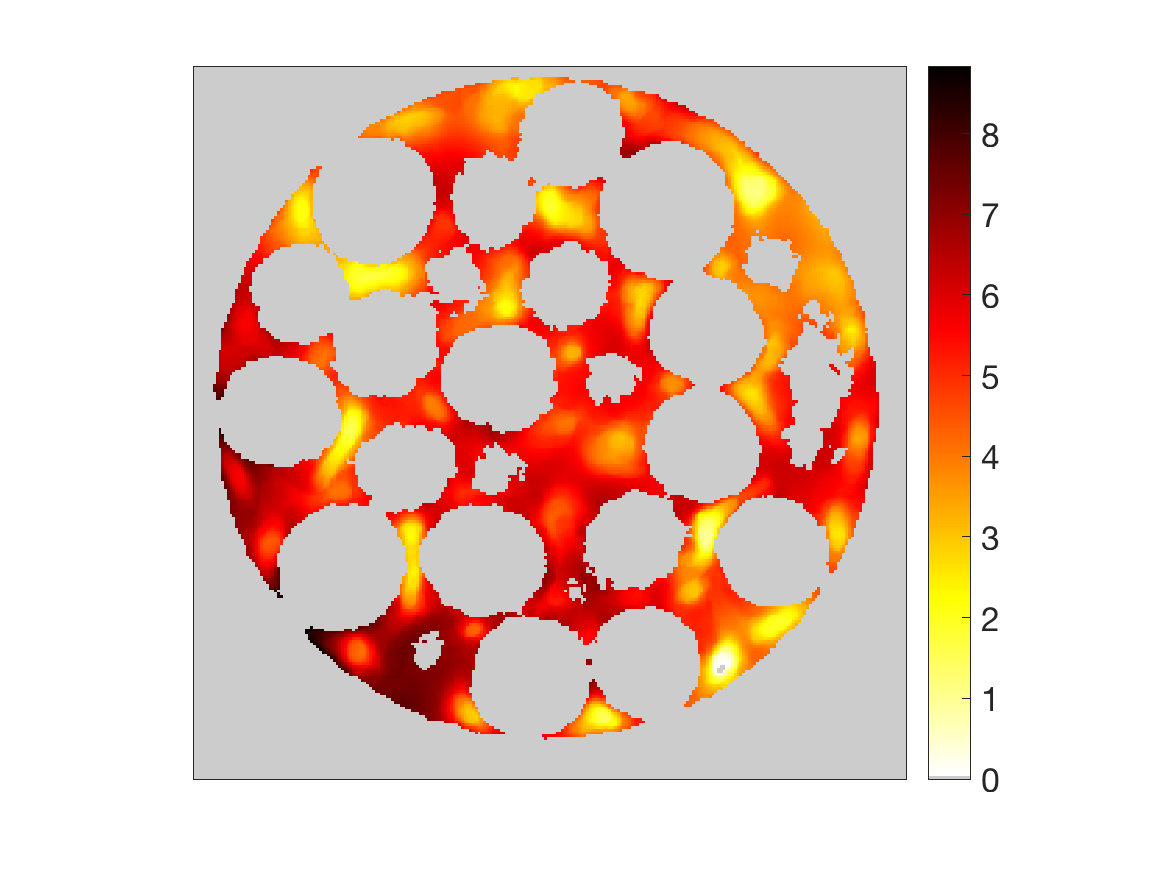}}
\end{center}
\caption{We see the different velocity-reconstructions that correspond to the magnitude reconstructions in Figure \ref{fig:velmri1}. }
\label{fig:velmri2}
\end{figure}

In order to demonstrate the capabilities of the Bregman iteration, Algorithm \eqref{alg:bregiter} has been qualitatively analyzed for different sub-sampling strategies and different initial choices of $\alpha > 0$ in \cite{benning2014phase}. These comparisons for different sub-sampling strategies are visualized in Figure \ref{fig:quantmri}. In Figure \ref{fig:velmribregiter} we see the magnitude images of 20 Bregman iterations computed with Algorithm \ref{alg:bregiter} for the same setup as described earlier, and the parameter choices $\regparam = 1.5$ and $\beta = 3$.

\begin{figure}[!t]
\begin{center}
\includegraphics[width=\textwidth]{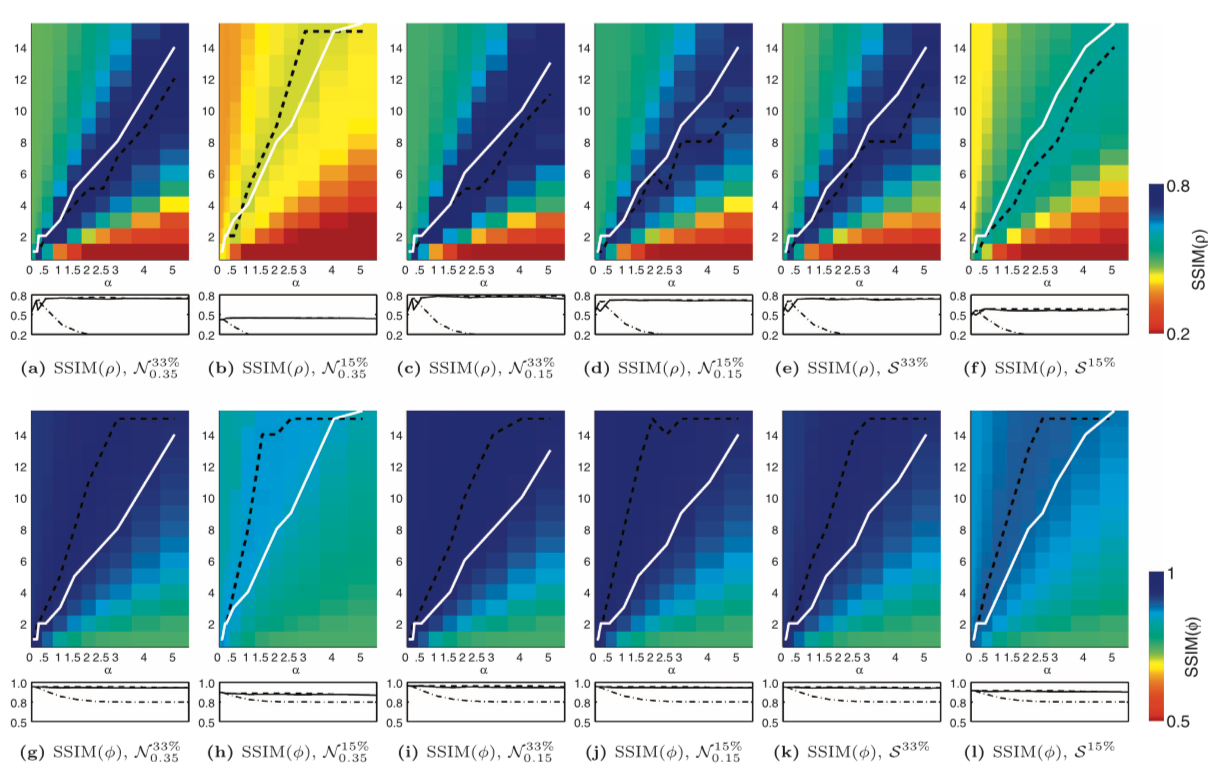}
\end{center}
\caption{The structural similarity index measure (SSIM) (see \cite{wang2004image}) of the magnitude images (top row) and the velocity images (bottom row) for Bregmanized TV reconstructions of computer-generated test data with various sampling patterns and noise $\sigma = 0.2$. The parameter $\regparam$ is on the horizontal and the Bregman iteration on the vertical axis. The colors code the SSIM value, also shown in the small lower graph. The continuous line corresponds to violation of the discrepancy principle, and the dashed line to the optimal SSIM. The dash-dotted line in the small graph indicates the SSIM for the first iteration. From \cite{benning2014phase}.}
\label{fig:quantmri}
\end{figure}

\begin{figure}[p]
\begin{center}
\subfloat[Iterate 1]{\includegraphics[trim={3cm 1cm 2.2cm 1cm},width=0.24\textwidth]{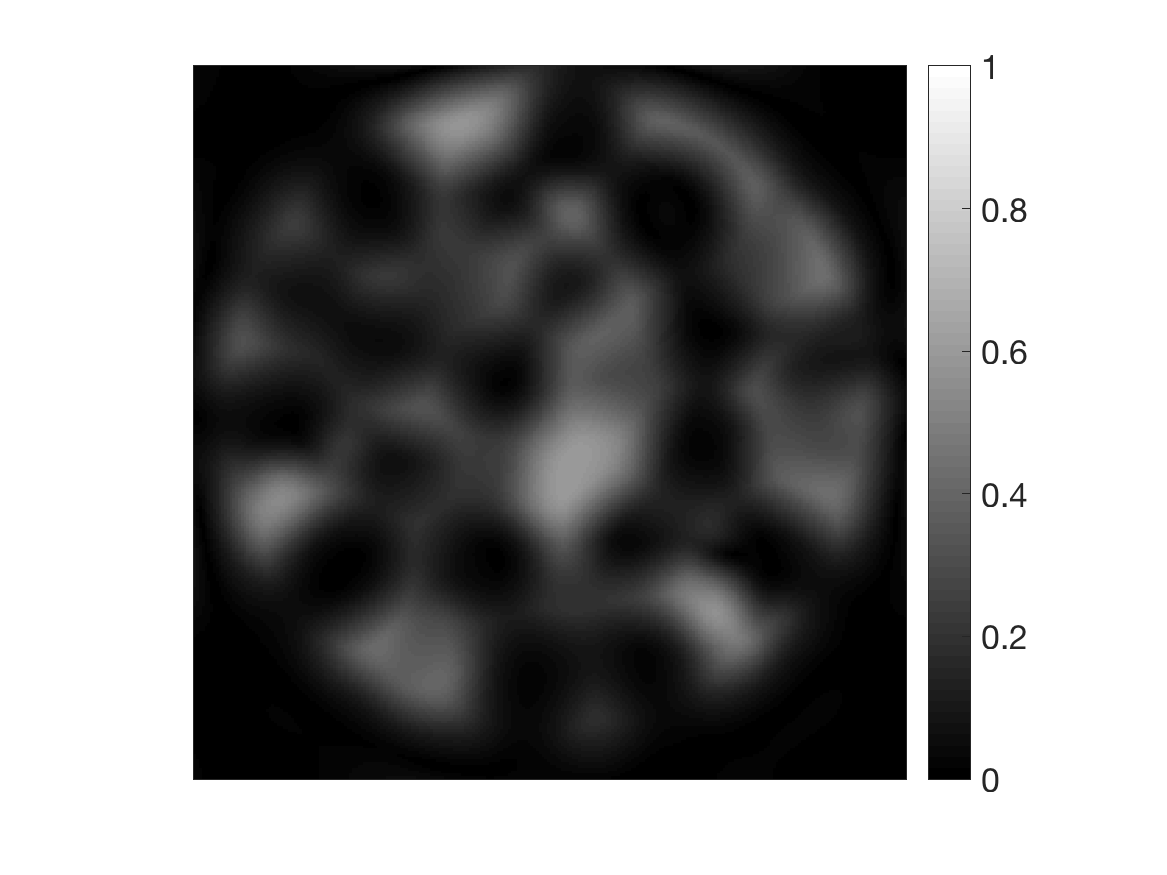}}\vspace{0.05cm}
\subfloat[Iterate 2]{\includegraphics[trim={3cm 1cm 2.2cm 1cm},width=0.24\textwidth]{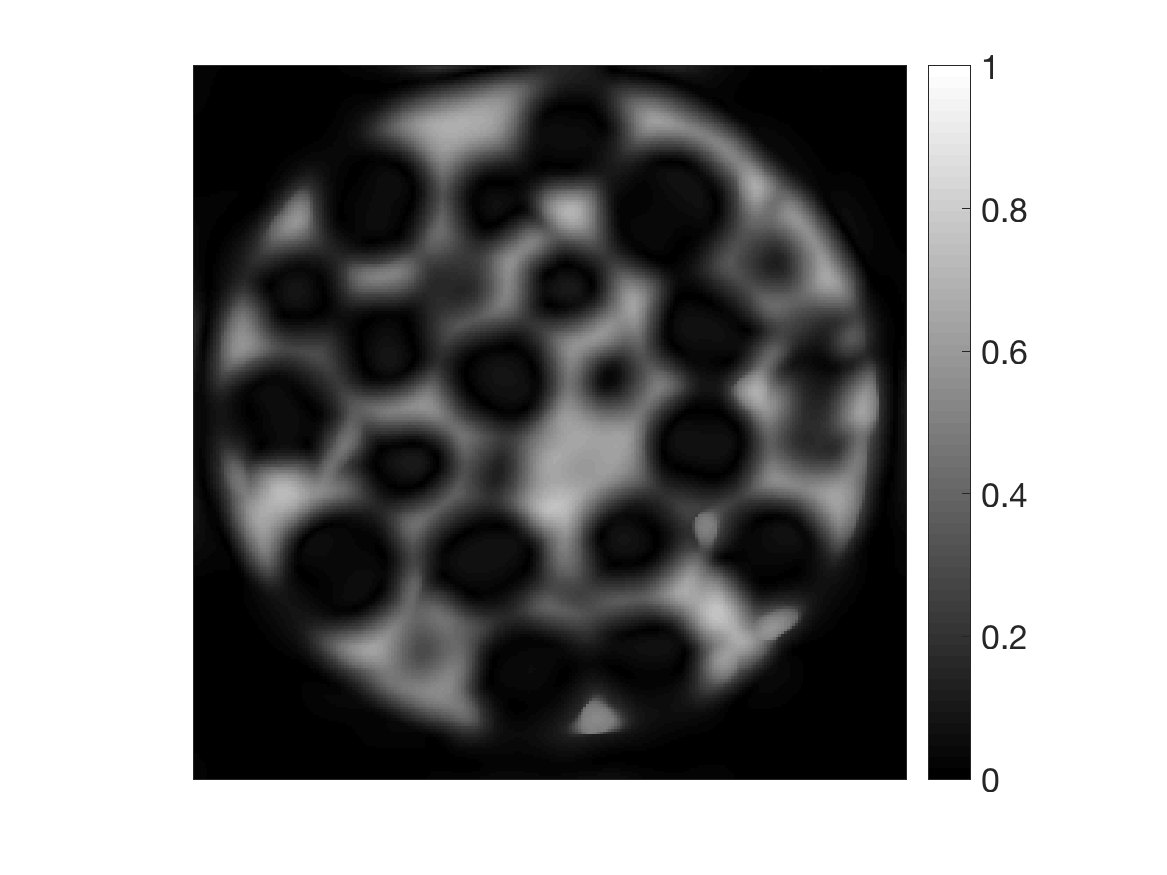}}\vspace{0.05cm}
\subfloat[Iterate 3]{\includegraphics[trim={3cm 1cm 2.2cm 1cm},width=0.24\textwidth]{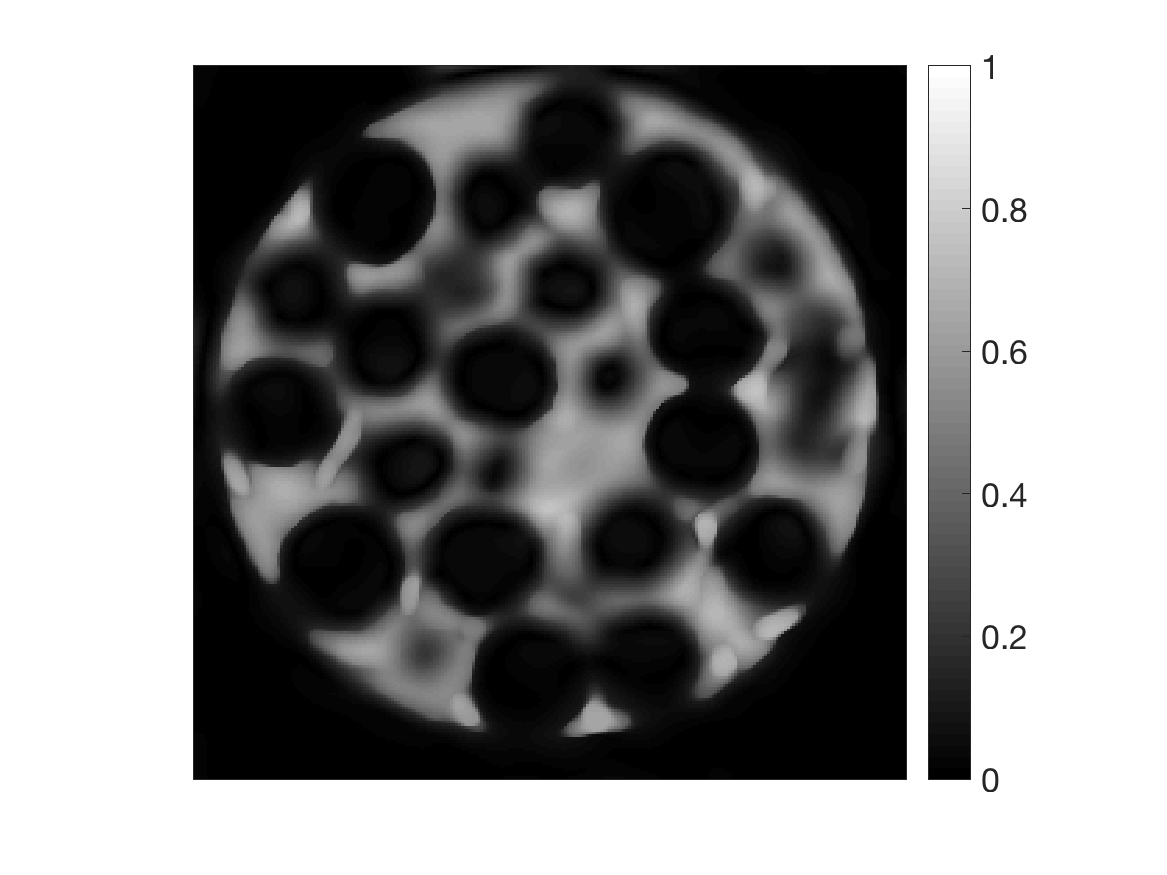}}\vspace{0.05cm}
\subfloat[Iterate 4]{\includegraphics[trim={3cm 1cm 2.2cm 1cm},width=0.24\textwidth]{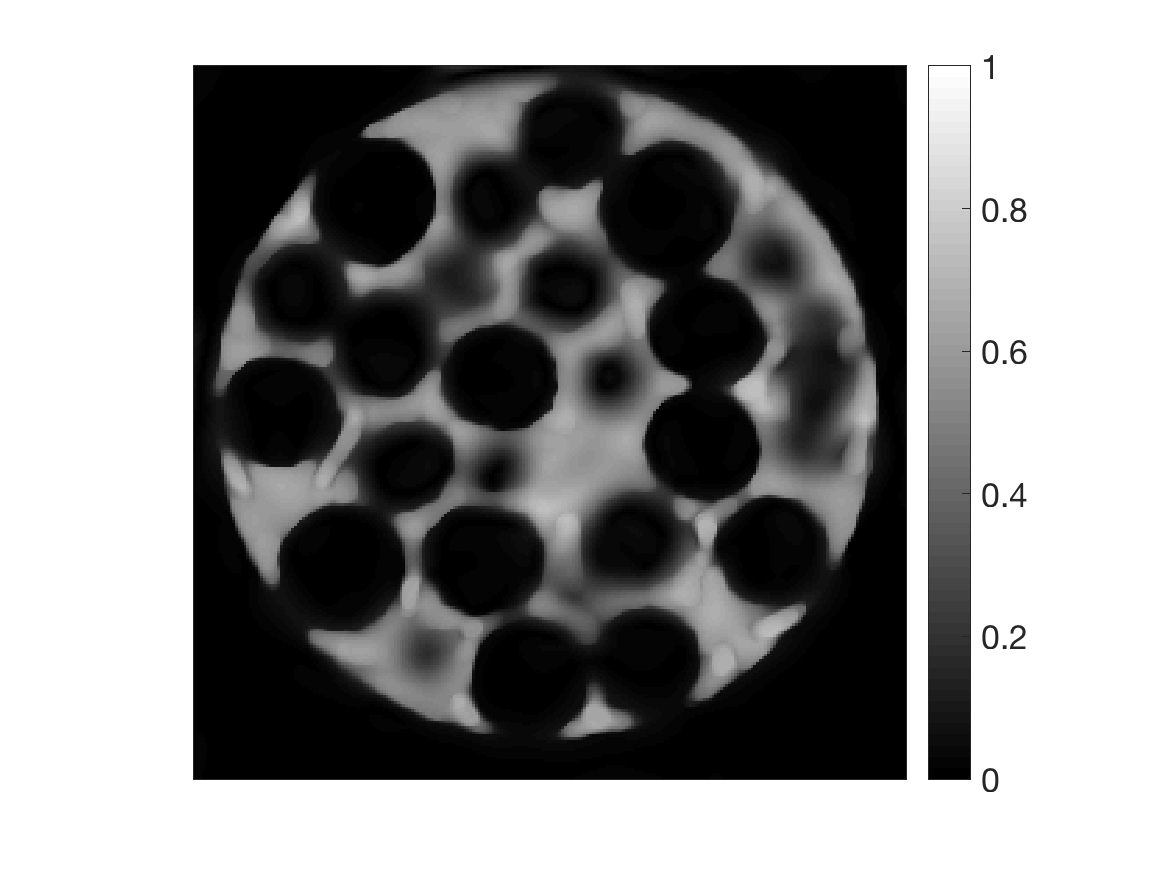}}\\
\subfloat[Iterate 5]{\includegraphics[trim={3cm 1cm 2.2cm 1cm},width=0.24\textwidth]{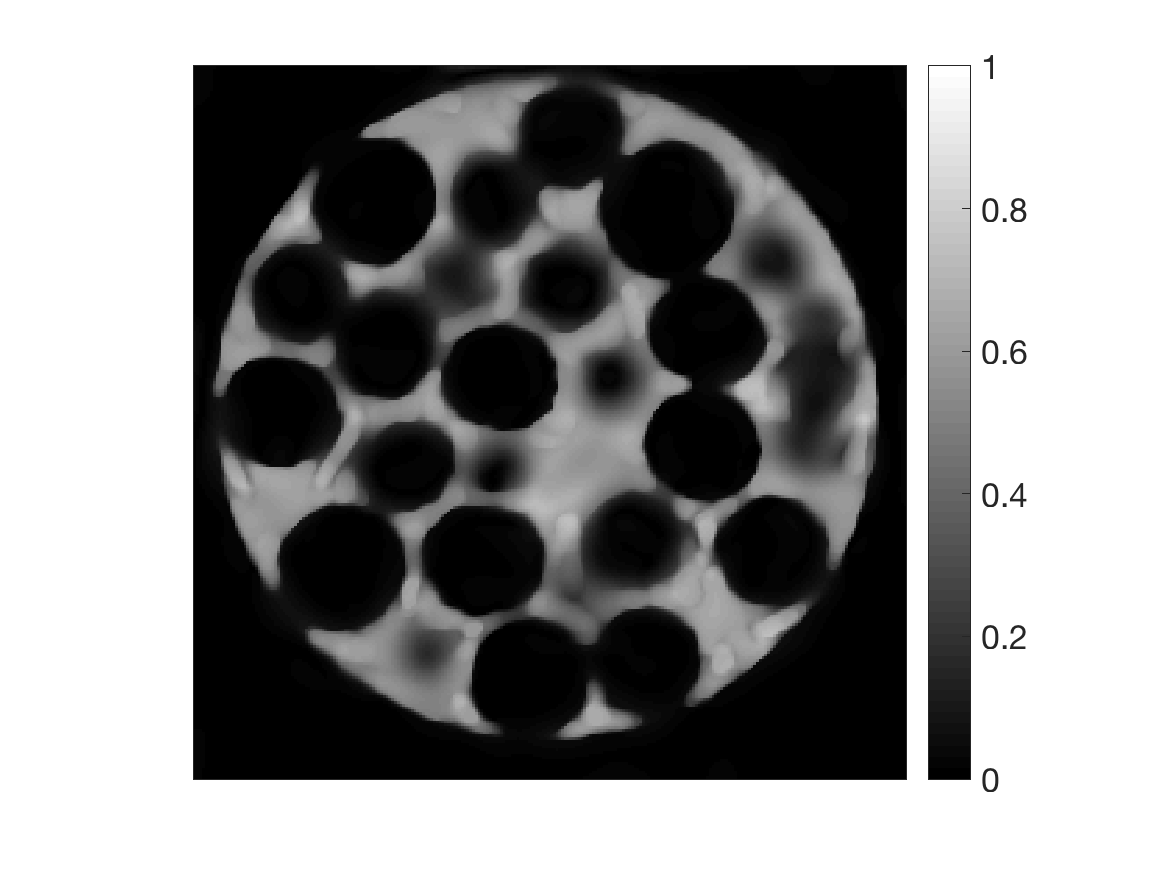}}\vspace{0.05cm}
\subfloat[Iterate 6]{\includegraphics[trim={3cm 1cm 2.2cm 1cm},width=0.24\textwidth]{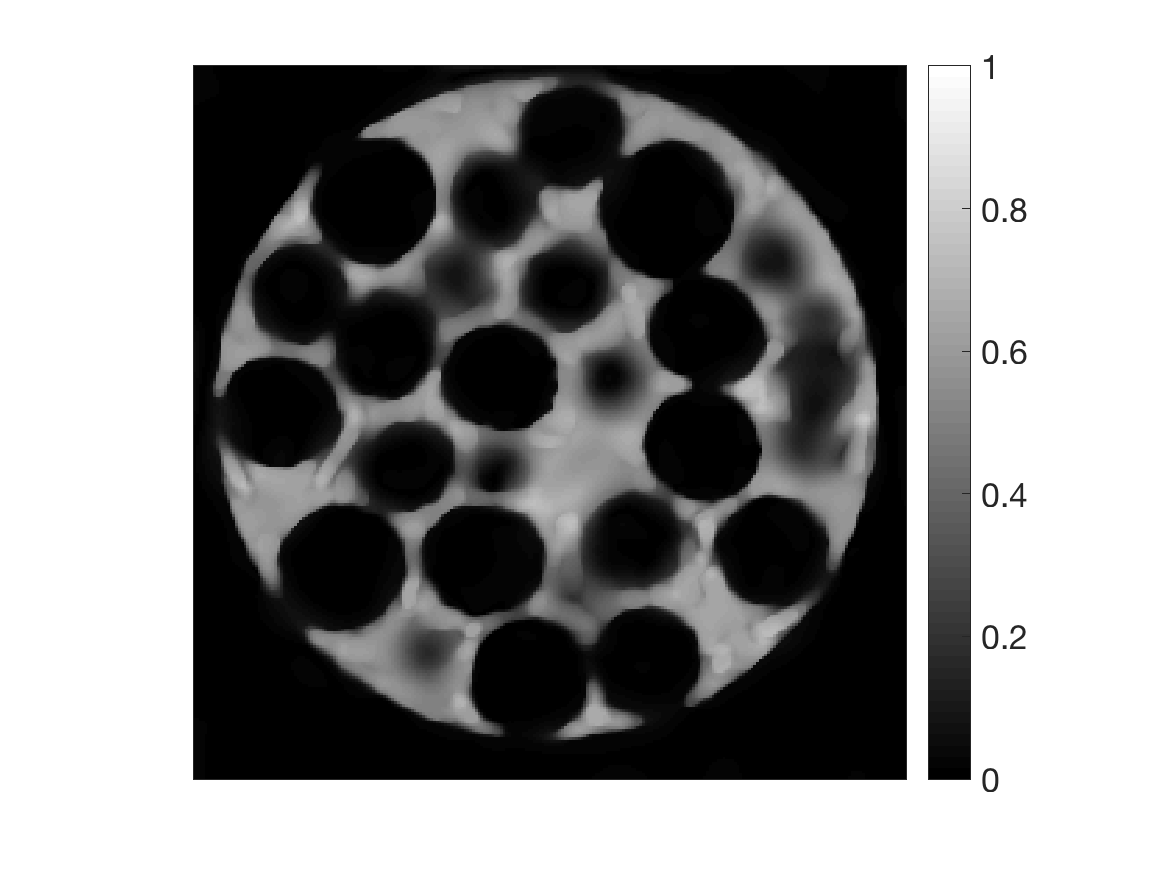}}\vspace{0.05cm}
\subfloat[Iterate 7]{\includegraphics[trim={3cm 1cm 2.2cm 1cm},width=0.24\textwidth]{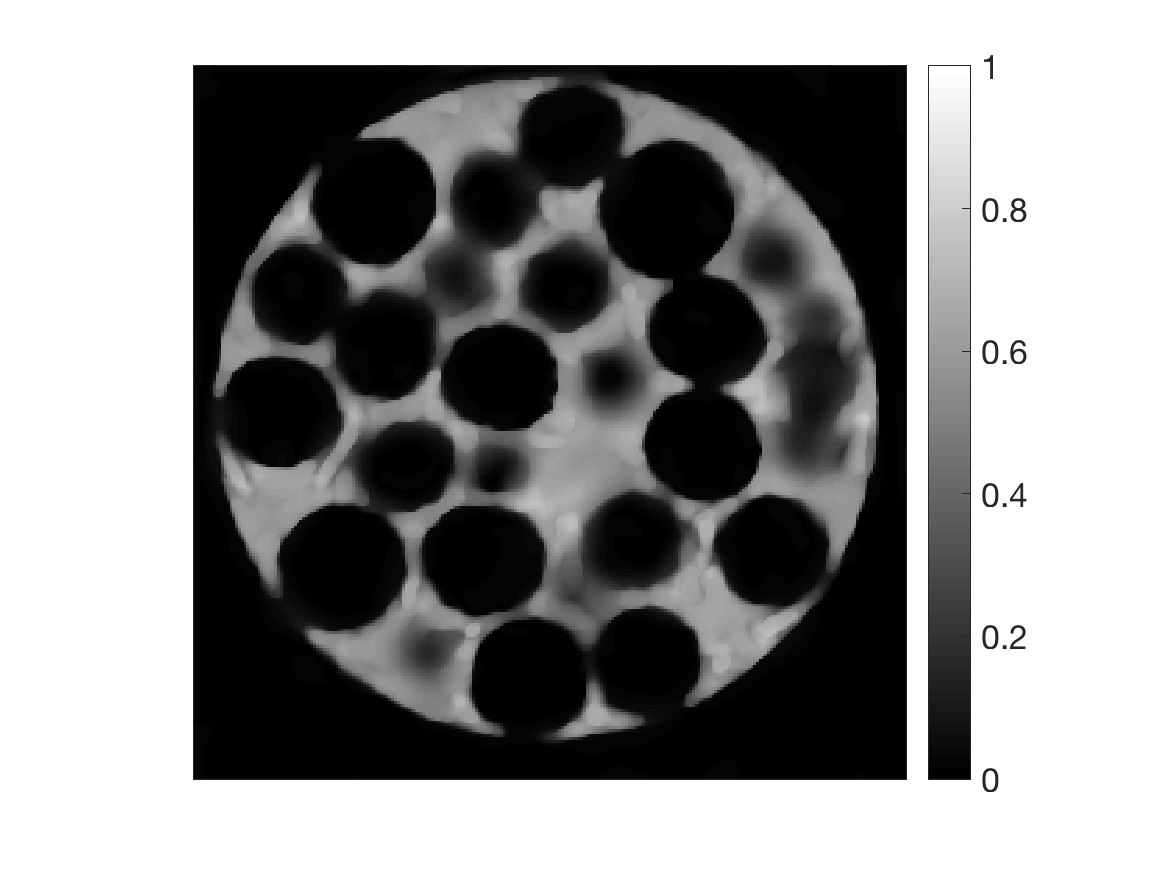}}\vspace{0.05cm}
\subfloat[Iterate 8]{\includegraphics[trim={3cm 1cm 2.2cm 1cm},width=0.24\textwidth]{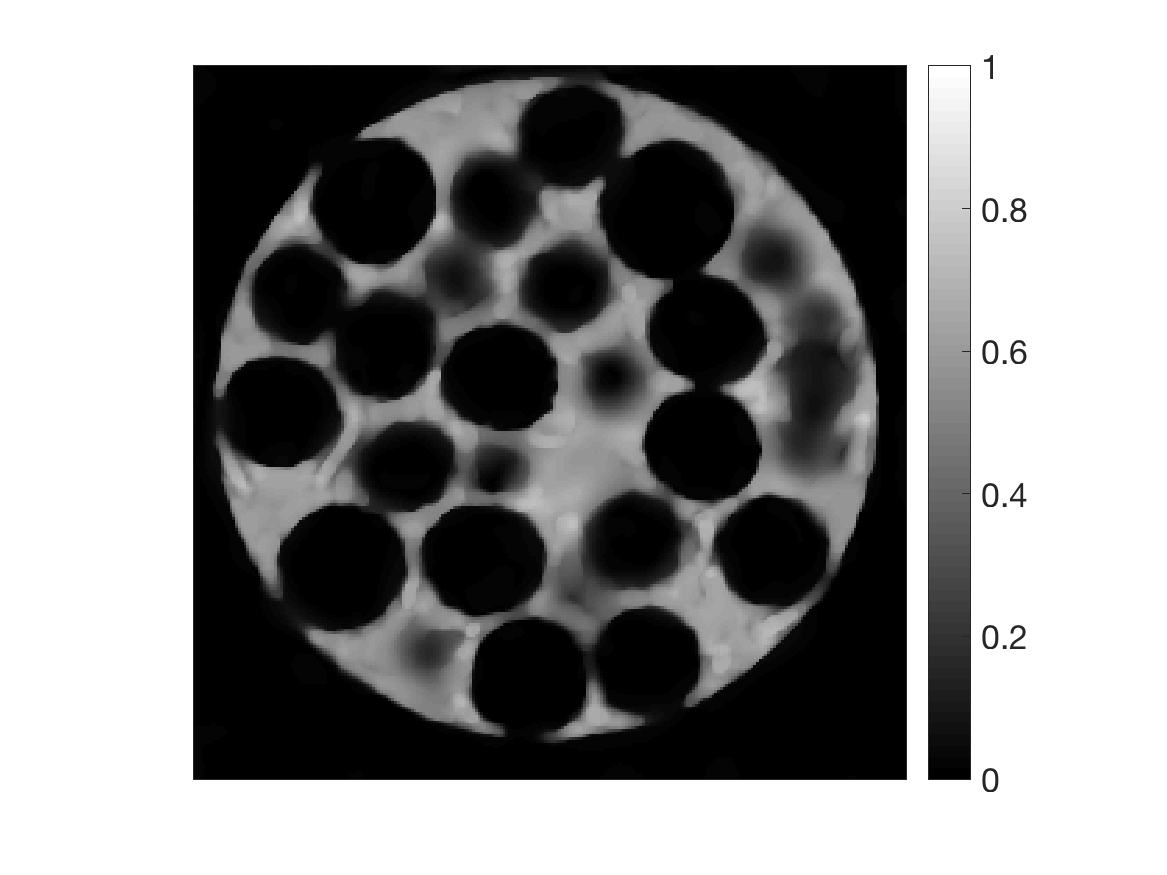}}\\
\subfloat[Iterate 9]{\includegraphics[trim={3cm 1cm 2.2cm 1cm},width=0.24\textwidth]{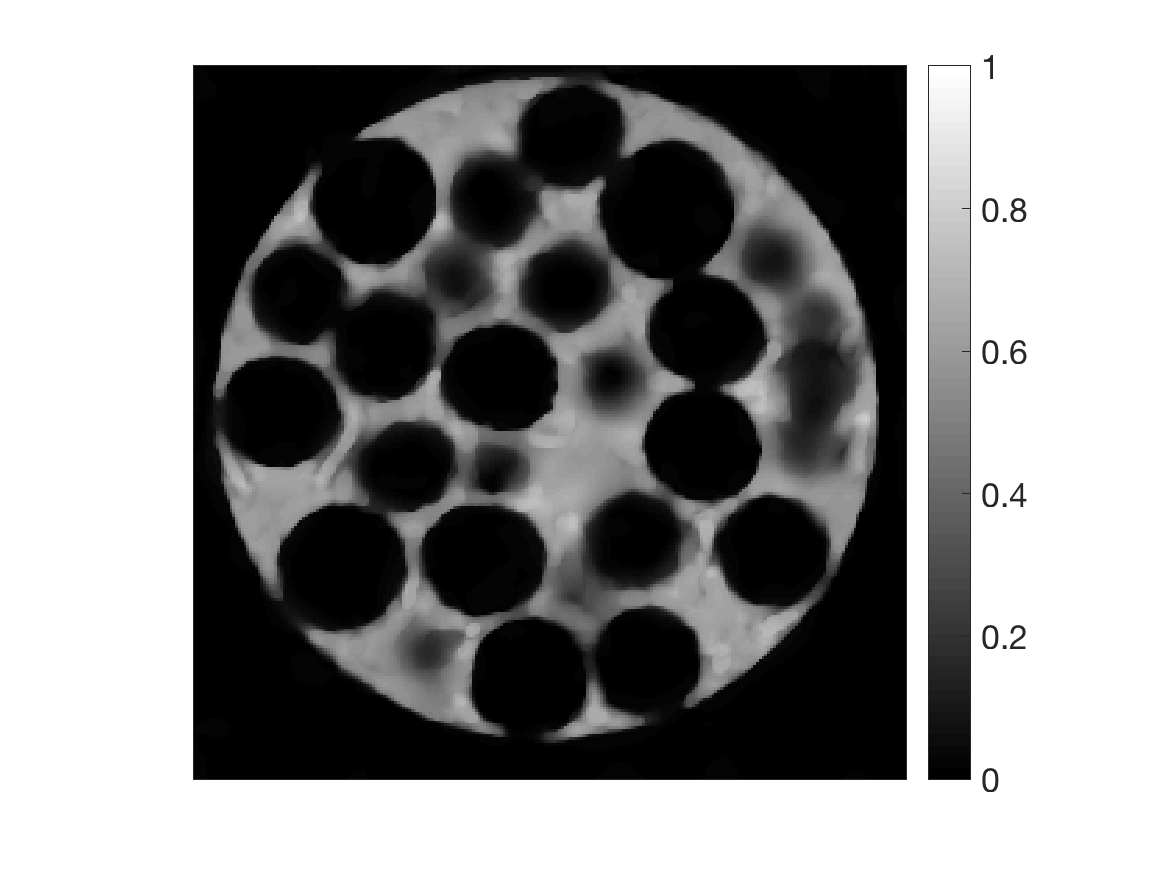}}\vspace{0.05cm}
\subfloat[Iterate 10]{\includegraphics[trim={3cm 1cm 2.2cm 1cm},width=0.24\textwidth]{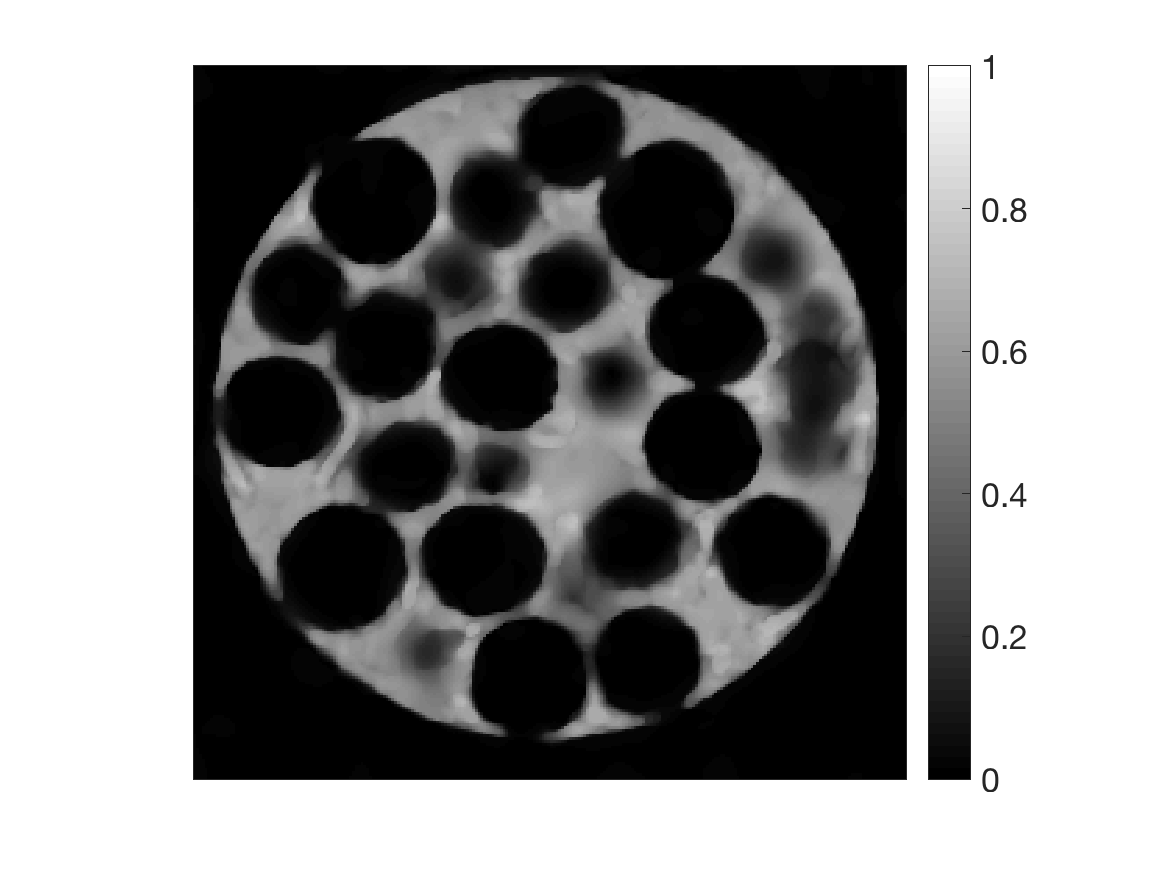}}\vspace{0.05cm}
\subfloat[Iterate 11]{\includegraphics[trim={3cm 1cm 2.2cm 1cm},width=0.24\textwidth]{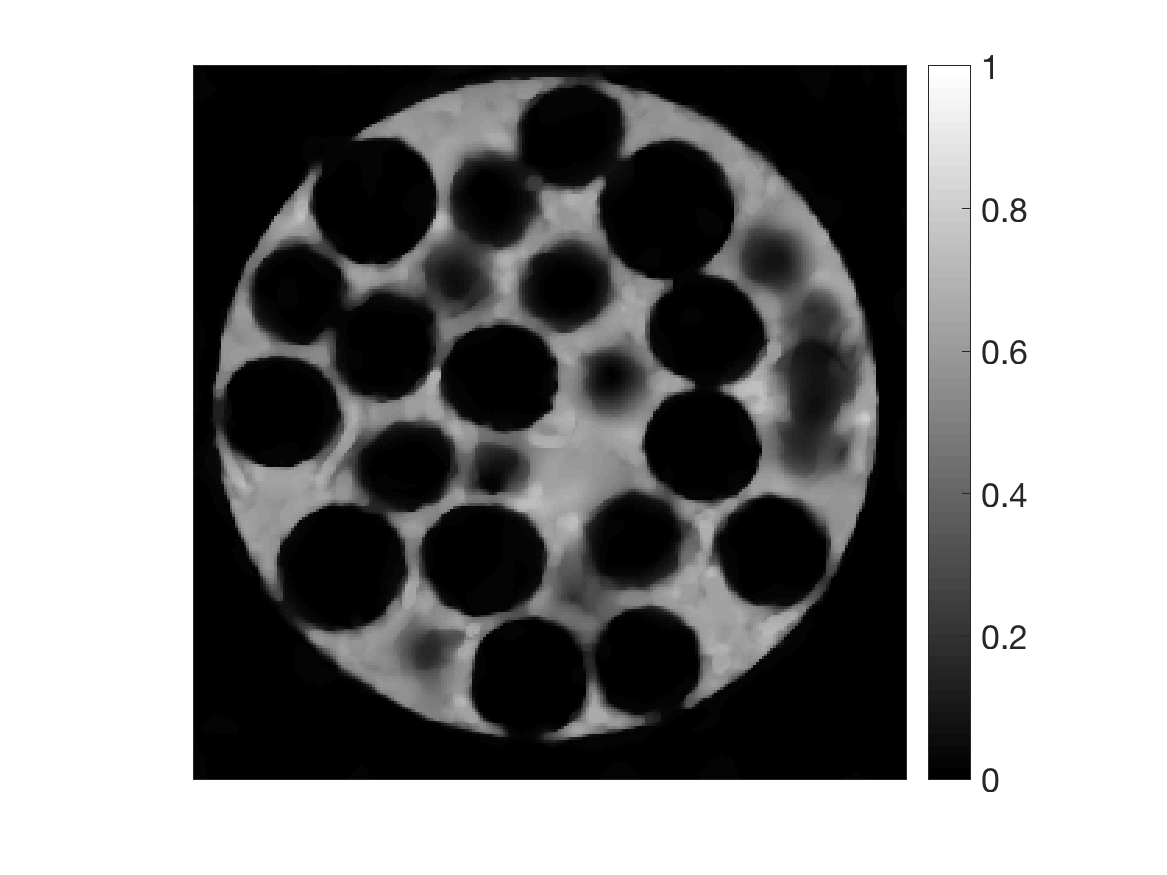}}\vspace{0.05cm}
\subfloat[Iterate 12]{\includegraphics[trim={3cm 1cm 2.2cm 1cm},width=0.24\textwidth]{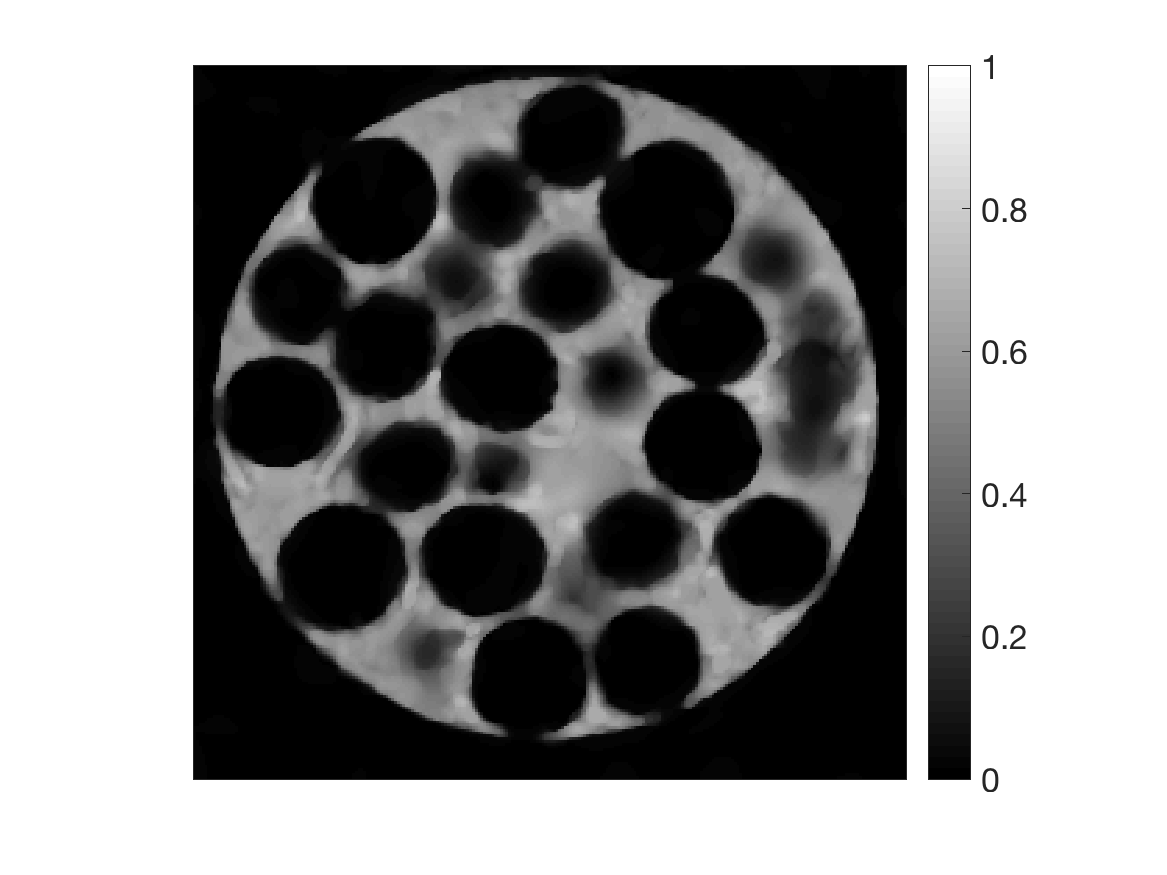}}\\
\subfloat[Iterate 13]{\includegraphics[trim={3cm 1cm 2.2cm 1cm},width=0.24\textwidth]{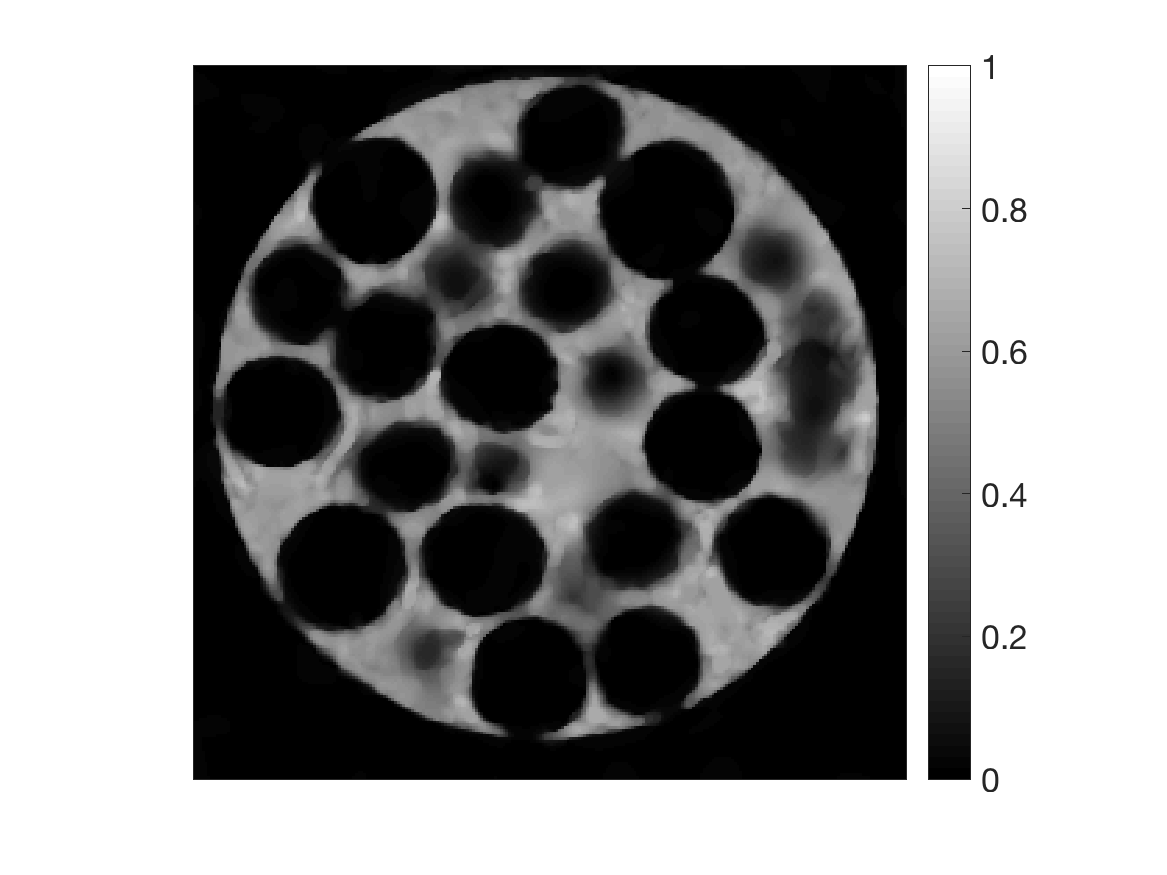}}\vspace{0.05cm}
\subfloat[Iterate 14]{\includegraphics[trim={3cm 1cm 2.2cm 1cm},width=0.24\textwidth]{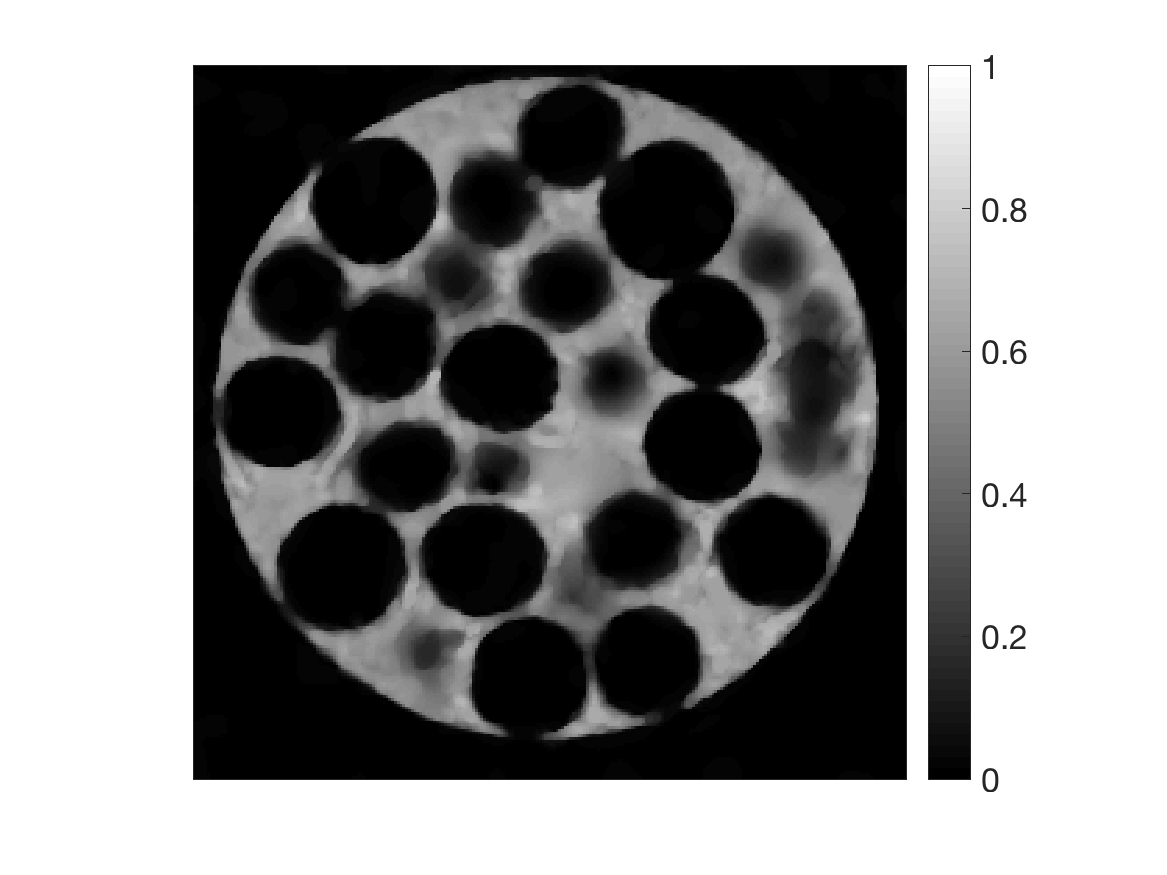}}\vspace{0.05cm}
\subfloat[Iterate 15]{\includegraphics[trim={3cm 1cm 2.2cm 1cm},width=0.24\textwidth]{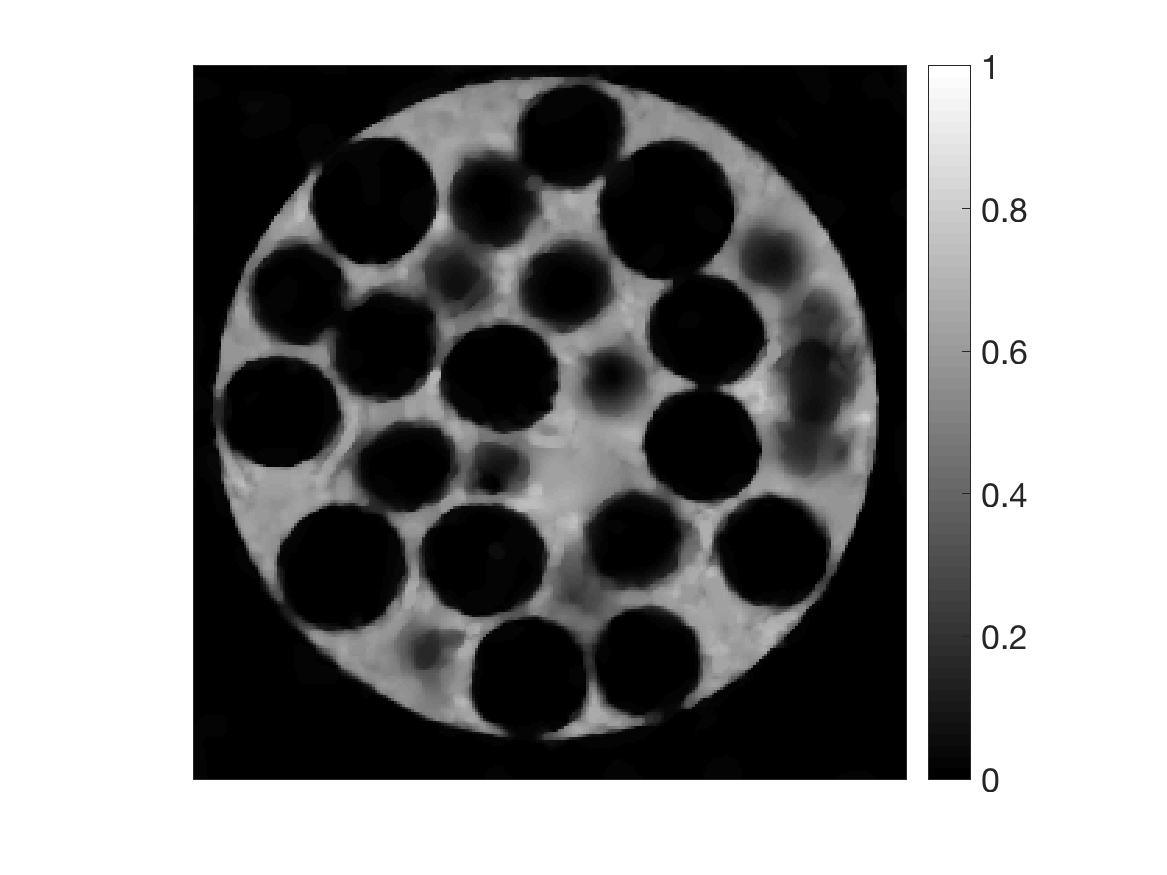}}\vspace{0.05cm}
\subfloat[Iterate 16]{\includegraphics[trim={3cm 1cm 2.2cm 1cm},width=0.24\textwidth]{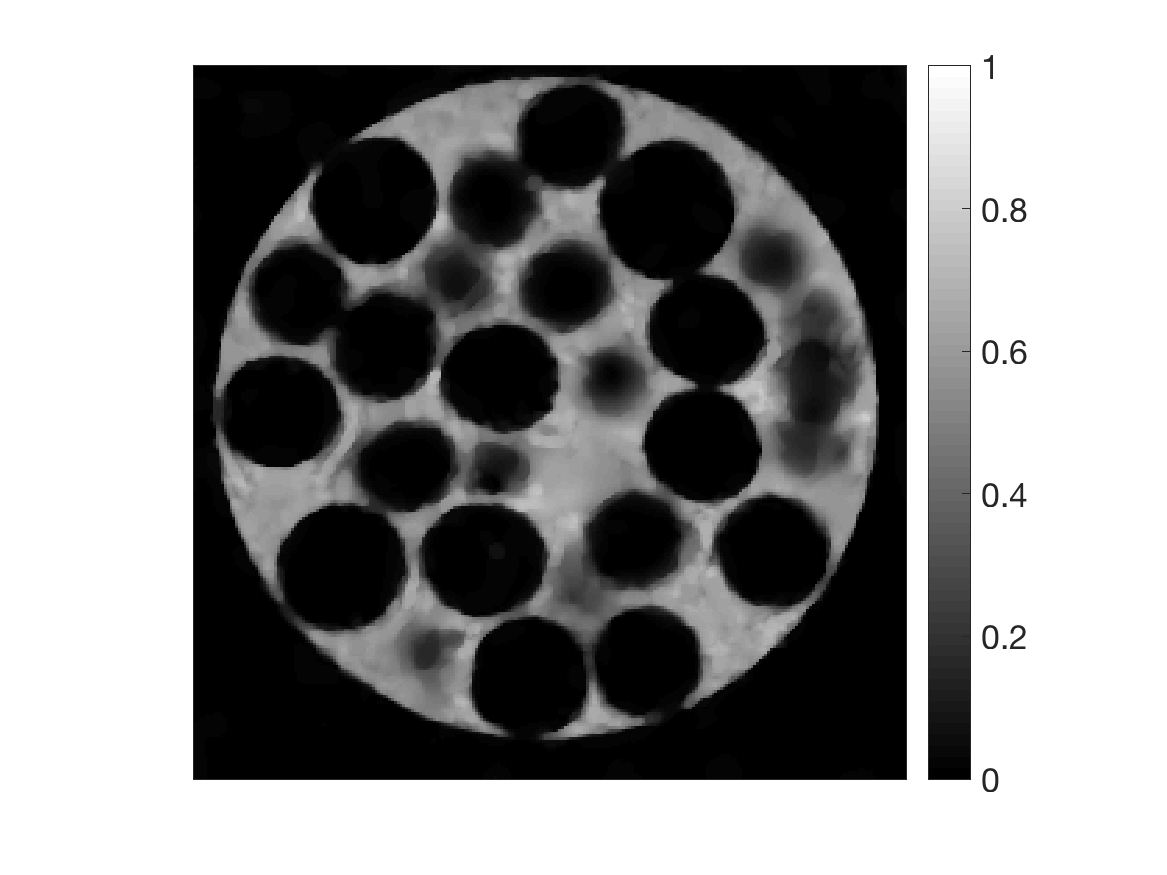}}\\
\subfloat[Iterate 17]{\includegraphics[trim={3cm 1cm 2.2cm 1cm},width=0.24\textwidth]{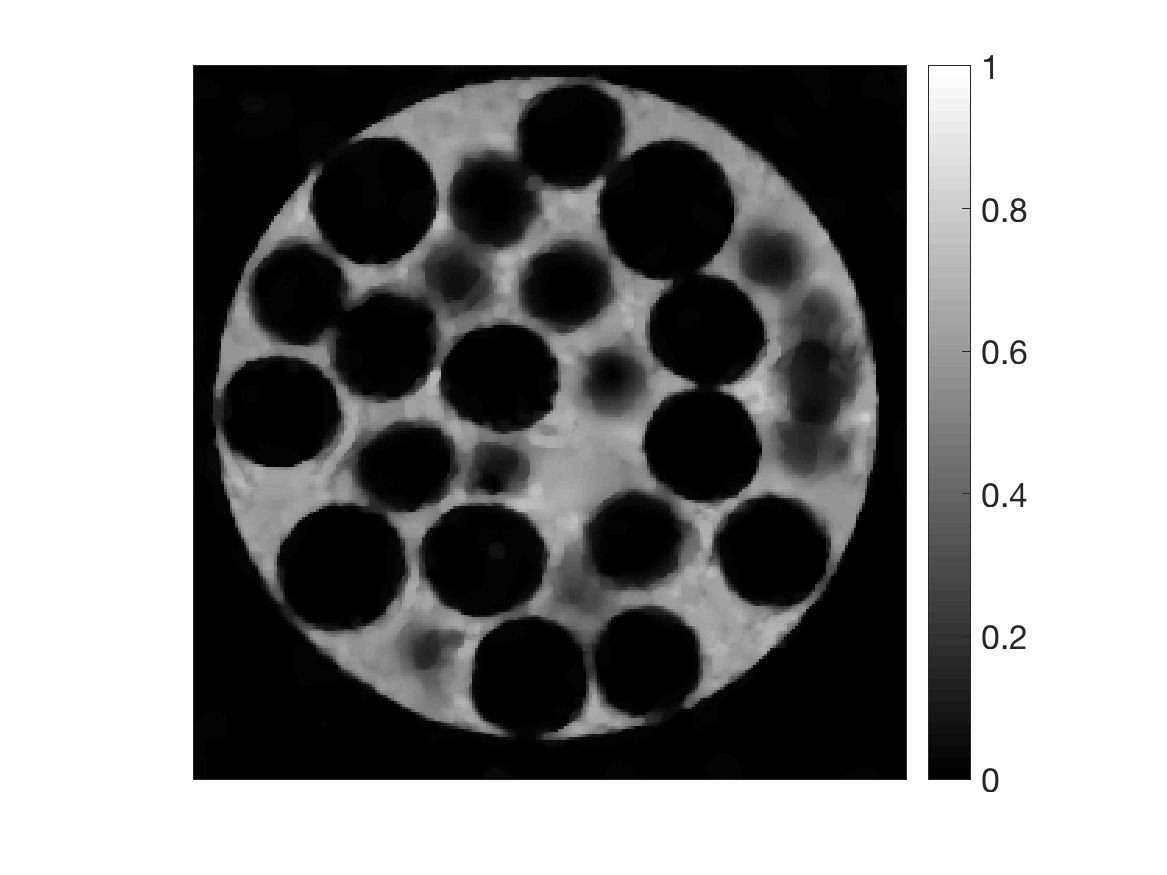}}\vspace{0.05cm}
\subfloat[Iterate 18]{\includegraphics[trim={3cm 1cm 2.2cm 1cm},width=0.24\textwidth]{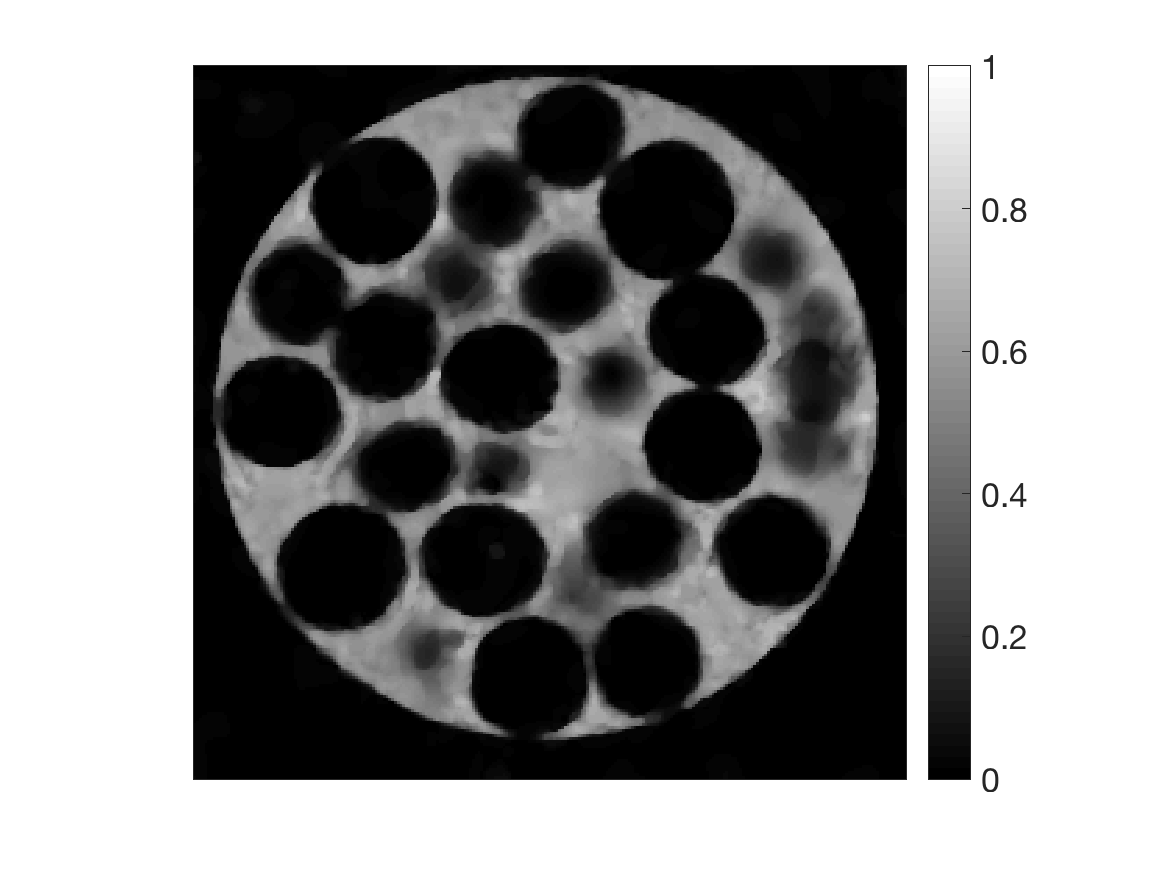}}\vspace{0.05cm}
\subfloat[Iterate 19]{\includegraphics[trim={3cm 1cm 2.2cm 1cm},width=0.24\textwidth]{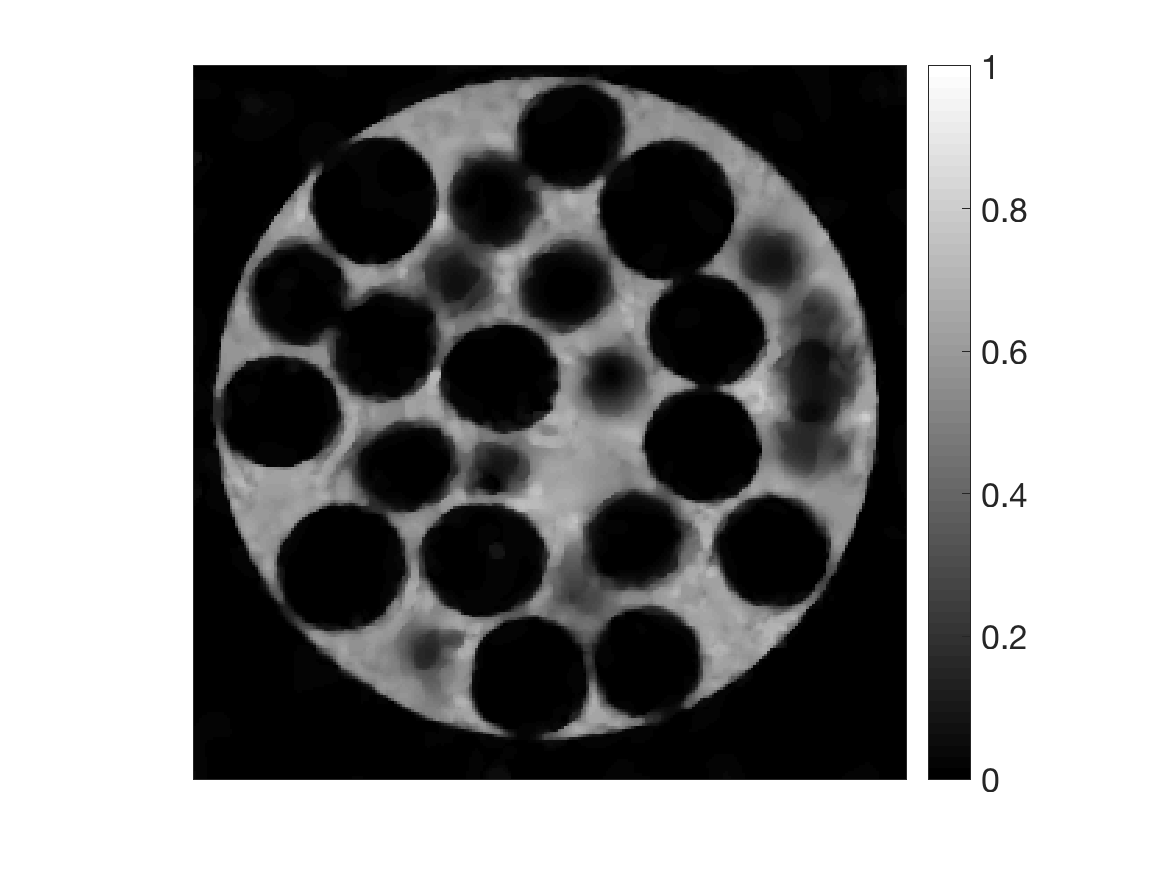}}\vspace{0.05cm}
\subfloat[Iterate 20]{\includegraphics[trim={3cm 1cm 2.2cm 1cm},width=0.24\textwidth]{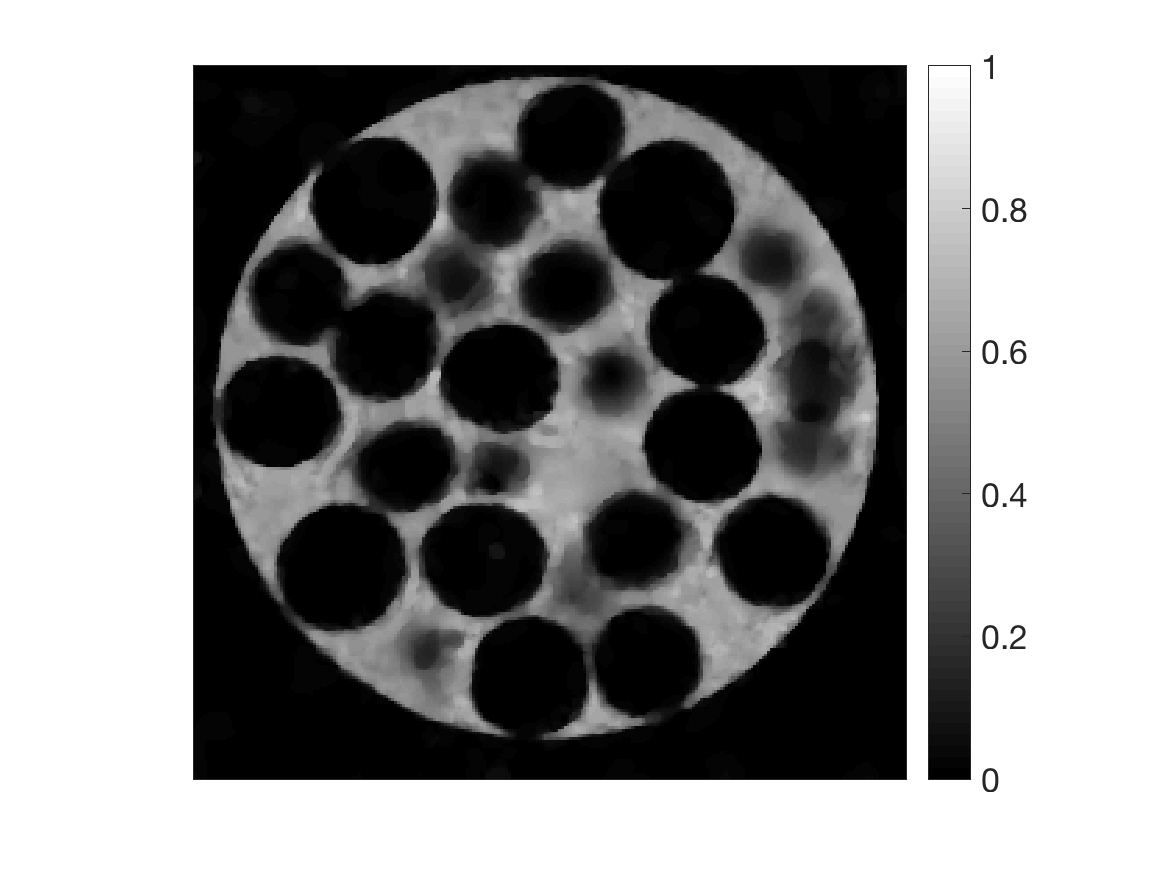}}\\
\end{center}
\caption{Magnitude images of 20 Bregman iterations computed via Algorithm \ref{alg:bregiter}, with $\regparam = 1.5$ and $\beta = 3$.}\label{fig:velmribregiter}
\end{figure}

We refer to \cite{benning2014phase} for more information on iterative regularization in the context of velocity-encoded MRI.

\subsection{Dynamic MRI with Structural Prior}

Dynamic MRI is a topic of high current relevance in biomedical imaging, with different techniques such as fMRI or DCE-MRI. The basic issue is to reconstruct a sequence of images $u=(u_1,\ldots,u_T)$ from measurements $(K_1u_1, \ldots, K_T u_T)$, with $K_t$ being a subsampled Fourier transform (with different subsampling at each time step). Due to the significant measurement times in MRI the subsampling is necessary to obtain a significant time resolution, the time resolution will improve with stronger  undersampling (e.g. in spokes).
The natural data fidelity in this case is thus
$$ \fidelfct(Ku,f) = \frac{1}2 \sum_{t=1}^T \Vert K_t u_t - f_t \Vert^2. $$
With a strong undersampling it becomes rather hopeless to reconstruct meaningful images from the data at a single time step, hence a regularization in time is needed in order to exploit correlations between close time steps. A natural assumptions is smoothness, in the time direction, for this sake a discrete gradient $\Vert u_{t+1} - u_t \Vert^2$ can be penalized in a regularization functional. Moreover, in order to take into account the edges it is natural to include some total variation regularization for each $u_t$. So far, this is an approach that can be used for many dynamic reconstruction problems. A particular feature of such MR investigations is however the existence of a structural prior $u_0$, which is a high resolution MR image at different contrast (e.g. a standard anatomical T1 scan) taken before the start of the dynamic imaging. The prior is reconstructed from a very dense sampling and thus at very high resolution. The important step is to notice that most edges in the images $u_t$ will arise from anatomical structures and are thus present in $u_0$. Hence, an additional structural regularization like the infimal convolution of Bregman distances
$$ ICBV^{p_0}(\cdot,u_0) = D_{TV}^{p_0}(\cdot,u_0) \square D_{TV}^{-p_0}(\cdot,-u_0)$$
can be used to achieve superresolution in the dynamic imaging series. 

The regularization functional
$$ J(u)= \sum_{t=1}^T \omega_t |u_t|_{BV} + 
\sum_{t=1}^T (1-\omega_t) ICBV^{p_0}(u,u_0) + 
\sum_{t=1}^{T-1} \frac{\gamma_t}2 \Vert u_{t+1} - u_t \Vert^2 $$
combining the three parts has been proposed and investigated in \cite{rasch2017dynamic}. The results indicate an enormous potential to obtain reconstructions at high resolution from rather extreme undersampling in time. Those are illustrated in Figure \ref{fig:dynamicMR} for a several different time steps of a simulated data set. The first line shows the sampling at different time steps, the last column shows the prior image $u_0$ instead. The second line provides direct reconstruction without regularization (note that the Fourier transform is continuously invertible, so without undersampling the direct inversion is a standard technique). The third line displays the results with the proposed method to be compared to the ground truth used for simulating data in the fourth line. These results are obtained on simulated MR data, we refer to \cite{rasch2017dynamic} for a further study on real data.

 \setlength{\tabcolsep}{0.2mm}
 \begin{figure} \label{fig:dynamicMR}
 \begin{tabular}{B{0.04\textwidth}B{0.15\textwidth}B{0.15\textwidth}B{0.15\textwidth}B{0.15\textwidth}B{0.15\textwidth}}
 
    \rotatebox[origin=c]{90}{samp/prior} &
    \includegraphics[width=0.15\textwidth]{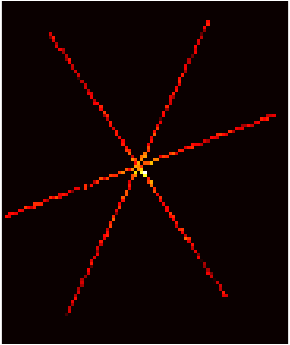} &
    \includegraphics[width=0.15\textwidth]{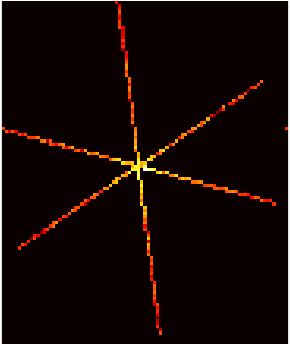} & 
    \includegraphics[width=0.15\textwidth]{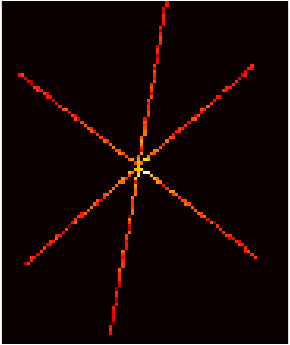} &
    \includegraphics[width=0.15\textwidth]{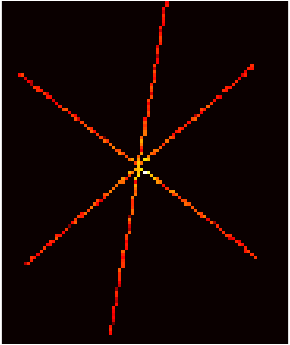} &
    \includegraphics[width=0.15\textwidth]{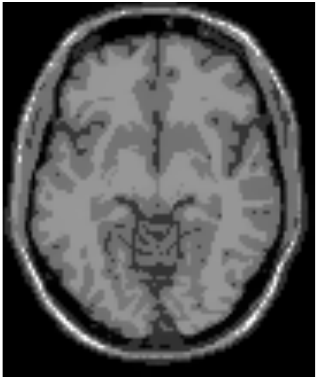}\\ \\[1mm]
 
    \rotatebox[origin=c]{90}{LS} &
    \includegraphics[width=0.15\textwidth]{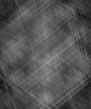} &
    \includegraphics[width=0.15\textwidth]{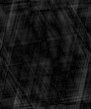} & 
    \includegraphics[width=0.15\textwidth]{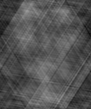} &
    \includegraphics[width=0.15\textwidth]{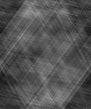} &
    \includegraphics[width=0.15\textwidth]{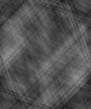} \\ 
    
    \rotatebox[origin=c]{90}{proposed} &
    \includegraphics[width=0.15\textwidth]{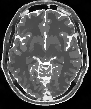} &
    \includegraphics[width=0.15\textwidth]{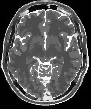} & 
    \includegraphics[width=0.15\textwidth]{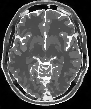} &
    \includegraphics[width=0.15\textwidth]{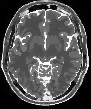} &
    \includegraphics[width=0.15\textwidth]{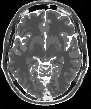} \\ 
    
    \rotatebox[origin=c]{90}{ground truth} &
    \includegraphics[width=0.15\textwidth]{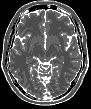} &
    \includegraphics[width=0.15\textwidth]{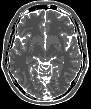} & 
    \includegraphics[width=0.15\textwidth]{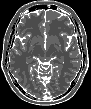} &
    \includegraphics[width=0.15\textwidth]{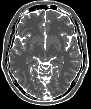} &
    \includegraphics[width=0.15\textwidth]{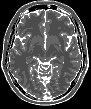} \\ \\[0.5mm] 
    & $t = 8$ & $t = 16$ & $t = 21$ & $t = 28$ & $t = 42$ \\
    
 \end{tabular}
\caption{Results of undersampled dynamic MRI reconstruction with different methods at five different time steps.}
\end{figure}

\subsection{Nonlinear Spectral Image Fusion}

\begin{figure}[!t]
\begin{center}
\includegraphics[width=\textwidth]{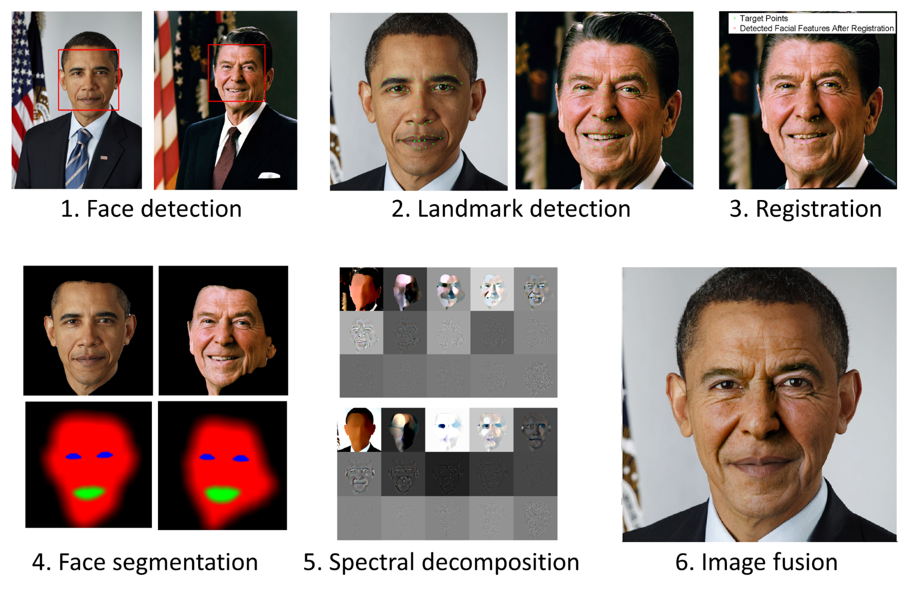}
\end{center}
\caption{Illustration of the pipeline for facial image fusion using nonlinear spectral decompositions. From \cite{Benning2017}.}
\label{fig:fusionpipeline}
\end{figure}

The nonlinear spectral transform as introduced in Section \ref{subsec:nonlinspectrafo} can be used to suppress, enhance or extract features of signals at different scales. In \cite{Benning2017} it has been used to fuse features at different scales from two images into a single image, in order to create realistically looking image fusions. The mathematical procedure is as follows: Given two images, both images are preprocessed such that they are aligned (registered) and that regions within the images are segmented such that the images are fused only in selected regions. Denoting the registered images as $f_1$ and $f_2$, they can be represented via their spectral transforms, i.e. 
\begin{align*}
u_1 = \mathcal{S}(f_1, (u)_{n=0}^{k^\ast}, \boldsymbol{c}_1, \regparam) + \underbrace{f_1 - \mathcal{S}(f_1, (u)_{n=0}^{k^\ast}, \textbf{1}, \regparam)}_{=: r_1^{\regparam, k^\ast}} \, ,
\intertext{and}
u_2 = \mathcal{S}(f_2, (u)_{n=0}^{k^\ast}, \boldsymbol{c}_2, \regparam) + \underbrace{f_2 - \mathcal{S}(f_2, (u)_{n=0}^{k^\ast}, \textbf{1}, \regparam)}_{=: r_2^{\regparam, k^\ast}} \, ,
\end{align*}
for $k^\ast \geq 1$, $\regparam \in \regdomain$ and coefficients $\boldsymbol{c}_1 \in \R^{k^\ast}$, $\boldsymbol{c}_2 \in \R^{k^\ast}$ and $\textbf{1} \in \{1 \}^{k^\ast}$ being the constant one-vector. Obviously we have $u_1 = f_1$ and $u_2 = f_2$ if $\boldsymbol{c}_1 = \textbf{1}$ and $\boldsymbol{c}_2 = \textbf{1}$. 

\begin{figure}[!t]
\begin{center}
\includegraphics[width=0.4\textwidth]{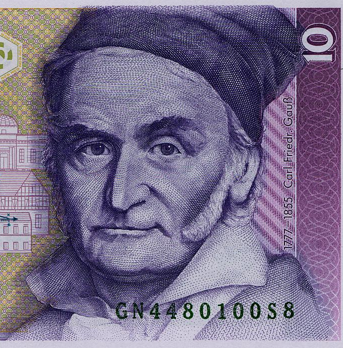}\vspace{0.05cm}
\includegraphics[width=0.4\textwidth]{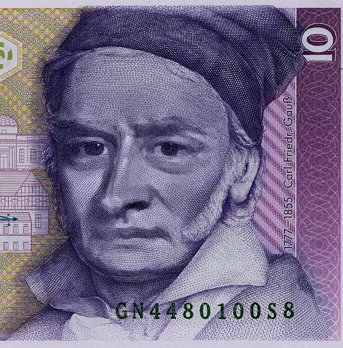}\\
\includegraphics[width=0.4\textwidth]{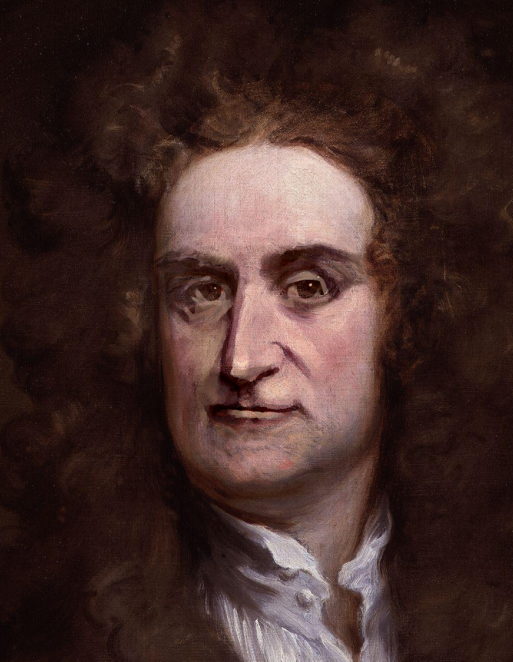}\vspace{0.05cm}
\includegraphics[width=0.4\textwidth]{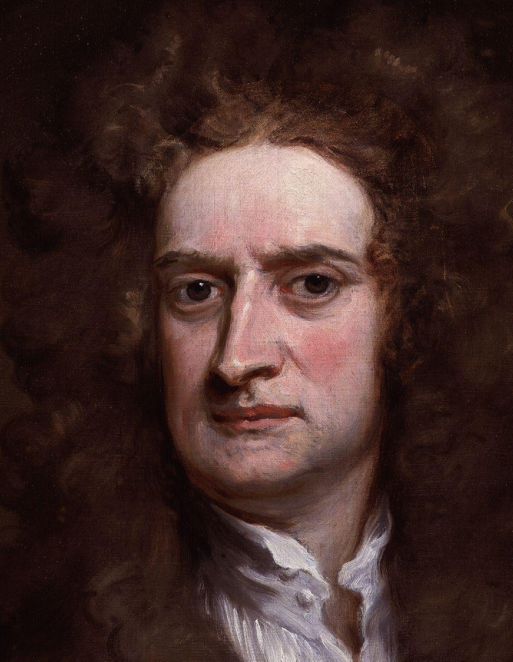}
\end{center}
\caption{Image fusion using the nonlinear spectral TV decomposition on the challenging example of fusing a banknote with a picture of Gau\ss\ and a painting of Newton. From the supplementary material of \cite{Benning2017}.}
\label{fig:gaussnewton}
\end{figure}

In order to incorporate the face segmentation into the image fusion process, we allow the coefficient vectors $\boldsymbol{c}_1$ and $\boldsymbol{c}_2$ to be spatially varying functions $\boldsymbol{c}_1 :\Omega \rightarrow \R^{k^\ast}$ and $\boldsymbol{c}_2 :\Omega \rightarrow \R^{k^\ast}$, respectively. Here $\Omega$ denotes the image domain. The image fusion process can then mathematically be described as 
\begin{align*}
u_{\text{fused}} := \mathcal{S}(f_1, k^\ast, \regparam, \boldsymbol{c}_1) + \mathcal{S}(f_2, k^\ast, \regparam, \boldsymbol{c}_2) + r_1^{\regparam, k^\ast} \, .
\end{align*}
The individual steps of the image fusion pipeline are visualized in Figure \ref{fig:fusionpipeline}. For challenging examples this automation may very well fail. Nevertheless, the spectral image fusion still works if registration and segmentation are carried out manually, as can be seen in Figure \ref{fig:gaussnewton}. For more information on the nonlinear spectral image fusion we refer to \cite{Benning2017}.


\section{Advanced Issues}
 
In the following we comment on some advanced issues in particular related to iterative variational methods extending the ones presented above, namely extension to nonconvex problems, in particular with respect to the data fidelity as arising in nonlinear inverse problems, and to modern machine learning approaches.

\subsection{Nonconvex Optimization}\label{subsec:nonconvreg}
In the context of inverse problems one usually deals with data fidelities of the form $\fidelity{Ku}{\noisy}$ that measure the deviation between $Ku$ and $\noisy$ in some sense. So far we have always assumed this particular structure, and also that $\fidelfct$ is convex. Both assumptions can be relaxed. In the following we assume that we simply have some nonconvex energy functional $\energy:\domain \rightarrow \R$ that is Fr\'{e}chet-differentiable with gradient $\nabla \energy$. As there may not exist critical points or finding them is unstable due to ill-posedness, it makes sense to generalize \eqref{eq:varreg} to
\begin{align}
\regop(\regparambold) = \argmin_{u \in \domain} \left\{ \energy(u) + \regfctarg{u} \right\} \, .\label{eq:nonconvexvarreg}
\end{align}
Here we want to emphasize that $\regop(\regparambold)$ is not necessarily a regularization operator in the classical sense as in general we do not deal with an inverse problem that depends on some data $\noisy$. In Section \ref{subsec:nonlininvprob} we particularly investigate the case in which $\energy$ is of the form $\energy(\cdot) = \fidelfct(K(\cdot), \noisy)$, where $K$ stems from a nonlinear inverse problem, and where $\fidelfct(\cdot, \noisy)$ is potentially nonconvex also in its first argument. 

It is important to emphasize that even for nonsmooth, nonconvex optimization there is a vast amount of recent publications, ranging from forward-backward, respectively proximal-type, schemes \cite{attouch2009convergence,attouch2010proximal,attouch2013convergence,Bonettini2015,Bonettini2016}, over linearized proximal schemes \cite{xu2013block,bolte2014proximal,xu2017globally,nikolova2017alternating}, to inertial methods \cite{ochs2014ipiano,pock2016inertial}, primal-dual algorithms \cite{valkonen2014primal,li2015global,moeller2015variational,benning2015preconditioned}, scaled gradient projection methods \cite{Prato.ea2016}, nonsmooth Gau\ss-Newton extensions \cite{drusvyatskiy2016nonsmooth,ochs2017non} and nonlinear Eigenproblems \cite{hein2010inverse,bresson2012convergence,benning2016learning,boct2017proximal,NIPS2016_6226,benning2017learning}. We focus mainly on recent generalizations of the proximal gradient method and the linearized Bregman iteration for nonconvex functionals $\energy$ in the following.

\subsubsection{Proximal Gradient Method}\label{subsubsec:proxgrad}
A most basic approach to find solutions of \eqref{eq:nonconvexvarreg} iteratively is via proximal gradient descent, respectively forward-backward splitting \cite{lions1979splitting}. The idea is to linearize the nonconvex part $\energy$ and to add a damping with respect to the previous iterate. If we allow this damping to be carried out via a Bregman distance w.r.t. a Legendre functional $H$, we obtain the recently proposed \emph{Bregman proximal gradient method} \cite{bolte2017first}
\begin{align}
\begin{split}
\regopit(u^{k - 1}, \regparambold) &= \argmin_{u \in \domain} \left\{ \regparam^{k - 1} \langle \nabla \energy(u^{k - 1}), u - u^{k - 1} \rangle + \bregdis[H]{}{u}{u^{k - 1}} + \regparam^{k - 1} \regfct(u, \regparam) \right\}\\
u^k &\in \regopit(u^{k - 1}, \regparambold)
\end{split} \, ,\label{eq:proxgrad}
\end{align}
for $\regparambold = (\regparam, \regparam^0, \ldots, \regparam^{k - 1})$. Here we want to emphasize that $\regopitarg{u^{k - 1}}$ is no longer a regularization operator in the classical sense as we do not necessarily deal with an inverse problem anymore. Obviously, if $\energy$ is a (potentially nonconvex) data fidelity of some nonlinear inverse problem, $\regopitarg{u^{k - 1}, p^{k - 1}}$ depends on some data $\noisy$ and we again deal with a regularization problem, which this time approaches the solution of a (potentially) nonlinear inverse problem. This more specific scenario will be addressed in Section \ref{subsec:nonlininvprob}. Without additional assumptions on $\energy$, $H$ and $\regfct$ there is little chance that we can carry out a convergence analysis for \eqref{eq:proxgrad} or even prove existence of the updates. A typical assumption is Lipschitz-continuity of $\nabla E$, i.e. we guarantee
\begin{align*}
\| \nabla E(u) - \nabla E(v) \|_{\domain^\ast} \leq L \| u - v \|_{\domain}
\end{align*}
for all $u, v \in \domain$ and a constant $L > 0$. A nice aspect about this property is that it implies convexity of the family of functionals
\begin{align}
\frac{L}{\gamma_i} H_i - E \, , \label{eq:famofconvfcts}
\end{align}
see \cite{bauschke2016descent,benning2016gradient,bolte2017first}, where $\{ H_i \}_{i = 1, \ldots}$ is a family of $\gamma_i$-strongly convex functionals, i.e.
\begin{align*}
\frac{\gamma_i}{2} \| u - v \|_{\domain}^2 \leq \bregdis[H_i]{}{u}{v} \, ,
\end{align*}
for all $u, v \in \domain$. Let us now assume that $H$ in \eqref{eq:proxgrad} is member of \eqref{eq:famofconvfcts} with strong convexity constant $\gamma$, i.e.
\begin{align}
H_\gamma(u) := \frac{L}{\gamma} H(u) - E(u) \label{eq:hgamma}
\end{align}
is convex for all $u \in \domain$. Then this convexity assumption is already enough to ensure a sufficient decrease of the energy $E + \regfct$ in each iteration of \eqref{eq:proxgrad}. 

\begin{lemma}\label{lem:proxgradsuffdec}
Suppose $\energy$ is coercive or has bounded level-sets, $\inf_u \energy(u) > - \infty$ and $\nabla \energy$ is Lipschitz continuous with constant $L$, and let $H$ be a Legendre functional in the sense of Definition \ref{def:legendre} that is also $\gamma$-strongly convex. Further assume 
\begin{align}
0 < \regparam^{k - 1} < \frac{\gamma C^k}{L + \gamma C^ k\rho} \qquad \text{for} \qquad C^k := \frac{\symbreg[H]{u^k}{u^{k - 1}}}{\bregdis[H]{}{u^k}{u^{k - 1}}} \, ,\label{eq:stepsizecond}
\end{align}
for a constant $0 <\rho$, for all $k \in \N$, and that $E + J(\cdot, \regparam)$ has at least one critical point. Then the iterates of \eqref{eq:proxgrad} satisfy
\begin{align}
\energy(u^k) + \regfct(u^k, \regparam) + \rho \symbreg[H]{u^k}{u^{k - 1}} \leq \energy(u^{k - 1}) + \regfct(u^{k - 1}, \regparam) \, ,\label{eq:proxgradsuffdecrease1}
\end{align}
for $u^k \in \regop(u^{k - 1}, \regparambold)$ and all $k \in \N$.
\begin{proof}
From the convexity of \eqref{eq:hgamma} we immediately observe
\begin{align*}
\begin{split}
0 {} \leq {} \bregdis[H_\gamma]{}{u^k}{u^{k - 1}} {} = {} &\frac{L}{\gamma} \bregdis[H]{}{u^k}{u^{k - 1}} \\
&- \left( \energy(u^k) - \energy(u^{k - 1}) - \langle \nabla \energy(u^{k - 1}), u^k - u^{k - 1} \rangle \right)
\end{split} \, .
\end{align*}
As a direct consequence, we have derived the estimate
\begin{align}
E(u^k) + \langle \nabla E(u^{k - 1}), u^{k - 1} - u^k \rangle - \frac{L}{\gamma} \bregdis[H]{}{u^k}{u^{k - 1}} \leq E(u^{k - 1})\, .\label{eq:suffdecreaseproxgrad1}
\end{align}
From the optimality condition of \eqref{eq:proxgrad} we obtain
\begin{align}
\nabla \energy(u^{k - 1}) = \frac{1}{\regparam^{k - 1}} \left( \nabla H(u^{k - 1}) - \nabla H(u^k) \right) - p^k \, ,\label{eq:proxgradopt}
\end{align}
for $p^k \in \partial \regfct(u^k, \regparam)$. Inserting \eqref{eq:proxgradopt} into \eqref{eq:suffdecreaseproxgrad1} yields
\begin{align}
E(u^k) + \frac{1}{\regparam^{k - 1}} \symbreg[H]{u^k}{u^{k - 1}} - \frac{L}{\gamma} \bregdis[H]{}{u^k}{u^{k - 1}} \leq E(u^{k - 1}) + \langle p^k, u^{k - 1} - u^k \rangle \, .\label{eq:suffdecreaseproxgrad2}
\end{align}
Due to the convexity of $\regfct(\cdot, \regparam)$ we can estimate $\langle p^k, u^{k - 1} - u^k \rangle \leq \regfct(u^{k - 1}, \regparam) - \regfct(u^k, \regparam)$. Applying this estimate to \eqref{eq:suffdecreaseproxgrad2} results in
\begin{align*}
\begin{split}
&E(u^k) + \regfct(u^k, \regparam) + \frac{1}{\regparam^{k - 1}} \symbreg[H]{u^k}{u^{k - 1}} - \frac{L}{\gamma} \bregdis[H]{}{u^k}{u^{k - 1}} \\
{} \leq {} &E(u^{k - 1}) + \regfct(u^{k - 1}, \regparam)
\end{split}
\, .
\end{align*}
Together with the stepsize bound \eqref{eq:stepsizecond} this concludes the proof.
\end{proof}
\end{lemma}
\begin{remark}
Note that we haven't made use of the Lipschitz continuity of $\nabla E$, but only of the convexity of \eqref{eq:hgamma} in order to obtain a sufficient decrease.
\end{remark}
\begin{remark}
Due to the $\gamma$-strong convexity of $H$ the estimate \eqref{eq:proxgradsuffdecrease1} automatically implies
\begin{align}
\energy(u^k) + \regfct(u^k, \regparam) + \rho \gamma \| u^k - u^{k - 1} \|_{\domain}^2 \leq \energy(u^{k - 1}) + \regfct(u^{k - 1}, \regparam) \, .\label{eq:proxgradsuffdecrease2}
\end{align}
\end{remark}
If we additionally assume that both $\nabla E$ and $\nabla H$ are Lipschitz-continuous, we further obtain a bound for the gradient of the energy $\energy + \regfct$ at iterate $u^k$.
\begin{lemma}\label{lem:proxgraditgap}
Suppose the same assumptions hold as in Lemma \ref{lem:proxgradsuffdec}. We further assume that $\nabla E$ is Lipschitz-continuous with constant $L$ and $\nabla H$ is Lipschitz-continuous with constant $\delta$. Then we observe
\begin{align*}
\| \nabla E(u^k) + p^k \|_{\domain^\ast} \leq \left(L + \frac{\delta}{\regparam^{k - 1}}\right) \| u^k - u^{k - 1} \|_{\domain}
\end{align*}
for all $p^k \in \partial \regfct(u^k, \regparam)$.
\begin{proof}
This follows trivially from \eqref{eq:proxgradopt} and the Lipschitz-continuity of both $\nabla E$ and $\nabla H$.
\end{proof}
\end{lemma}
In a finite dimensional setting $\domain = \R^n$ it is now sufficient to assume that $\energy + \regfct$ is a Kurdyka-\L ojasiewicz (KL) function \cite{lojasiewicz1963propriete,kurdyka1998gradients,bolte2007lojasiewicz} in order to show that the iterates \eqref{eq:proxgrad} converge globally to a critical point of $\energy + \regfct$.
\begin{theorem}\label{thm:proxgradglobconv}
Let the same assumptions hold true as in Lemma \ref{lem:proxgraditgap}. Further assume $\domain = \R^n$ and that $\energy + \regfct$ is a KL function that has at least one critical point. Then the iterates \eqref{eq:proxgrad} converge globally to a critical point of the energy $\energy + \regfct$.
\begin{proof}
See proof of \cite[Theorem 4.1 (ii)]{bolte2017first}.
\end{proof}
\end{theorem}
\noindent We refer the reader to \cite{bolte2010characterizations} for a detailed investigation of the class of KL functions, and to \cite{bolte2017first} for more information on the Bregman proximal gradient.

\subsubsection{Linearized Bregman Iteration for Nonconvex Functionals}
The linearized Bregman iteration introduced in Section \ref{subsec:linbregiter} can easily be adapted to tackle general, non-convex optimization problems. Suppose a Fr\'{e}chet-differentiable functional $\energy:\domain \rightarrow \R$ with Fr\'{e}chet-gradient $\nabla \energy$, then we can simply modify Algorithm \ref{alg:linbregiter} to
\begin{align}
\begin{split}
\regopitarg{u^{k - 1}, p^{k - 1}} &= \argmin_{u \in \domain} \left\{ \langle \nabla \energy(u^{k - 1}), u - u^{k - 1} \rangle + \regparam^{k - 1} \bregdis[\regfct(\cdot, \regparam)]{p^{k - 1}}{u}{u^{k - 1}} \right\} \\
u^k &\in \regopitarg{u^{k - 1}, p^{k - 1}}\\
p^k &= p^{k - 1} - \frac{1}{\regparam^{k - 1}} \nabla \energy(u^{k - 1})
\end{split} \, ,\label{eq:nonconvlinbregman}
\end{align}
for $\regparambold = (\regparam, \regparam^0, \ldots, \regparam^{k - 1})$ and $p^{k - 1} \in \partial \regfct(u^{k - 1}, \regparam)$. This method for arbitrary nonconvex energies $\energy$ has first been introduced in \cite{benning2016gradient} and mathematically analyzed in \cite{2017arXiv171204045B}. As in the previous section, $\regopitarg{u^{k - 1}, p^{k - 1}}$ is no longer a regularization operator in the classical sense, unless $\energy$ is a (potentially nonconvex) data fidelity of some nonlinear inverse problem.

It becomes evident that \eqref{eq:nonconvlinbregman} and \eqref{eq:proxgrad} coincide if $\regfct$ in \eqref{eq:nonconvlinbregman} is a Legendre functional and if $\regfct$ in \eqref{eq:proxgrad} is zero. Hence, the convergence analysis closely follows the convergence analysis of the proximal gradient method. We assume that $\regfct( \cdot, \regparam)$ is $\gamma$-strongly convex and that 
\begin{align}
\regfct_\gamma(u, \regparam) := \frac{L}{\gamma} \regfct(u, \regparam) - E(u)\label{eq:regfctgamma}
\end{align}
is convex. Then we can show the following sufficient decrease of the energy \cite{benning2016gradient}.
\begin{lemma}\label{lem:nonconvlinbregmansuffdecrease}
Suppose $\energy$ is coercive or has bounded level-sets, $\inf_u \energy(u) > - \infty$, $\regparam^{k - 1}$ satisfies \eqref{eq:stepsizecond} with
\begin{align*}
C^k := \frac{\symbreg[\regfct(\cdot, \regparam)]{u^k}{u^{k - 1}}}{\bregdis[\regfct(\cdot, \regparam)]{p^{k - 1}}{u^k}{u^{k - 1}}} \, ,
\end{align*}
and that $\energy$ has at least one critical point. Then the iterates of \eqref{eq:nonconvlinbregman} satisfy
\begin{align}
\energy(u^k) + \rho \bregdis[\regfct(\cdot, \regparam)]{p^{k - 1}}{u^k}{u^{k - 1}} \leq \energy(u^{k - 1}) \, .\label{eq:suffdecrease}
\end{align}
\begin{proof}
From the convexity of \eqref{eq:regfctgamma} we immediately observe
\begin{align*}
\begin{split}
0 {} \leq {} \bregdis[\regfct_\gamma( \cdot, \regparam)]{q^{k - 1}}{u^k}{u^{k - 1}} {} = {} &\frac{L}{\gamma} \bregdis[\regfct(\cdot, \regparam)]{p^{k - 1}}{u^k}{u^{k - 1}} \\
&- \left( \energy(u^k) - \energy(u^{k - 1}) - \langle \nabla \energy(u^{k - 1}), u^k - u^{k - 1} \rangle \right)
\end{split} \, ,
\end{align*}
for $q^{k - 1} \in \partial \regfct_\gamma(u, \regparam)$. As a direct consequence, we have derived the estimate
\begin{align}
\energy(u^k) + \langle \nabla \energy(u^{k - 1}), u^{k - 1} - u^k \rangle - \frac{L}{\gamma} \bregdis[\regfct(\cdot, \regparam)]{p^{k - 1}}{u^k}{u^{k - 1}} \leq \energy(u^{k - 1}) \, .\label{eq:suffdecreasestep1}
\end{align}
Inserting the dual update formula of \eqref{eq:nonconvlinbregman} into \eqref{eq:suffdecreasestep1} then yields
\begin{align*}
\energy(u^k) + \frac{1}{\regparam^{k - 1}} \symbreg[\regfct(\cdot, \regparam)]{u^k}{u^{k - 1}} - \frac{L}{\gamma} \bregdis[\regfct(\cdot, \regparam)]{p^{k - 1}}{u^k}{u^{k - 1}} \leq \energy(u^{k - 1}) \, .
\end{align*}
Together with the stepsize bound \eqref{eq:stepsizecond} we conclude \eqref{eq:suffdecrease}.
\end{proof}
\end{lemma}
\noindent If we further assume that $\regfct$ is $\delta$-strongly convex w.r.t. its first argument, i.e.
\begin{align*}
\frac{\delta}{2} \| p - q \|_{\domain^\ast}^2 \leq \bregdis[\regfct^\ast(\cdot, \regparam)]{v}{p}{q} \, ,
\end{align*}
for all $p, q \in \domain^\ast$ and $v \in \partial \regfct^\ast(q, \regparam)$, then we can easily derive the following bound for the gradient at each iteration \cite{benning2016gradient}.
\begin{lemma}\label{lem:nonconvlinbregmangradbound}
Let the same assumptions hold true as in Lemma \ref{lem:nonconvlinbregmansuffdecrease}, and further assume that $\regfct$ is $\delta$-strongly convex for all arguments and corresponding subgradients. Then the iterates \eqref{eq:nonconvlinbregman} satisfy
\begin{align*}
\| \nabla \energy(u^{k - 1}) \|_{\domain^\ast} \leq \frac{\regparam^{k - 1}}{\delta} \| u^k - u^{k - 1} \|_{\domain} \, ,
\end{align*}
for all $k \in \N$.
\begin{proof}
From the standard duality estimate $\langle p, u \rangle \leq \| u \|_{\domain} \| p \|_{\domain^\ast}$ we observe
\begin{align*}
\symbreg[\regfct(\cdot, \regparam)]{p^k}{p^{k - 1}} = \langle p^k - p^{k - 1}, u^k - u^{k - 1} \rangle \leq \| p^k - p^{k - 1} \|_{\domain^\ast} \| u^k - u^{k - 1} \|_{\domain} \, .
\end{align*}
Together with the strong convexity of $\regfct^\ast( \cdot, \regparam)$ we therefore estimate
\begin{align*}
\delta \| p^k - p^{k - 1} \|_{\domain^\ast} \leq \frac{\symbreg[\regfct(\cdot, \regparam)]{p^k}{p^{k - 1}}}{\| p^k - p^{k - 1} \|_{\domain^\ast}} \leq \| u^k - u^{k - 1} \|_{\domain} \, .
\end{align*}
Inserting the dual update formula from \eqref{eq:nonconvlinbregman} thus yields
\begin{align*}
\frac{\delta}{\regparam^{k - 1}} \| \nabla \energy(u^{k - 1}) \|_{\domain^\ast} \leq \| u^k - u^{k - 1} \|_{\domain} \, .
\end{align*}
This concludes the proof.
\end{proof}
\end{lemma}
Note that we require no Lipschitz-continuity assumptions for $\nabla \energy$ in order for Lemma \ref{lem:nonconvlinbregmansuffdecrease} and Lemma \ref{lem:nonconvlinbregmangradbound} to go through, but just that \eqref{eq:regfctgamma} is convex. As in the case of the proximal gradient method, we can prove global convergence of the iterates \eqref{eq:nonconvlinbregman} for finite-dimensional $\domain = \R^n$.

\begin{theorem}
Let the same assumptions hold true as in Lemma \ref{lem:nonconvlinbregmansuffdecrease}. Further assume $\domain = \R^n$ and that $\energy$ is a KL function. Then the iterates \eqref{eq:nonconvlinbregman} converge globally to a critical point of the energy $\energy$.
\begin{proof}
The proof is a special case of the more general proof of \cite[Theorem 5.6 \& Corollary 5.7]{2017arXiv171204045B}.
\end{proof}
\end{theorem}

We do want to emphasize that we require $\regfct^\ast(\cdot, \regparam)$ to be strongly convex, which in return implies the restrictive assumption that $\regfct( \cdot, \regparam)$ is a smooth functional with Lipschitz-continuous gradient. In order to get rid of this restrictive condition we split the functional $\regfct(\cdot, \regparam)$ into the two parts 
\begin{align*}
\regfct(u, \regparam) = H(u) + \frac{1}{\alpha^{k - 1}} G(u, \regparam) \, ,
\end{align*}
and assume that $H$ is $\gamma$-strongly convex and has $\delta$-Lipschitz gradient $\nabla H$, and that $G(\cdot, \regparam)$ is proper, l.s.c. and convex. Hence, we modify \eqref{eq:nonconvlinbregman} as follows. 
\begin{align}
\begin{split}
\regopitarg{u^{k - 1}, q^{k - 1}} &= \argmin_{u \in \domain} \left\{ \langle \nabla \energy(u^{k - 1}), u - u^{k - 1} \rangle + \bregdis[G(\cdot, \regparam)]{q^{k - 1}}{u}{u^{k - 1}} + \regparam^{k - 1} \bregdis[H]{}{u}{u^{k - 1}} \right\} \\
u^k &\in \regopitarg{u^{k - 1}, q^{k - 1}}\\
q^k &= q^{k - 1} - \left( \nabla \energy(u^{k - 1}) + \regparam^{k - 1} \left(\nabla H(u^k) - \nabla H(u^{k - 1}) \right) \right)
\end{split} \, ,\label{eq:nonconvlinbregmanmodified}
\end{align}
for $q^0 \in \partial G(u^0, \regparam)$. We then define the surrogate energy
\begin{align}
\energy^k(u^k) := \energy(u^k) + \bregdis[G( \cdot, \regparam)]{q^{k - 1}}{u^k}{u^{k - 1}} \, ,\label{eq:nonconvexsurrogate}
\end{align}
for $q^{k - 1} \in \partial G(u^{k - 1}, \regparam)$. For this surrogate energy we can show the following results. 
\begin{lemma}\label{lem:nonconvlinmodifiedsuffdecrease}
Suppose $\energy$ is coercive or has bounded level-sets, $\inf_u \energy(u) > -\infty$ and $\energy$ has at least one critical point, and assume $H$ is $\gamma$-strongly convex with $\delta$-Lipschitz gradient $\nabla H$, and $\regparam^{k - 1}$ satisfies \eqref{eq:stepsizecond}. Then the iterates of \eqref{eq:nonconvlinbregmanmodified} satisfy 
\begin{align*}
\energy^{k - 1}(u^k) + \rho \bregdis[H]{}{u^k}{u^{k - 1}} \leq \energy^{k - 2}(u^{k - 1}) \, .
\end{align*}
\begin{proof}
The proof follows the same principle as the proofs of Lemma \ref{lem:proxgradsuffdec} and Lemma \ref{lem:nonconvlinbregmansuffdecrease}. Convexity of $\frac{L}{\gamma} H - \energy$ implies the estimate in \eqref{eq:suffdecreaseproxgrad1}. Inserting the optimality condition (respectively the dual update formula) of \eqref{eq:nonconvlinbregmanmodified}, applying \eqref{eq:stepsizecond} and adding $\bregdis[G(\cdot, \regparam)]{q^{k - 2}}{u^{k - 1}}{u^{k - 2}}$ to both sides of the inequality then yields the desired estimate.
\end{proof}
\end{lemma}
A bound of the gradient of $\energy^{k - 1}(u^k)$ follows from the Lipschitz-continuity of both $\nabla \energy$ and $\nabla H$.
\begin{lemma}\label{lem:nonconvlinbregmanmodifiedgradbound}
Let the same assumptions hold true as in Lemma \ref{lem:nonconvlinmodifiedsuffdecrease}. Then the iterates \eqref{eq:nonconvlinbregmanmodified} satisfy
\begin{align*}
\| \nabla \energy(u^k) + q^k- q^{k - 1} \|_{\domain^\ast} \leq \left( L + \delta \regparam^{k - 1} \right) \| u^k - u^{k - 1} \|_{\domain} \, .
\end{align*}
\begin{proof}
Using the dual update formula \eqref{eq:nonconvlinbregmanmodified} and the Lipschitz-continuity of $\nabla E$ and $\nabla H$ leads to
\begin{align*}
\| \nabla \energy(u^k) + q^k- q^{k - 1} \|_{\domain^\ast} &= \left\| \nabla \energy(u^k) -  \nabla \energy(u^{k - 1}) + \regparam^{k - 1} \left( \nabla H(u^{k - 1}) - \nabla H(u^k) \right) \right\|_{\domain^\ast} \\
&\leq L \| u^k - u^{k - 1} \|_{\domain} + \regparam^{k - 1}\delta \| u^{k - 1} - u^k \|_{\domain} \, ,
\end{align*}
which proves the conjecture.
\end{proof}
\end{lemma}
As in the previous case, global convergence can be achieved under the assumption that the domain is finite-dimensional and that $\energy^k(u)$ is a KL-function.
\begin{theorem}\label{thm:nclbregmanglobalconv}
Let the same assumptions hold true as in Lemma \ref{lem:nonconvlinbregmanmodifiedgradbound}. Further assume $\domain = \R^n$ and that $\energy^k$ is a KL function for all $k \in \N$. Then the iterates \eqref{eq:nonconvlinbregmanmodified} converge globally. If, in addition, the sequence $\{ q^k \}_{k \in \N}$ is bounded, then the iterates even convergence to a critical point of the energy $\energy$.
\begin{proof}
The proof is a special case of the more general proof of The proof is a special case of the more general proof of \cite[Theorem 5.10]{2017arXiv171204045B}.
\end{proof}
\end{theorem}
\begin{remark}
Given the structure of the problem, it is tempting to also look at a Fej\'{e}r-monotonicity w.r.t. $\bregdis[\regfct_{\regparam^{k - 1}}( \cdot, \regparam)]{q^k}{\minsol}{u^k}$, for
\begin{align*}
\regfct_{\regparam^{k - 1}}(u, \regparam) := \regparam^{k - 1} \regfctarg[\regparam]{u} - \energy(u) \, .
\end{align*}
If we make the same attempt as in Section \ref{sec:iterreg}, we observe
\begin{align*}
\bregdis[\regfct_{\regparam^{k - 1}}( \cdot, \regparam)]{q^k}{\minsol}{u^k} - \bregdis[\regfct_{\regparam^{k - 1}}( \cdot, \regparam)]{q^{k - 1}}{\minsol}{u^{k - 1}} {} = {} &-\bregdis[\regfct_{\regparam^{k - 1}}( \cdot, \regparam)]{q^{k - 1}}{u^k}{u^{k - 1}} - \langle q^k - q^{k - 1}, \minsol - u^k \rangle\\
{} = {} &-\regparam^{k - 1} \bregdis[\regfct( \cdot, \regparam)]{p^{k - 1}}{u^k}{u^{k - 1}} - \left( \energy(u^k) - \energy(u^{k - 1}) \right.\\
&\left.- \langle \nabla E(u^{k - 1}), u^k - u^{k - 1} \rangle \right) \\
&- \regparam^{k - 1} \langle p^k - p^{k - 1}, \minsol - u^k \rangle\\
&+ \langle \nabla \energy(u^k) - \energy(u^{k - 1}), \minsol - u^k \rangle\\
{} = {} &-\regparam^{k - 1} \bregdis[\regfct( \cdot, \regparam)]{p^{k - 1}}{u^k}{u^{k - 1}} - \left( \energy(u^k) - \energy(u^{k - 1}) \right.\\
&\left.- \langle \nabla E(u^{k - 1}), u^k - u^{k - 1} \rangle \right) \\
&+ \langle \nabla \energy(u^k), \minsol - u^k \rangle \, .
\end{align*}
Since we also know that 
\begin{align*}
\langle \nabla \energy(u^k), \minsol - u^k \rangle = \bregdis[\regfct_{\regparam^{k - 1}}( \cdot, \regparam)]{q^k}{\minsol}{u^k} - \regparam^{k - 1} \bregdis[\regfct( \cdot, \regparam)]{p^k}{\minsol}{u^k} + \energy(\minsol) - \energy(u^k) \, ,
\end{align*}
we can combine this equality with the previous to obtain
\begin{align*}
\bregdis[\regfct_{\regparam^{k - 1}}( \cdot, \regparam)]{q^k}{\minsol}{u^k} - \bregdis[\regfct_{\regparam^{k - 1}}( \cdot, \regparam)]{q^{k - 1}}{\minsol}{u^{k - 1}} {} = {} &-\bregdis[\regfct_{\regparam^{k - 1}}( \cdot, \regparam)]{q^{k - 1}}{u^k}{u^{k - 1}} \\
{} = {} &+\bregdis[\regfct_{\regparam^{k - 1}}( \cdot, \regparam)]{q^k}{\minsol}{u^k} - \regparam^{k - 1} \bregdis[\regfct( \cdot, \regparam)]{q^k}{\minsol}{u^k}\\
&+ \energy(\minsol) - \energy(u^k) \, .
\end{align*}
Hence, for $\energy(\minsol) \leq \energy(u^k)$ we only observe
\begin{align*}
\regparam^{k - 1} \bregdis[\regfct( \cdot, \regparam)]{q^k}{\minsol}{u^k} \leq \bregdis[\regfct_{\regparam^{k - 1}}( \cdot, \regparam)]{q^{k - 1}}{\minsol}{u^{k - 1}} \, ,
\end{align*}
which is not quite sufficient to achieve Fej\'{e}r-monotonicity.
\end{remark}

We mention that non-convex data fidelities find applications in problems with advanced noise models, e.g. multiplicative noise
(cf. \cite{rudin2003multiplicative,aubert2008variational}), image registration problems (cf. \cite{Modersitzki,}), or most nonlinear inverse problems. 
In the next subsection we focus on the special case of $\energy$ representing a convex data fidelity $\fidelfct$ of a potentially nonlinear inverse problem, which leads to an overall non-convex problem.

Let us mention that so far no suitable theory of iterative regularization methods in the case of non-convex regularizations is available, although there are several applications such as the Mumford-Shah or Ambrosio-Tortorelli functional (cf. \cite{mumford1989optimal,ambrosio1990approximation,Pock:MinimizingMumfordShah,rondi2008reconstruction,klann2011mumford,klann2013regularization}) or polyconvex energies in image registration (cf. \cite{droske2003nonrigid,burger2013hyperelastic,kirisits2017convergence}).  

\subsection{Nonlinear Inverse Problems}\label{subsec:nonlininvprob}
Nonlinear inverse problems are extensions of \eqref{eq:basicequation} with nonlinear forward operators $K:\domain \rightarrow \range$. Given a convex or nonconvex data fidelity term $\fidelfct:\range \times \range \rightarrow \R$, we can formulate variational regularizations and iterative regularizations in the exact same way as in the linear case. As these problems are special cases of the nonconvex methodology discussed in Section \ref{subsec:nonconvreg}, we can further apply the proposed methodologies. In the context of variational regularization \eqref{eq:varreg} for nonlinear forward operators and possibly nonconvex but Fr\'{e}chet-differentiable fidelity terms, the $k$-th iterate of the proximal gradient method discussed in Section \ref{subsubsec:proxgrad} reads as
\begin{align*}
\regopitarg{\noisy, u^{k - 1}} = \argmin_{u \in \domain} \left\{  \regparam^{k - 1} \langle \partial_x \fidelity{K(u^{k - 1})}{\noisy}, u - u^{k - 1} \rangle + D_H(u, u^{k - 1}) + \regparam^{k - 1} \regfct(u, \regparam) \right\} \, .
\end{align*}
The convergence theory discussed in Section \ref{subsubsec:proxgrad} applies in identical fashion. However, questions of the convergence of the regularization can now also be addressed.

\subsubsection*{Gau\ss-Newton Methods}
The special structure of the nonconvex energy functional $E$ in case of regularizations of nonlinear inverse problems enables different solution strategies compared to arbitrary nonconvex functionals. Having a Fr\'{e}chet-differentiable operator $K$, one can approximate $K(u^k)$ via a Taylor-approximation around $u^{k - 1}$, i.e.
\begin{align*}
K(u^k) \approx K(u^{k - 1}) + K^\prime(u^{k - 1})(u^k - u^{k - 1}) \, .
\end{align*}
As a consequence, another strategy for solving variational regularization problems for nonlinear inverse problems is via the following iteratively regularized Gau\ss-Newton approach
\begin{align}
\regopitarg{\noisy, u^{k - 1}} = \argmin_{u \in \domain} \left\{ \fidelity{K(u^{k - 1}) + K^\prime(u^{k - 1})(u - u^{k - 1})}{\noisy} + \regparam^{k - 1} \regfct(u, \regparam) \right\} \, .
\end{align}
 We refer to 
\cite{schopfer2006nonlinear,stuck2011iteratively,bauer2009iteratively,kaltenbacher2009iterative,stuck2011iteratively,hohage2013iteratively} for further discussion

In the following sections we discuss extensions of the iterative regularization methods presented in Section \ref{sec:iterreg} to nonlinear inverse problems.

\subsubsection{Nonlinear Landweber Regularization}\label{sec:nonlinlandweber}
We easily observe that \eqref{eq:nonconvlinbregman} for $\energy(u) := \fidelity{K(u)}{\noisy}$ with nonlinear operator $K$ reads as 
\begin{align*}
\begin{split}
\regopitarg{\noisy, v^{k - 1}} &= \argmin_{u \in \domain} \left\{ \langle K^\prime(u^{k - 1})^\ast \partial_x \fidelity{K(u^{k - 1})}{\noisy}, u - u^{k - 1} \rangle + \regparam^{k - 1}  \bregdis[\regfct(\cdot, \regparam)]{p^{k - 1}}{u}{u^{k - 1}} \right\} \\
u^k &\in \regopitarg{\noisy, v^{k - 1}} \\
p^k &= p^{k - 1} - \frac{1}{\regparam^{k - 1}} K^\prime(u^{k - 1})^\ast \partial_x  \fidelity{K(u^{k - 1})}{\noisy}
\end{split} \, ,
\end{align*}
with $v^{k - 1} := (u^{k - 1}, p^{k - 1})$. For $\fidelity{K(u)}{\noisy} = \frac{1}{2} \| K(u) - \noisy \|_{L^2(\Sigma)}^2$ and $\regfct(u, \regparam) = \frac{\regparam}{p} \| u \|_{L^p(\Omega)}^p$ this method has first been introduced and analyzed in \cite{kaltenbacher2009iterative}. General convex regularization functionals $\regfct(\cdot, \regparam)$ with multi-valued subdifferential $\partial \regfct(\cdot, \regparam)$ have been considered in \cite{bachmayr2009iterative}. Both convergence analyses have been carried under additional assumptions on the nonlinear forward operator, such as the tangential cone condition. In a finite dimensional setting, convergence follows from Theorem \ref{thm:nclbregmanglobalconv}, see \cite{2017arXiv171204045B}. However, it is important to point out that, although existence of a critical point of $\energy(u)$ can usually be guaranteed in finite dimensions, ill-conditioning of the problem still requires early stopping of the iterates.

\subsubsection{Levenberg-Marquardt Regularization}
Replacing the regularization functional in the iterative Gau\ss-Newton regularization with a generalized Bregman distance w.r.t. the current and the previous iterate yields the following generalized Levenberg-Marquardt regularization
\begin{align*}
\begin{split}
\regopitarg{\noisy, v^{k - 1}} &= \argmin_{u \in \domain} \left\{ \fidelity{K(u^{k - 1}) + K^\prime(u^{k - 1})(u - u^{k - 1})}{\noisy} + \regparam^{k - 1} \bregdis[\regfct(\cdot, \regparam)]{p^{k - 1}}{u}{u^{k - 1}} \right\} \\
u^k &\in \regopitarg{\noisy, v^{k - 1}} \\
p^k &= p^{k - 1} - \frac{1}{\regparam^{k - 1}} K^\prime(u^{k - 1})^\ast \partial_x  \fidelity{K(u^{k - 1}) + K^\prime(u^{k - 1})(u^k - u^{k - 1})}{\noisy}
\end{split} \, ,
\end{align*}
for $v^{k - 1} := (u^{k - 1}, p^{k - 1})$. This method reduces to the classical Levenberg-Marquardt method \cite{levenberg1944method,marquardt1963algorithm} for the choices $\fidelity{K(u)}{\noisy} = \frac{1}{2}\| K(u) - \noisy \|_{L^2(\Sigma)}^2$ and $\regfct(u, 1) = \frac{1}{2}\| u \|_{L^2(\Omega)}^2$. For $\fidelity{K(u)}{\noisy} = \frac{1}{2}\| K(u) - \noisy \|_{L^2(\Sigma)}^2$ and proper, l.s.c. and convex $\regfct(u, \regparam)$ with potentially multi-valued subdifferential $\partial \regfct(u, \regparam)$ this method has been introduced and analyzed in \cite{bachmayr2009iterative}.

\subsubsection{Examples}
In the following we discuss two nonlinear inverse problems that are natural extensions of the linear inverse problems introduced in Example \ref{exm:deconvolution} and Section \ref{sec:velmri}.

\begin{figure}[!t]
\begin{center}
\subfloat[Original $\minsol$]{\includegraphics[width=0.24\textwidth]{Images/Pixelsmall.png}\label{subfig:pixelnclbregiter1}}\vspace{0.05cm}
\subfloat[Blurred \& noisy $\noisy$]{\begin{tikzpicture}[      
        every node/.style={anchor=south west,inner sep=0pt},
        x=1mm, y=1mm,
      ]   
     \node (fig1) at (0,0)
       {\includegraphics[width=0.24\textwidth]{Images/Pixelblurrednoisy.png}};
     \node (fig2) at (0,0)
       {\includegraphics[width=0.1\textwidth]{Images/convkernel.png}};  
\end{tikzpicture}\label{subfig:pixelnclbregiter2}}\vspace{0.05cm}
\subfloat[1st iterate]{\begin{tikzpicture}[      
        every node/.style={anchor=south west,inner sep=0pt},
        x=1mm, y=1mm,
      ]   
     \node (fig1) at (0,0)
       {\includegraphics[width=0.24\textwidth]{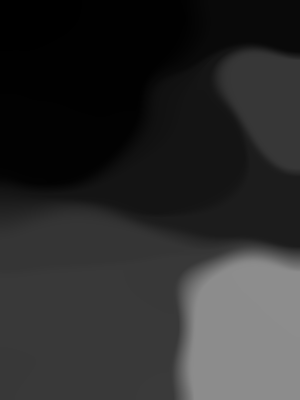}};
     \node (fig2) at (0,0)
       {\includegraphics[width=0.1\textwidth]{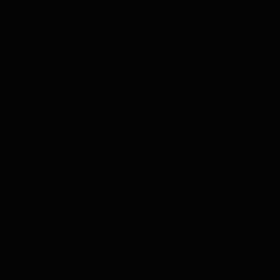}};  
\end{tikzpicture}\label{subfig:pixelnclbregiter3}}\vspace{0.05cm}
\subfloat[11th iterate]{\begin{tikzpicture}[      
        every node/.style={anchor=south west,inner sep=0pt},
        x=1mm, y=1mm,
      ]   
     \node (fig1) at (0,0)
       {\includegraphics[width=0.24\textwidth]{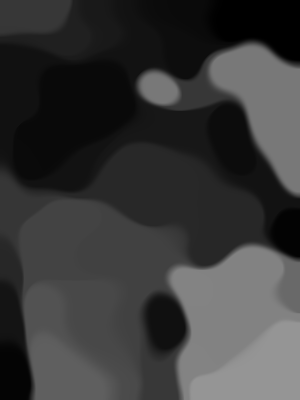}};
     \node (fig2) at (0,0)
       {\includegraphics[trim={2.5cm 2.5cm 2.5cm 2.5cm},clip,width=0.1\textwidth]{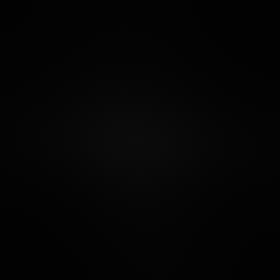}};  
\end{tikzpicture}}\vspace{0.05cm}\\
\subfloat[49th iterate]{\begin{tikzpicture}[      
        every node/.style={anchor=south west,inner sep=0pt},
        x=1mm, y=1mm,
      ]   
     \node (fig1) at (0,0)
       {\includegraphics[width=0.24\textwidth]{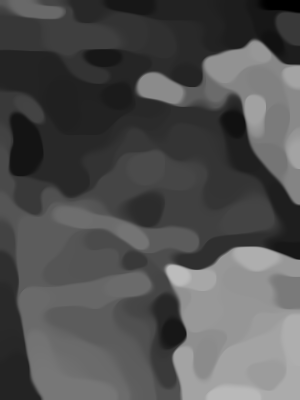}};
     \node (fig2) at (0,0)
       {\includegraphics[trim={2.5cm 2.5cm 2.5cm 2.5cm},clip,width=0.1\textwidth]{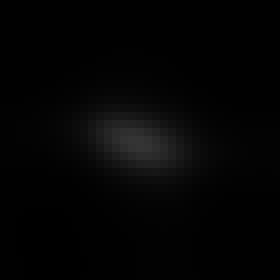}};  
\end{tikzpicture}}\vspace{0.05cm}
\subfloat[999th iterate]{\begin{tikzpicture}[      
        every node/.style={anchor=south west,inner sep=0pt},
        x=1mm, y=1mm,
      ]   
     \node (fig1) at (0,0)
       {\includegraphics[width=0.24\textwidth]{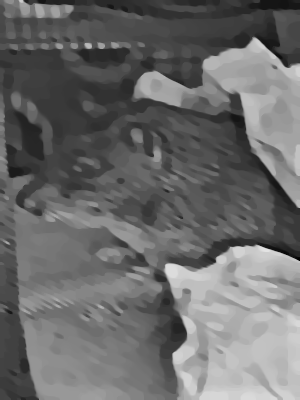}};
     \node (fig2) at (0,0)
       {\includegraphics[trim={2.5cm 2.5cm 2.5cm 2.5cm},clip,width=0.1\textwidth]{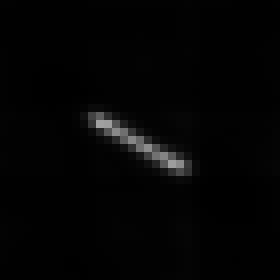}};  
\end{tikzpicture}}\vspace{0.05cm}
\subfloat[2199th iterate]{\begin{tikzpicture}[      
        every node/.style={anchor=south west,inner sep=0pt},
        x=1mm, y=1mm,
      ]   
     \node (fig1) at (0,0)
       {\includegraphics[width=0.24\textwidth]{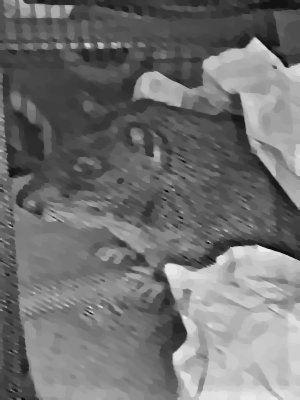}};
     \node (fig2) at (0,0)
       {\includegraphics[trim={2.5cm 2.5cm 2.5cm 2.5cm},clip,width=0.1\textwidth]{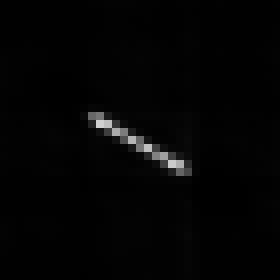}};  
\end{tikzpicture}}\vspace{0.05cm}
\subfloat[3474th iterate]{\begin{tikzpicture}[      
        every node/.style={anchor=south west,inner sep=0pt},
        x=1mm, y=1mm,
      ]   
     \node (fig1) at (0,0)
       {\includegraphics[width=0.24\textwidth]{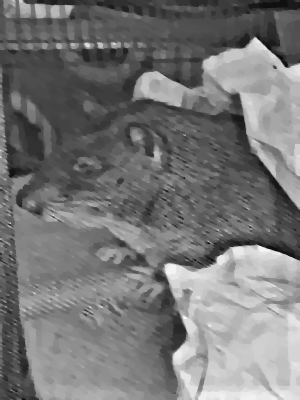}};
     \node (fig2) at (0,0)
       {\includegraphics[trim={2.5cm 2.5cm 2.5cm 2.5cm},clip,width=0.1\textwidth]{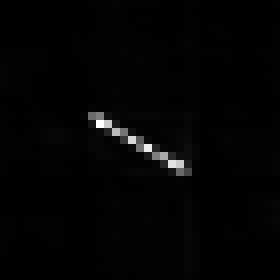}};  
\end{tikzpicture}\label{subfig:pixelnclbregiter4}}\vspace{0.05cm}
\end{center}
\caption{Figure \ref{subfig:pixelnclbregiter1} shows the image $\minsol \in \R^{400 \times 300}$ of Pixel, the Gambian pouched rat, originally introduced in Figure \ref{subfig:pixelbregiter1}. In Figure \ref{subfig:pixelnclbregiter2} we see the same degraded and noisy version $\noisy \in \R^{400 \times 300}$ together with the convolution kernel $h$ as shown in Figure \ref{subfig:pixelbregiter2} and Figure \ref{subfig:pixellinbregiter2}. Figure \ref{subfig:pixelnclbregiter3} - \ref{subfig:pixelnclbregiter4} show different iterates of Algorithm \ref{alg:linbregiter} for $\fidelity{K(u, h)}{\noisy}$ and $\regfct(u, h)$ as in Equation \eqref{eq:blinddeconvdatafidel} and Equation \eqref{eq:blinddeconvregfct}, respectively and $\regparam = 10$. The 3474th iterate visualized in Figure \ref{subfig:pixelnclbregiter4} is the first that violates Definition \ref{def:morozov}, for $\delta = 5.95$. The reconstructed kernels have been magnified for better visualization.}\label{fig:pixelnclbregiter}
\end{figure}

\subsubsection*{Blind deconvolution}
Following up on Example \ref{exm:deconvolution}, an obvious nonconvex extension of the problem of deconvolution is blind deconvolution, where the convolution kernel that degrades the image is also unknown (cf. \cite{Kundur1996blinddeblurring,Chan2005book,Campisi2016book}). We basically follow the setup of \cite[Section 6.2]{2017arXiv171204045B}, where we assume 
\begin{align}
\begin{split}
\fidelity{K(u, h)}{\noisy} &= \frac{1}{2} \| K(u, h) - \noisy \|_{L^2(\R^2)}^2 \\
&= \frac{1}{2} \| u \ast h - \noisy \|_{L^2(\R^2)}^2
\end{split}\label{eq:blinddeconvdatafidel}
\end{align}
and apply the nonlinear Landweber regularization as described in Section \ref{sec:nonlinlandweber} with
\begin{align}
\regfct(u, h, \regparam) = \frac{1}{2} \| u \|_{L^2(\R^2)}^2 + \regparam \tv(u) + \int_{\R^2} h(x) \log(h(x)) - h(x) \, dx + \chi_{P(\R^2)}(h) \, ,\label{eq:blinddeconvregfct}
\end{align}
where
\begin{align*}
\chi_{P(\R^2)}(h) = \begin{cases}
0 & h \in P(\R^2) \\ \infty & h \not\in P(\R^2)
\end{cases}
\end{align*}
denotes the characteristic functional over the (convex) set of probability distributions
\begin{align*}
P(\R^2) := \left\{ h \in L^2(\R^2) \, \left| \, h(x) \geq 0 \, a.e., \, \int_{\R^2} h(x) \, dx = 1 \right. \right\} \, .
\end{align*}
The rationale behind this choice of $\regfct$ is that convolution kernels in applications such as motion deblurring are usually non-negative and preserve the mean of the underlying signal. We refer to \cite[Section 6.2 \& Section 7.2]{2017arXiv171204045B} for more information on the discrete formulation of the problem and its numerical realization. 

\begin{figure}[!t]
\begin{center}
\subfloat[]{\includegraphics[width=0.24\textwidth]{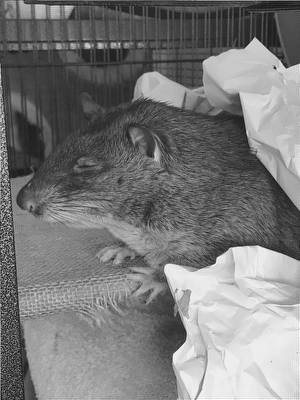}\label{subfig:deconvcomparison1}}\vspace{0.05cm}
\subfloat[]{\includegraphics[width=0.24\textwidth]{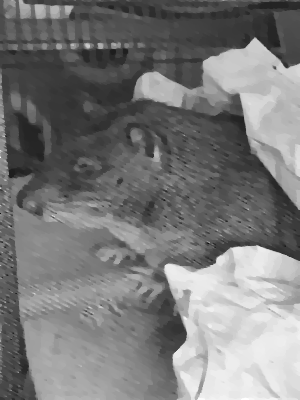}\label{subfig:deconvcomparison2}}\vspace{0.05cm}
\subfloat[]{\includegraphics[width=0.24\textwidth]{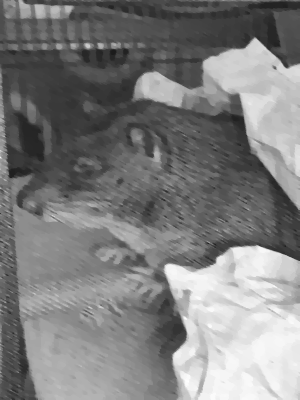}\label{subfig:deconvcomparison3}}\vspace{0.05cm}
\subfloat[]{\includegraphics[width=0.24\textwidth]{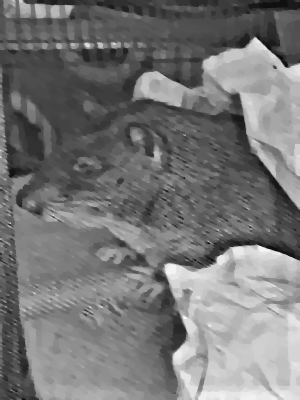}\label{subfig:deconvcomparison4}}
\end{center}
\caption{Deconvolution results for the image of Pixel, the Gambian pouched rat. Figure \ref{subfig:deconvcomparison1}: the original image. Figure \ref{subfig:deconvcomparison2}: the reconstruction discussed in Example \ref{exm:deconvolution}. Figure \ref{subfig:deconvcomparison3}: the reconstruction with Algorithm \ref{alg:linbregiter}. Figure \ref{subfig:deconvcomparison4}: the blind deconvolution result computed with the nonlinear Landweber regularization.}\label{fig:deconvcomparison}
\end{figure}

We use $\minsol$ and $\noisy$ from Example \ref{exm:deconvolution}, and therefore stop the nonlinear Landweber regularization via discrepancy principle for $\delta = 5.95$. The parameter $\regparam$, however, is chosen to be $\regparam = 10$ and is therefore much larger than in Example \ref{exm:deconvolution} and in Section \ref{subsec:linbregiter}. Hence, we require many more iterations in order to reach the same discrepancy. The necessity for this large choice of $\regparam$ stems from the fact that the iterates otherwise converge to unstable solutions with Dirac-delta-like convolution kernels. Several iterates of the nonlinear Landweber regularization are visualized in Figure \ref{fig:pixelnclbregiter}.

To conclude we visually compare the first iterates that violate the discrepancy principle of the Bregman iteration, the linearized Bregman iteration and the nonlinear Landweber regularization in Figure \ref{fig:deconvcomparison}. Between the reconstructions from the Bregman iteration and the linearized Bregman iteration there are at best small differences in contrast. The reconstruction from the the nonlinear Landweber regularization does have slight artifacts that originate from small imperfections in the reconstructed convolution kernel. Nevertheless, the result is still remarkable given that both image and convolution kernel were unknown and had to both be estimated.

\begin{figure}[!t]
\begin{center}
\subfloat[Fully-sampled]{\includegraphics[trim={3cm 1cm 2.2cm 1cm},clip,width=0.24\textwidth]{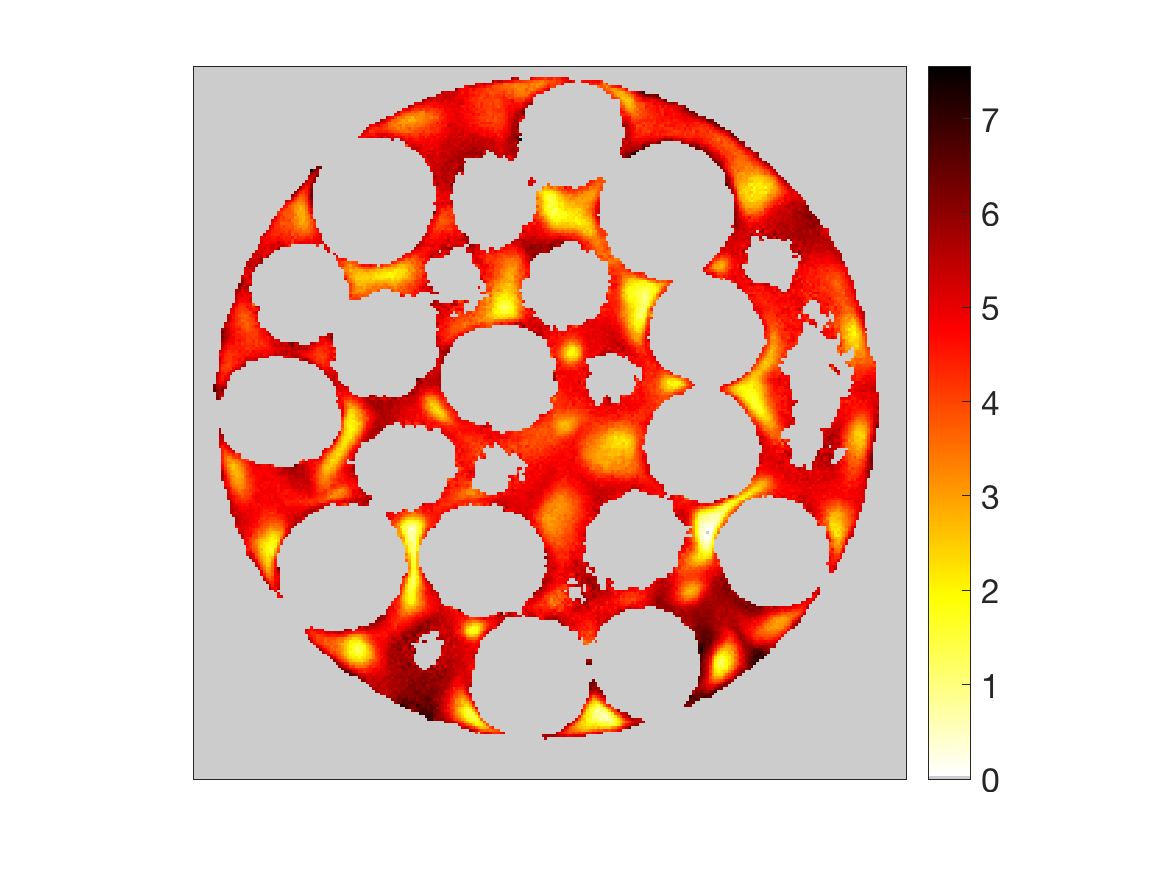}\label{subfig:nonlinvelmri1}}\vspace{0.05cm}
\subfloat[Zero-filled]{\includegraphics[trim={3cm 1cm 2.2cm 1cm},clip,width=0.24\textwidth]{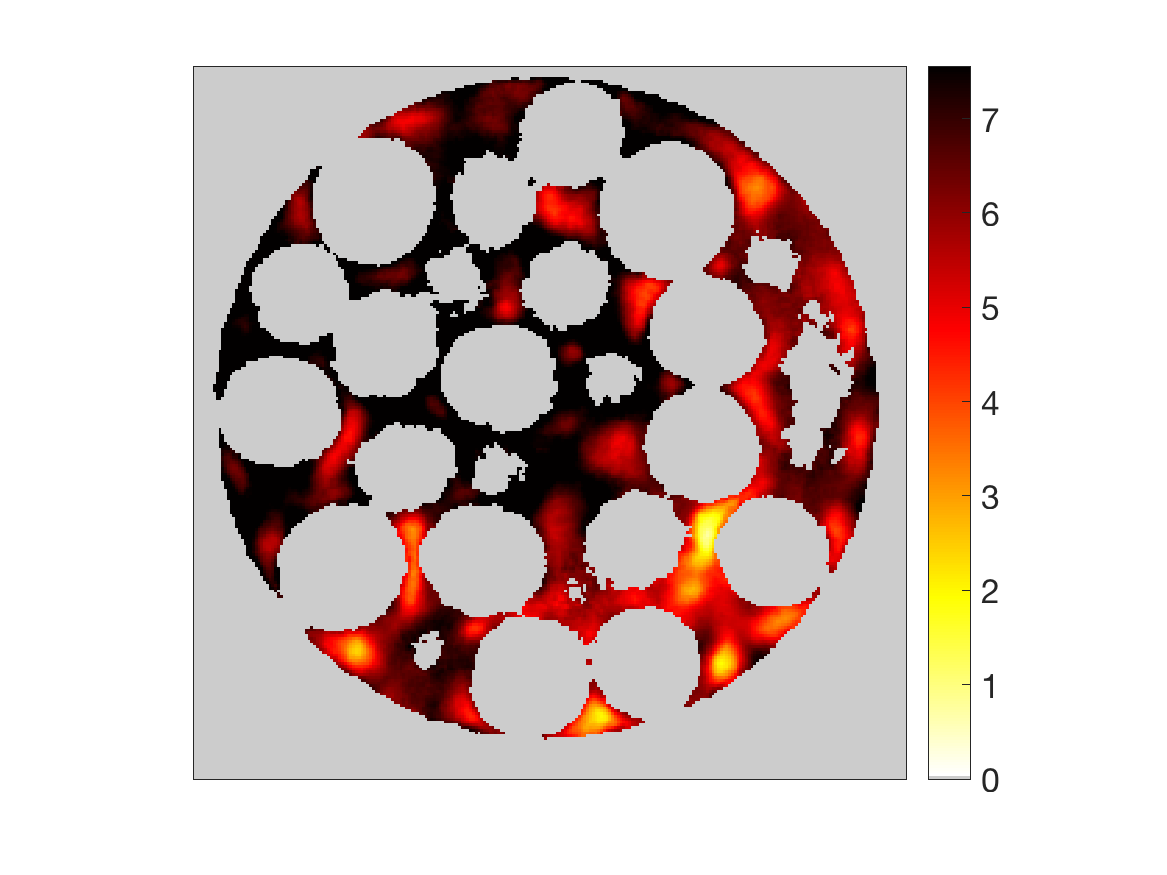}\label{subfig:nonlinvelmri2}}\vspace{0.05cm}
\subfloat[TGV recon.]{\includegraphics[trim={3cm 1cm 2.2cm 1cm},clip,width=0.24\textwidth]{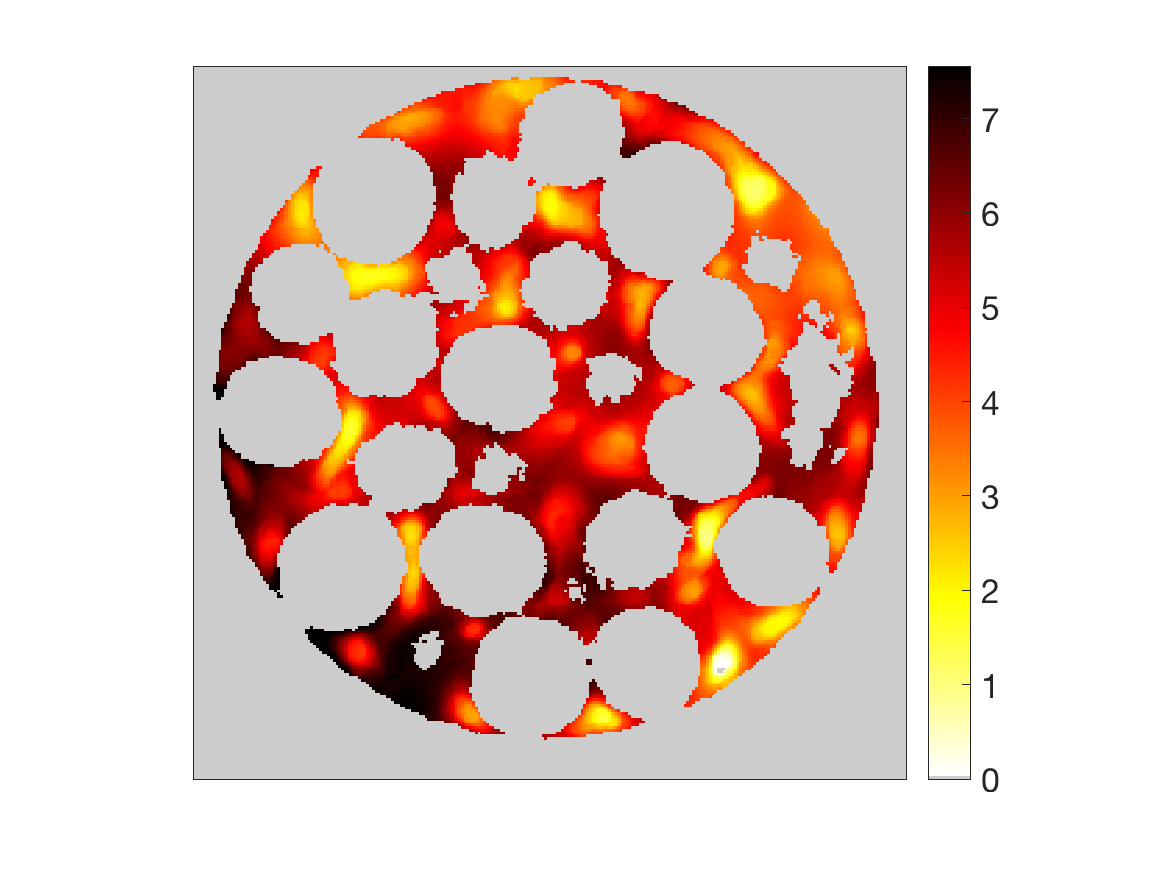}\label{subfig:nonlinvelmri3}}\vspace{0.05cm}
\subfloat[Nonlin. recon.]{\includegraphics[trim={3cm 1cm 2.2cm 1cm},clip,width=0.24\textwidth]{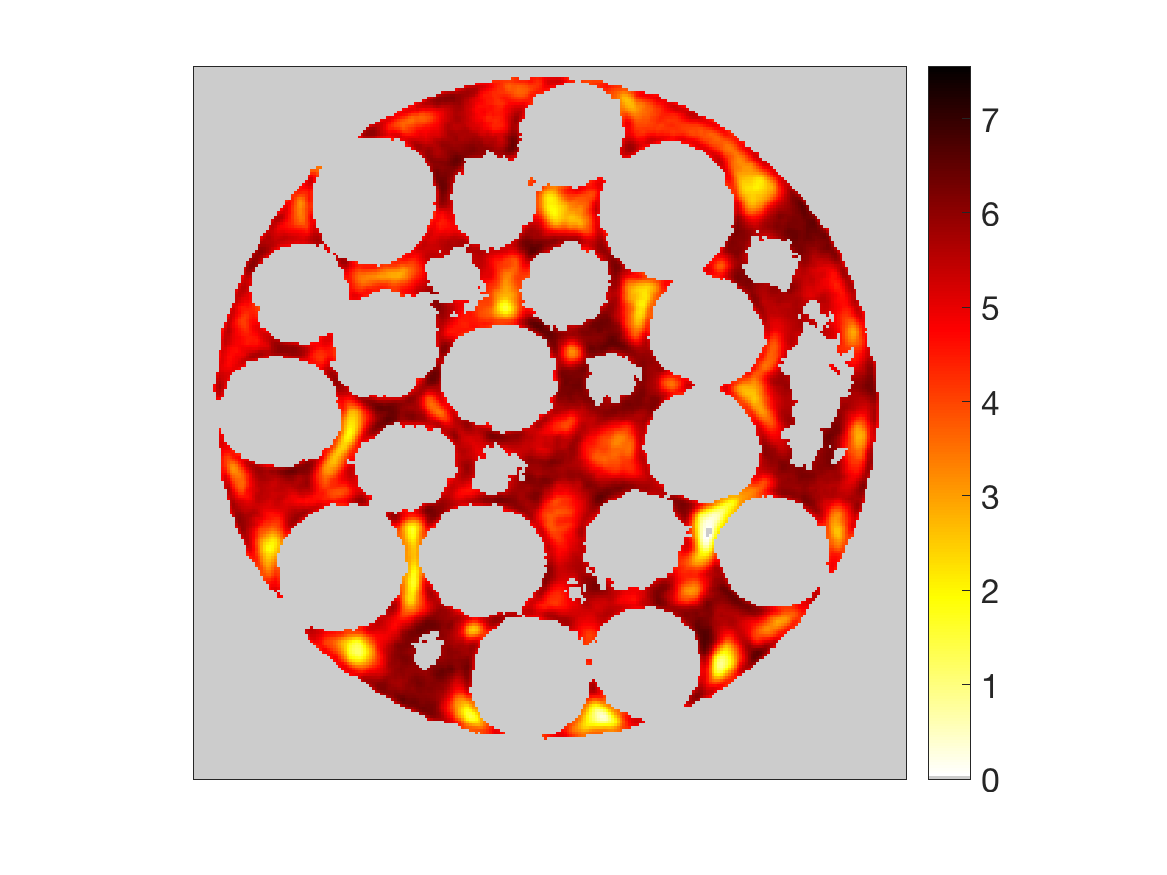}\label{subfig:nonlinvelmri4}}
\end{center}
\caption{Comparison of the different $z$-velocity reconstructions. Figure \ref{subfig:nonlinvelmri1}: the unwrapped velocity reconstruction from fully-sampled data. Figure \ref{subfig:nonlinvelmri2}: the unwrapped velocity reconstruction obtained from filling the missing samples of the sub-sampled data with zero. Figure \ref{subfig:nonlinvelmri3}: the unwrapped TGV-based reconstruction of the velocity from sub-sampled data. Figure \ref{subfig:nonlinvelmri4}: a nonlinear reconstruction of the velocity, computed via the nonlinear Landweber regularization.}\label{fig:nonlinvelmri}
\end{figure}

\subsubsection*{Velocity-encoded MRI}
We briefly revisit the velocity-encoded MRI problem of Section \ref{sec:velmri}. As the original forward problem \eqref{eq:velmriforward} is nonlinear, it is perfectly sensible to recover $v_z$ directly (instead of taking a detour via $w = u \exp(-i \sigma v_z)$). This idea is not new and has for instance already been addressed in \cite{zhao2012separate}. We again use the nonlinear Landweber regularization with the functionals
\begin{align*}
\fidelity{K(v_z)}{\noisy} = \frac{1}{2} \sum_{t = t_0}^{t_m} \left( \mathcal{F}(u, v_z) - \noisy_t \right)^2 \, ,
\end{align*}
where $u$ is a precomputed spin-proton density, and a scaled $H^1$-norm
\begin{align*}
\regfct(v_z, \regparam) = \frac{1}{2} \| v_z \|_{L^2(\R^2)}^2 + \frac{\regparam}{2} \| \nabla v_z \|_{L^2(\R^2)}^2
\end{align*}
as the regularization functional of choice. 

Figure \ref{fig:nonlinvelmri} shows the comparison of the velocity reconstruction from the fully-sampled data (Fig. \ref{subfig:nonlinvelmri1}), the zero-filled reconstruction from the sub-sampled data (Fig. \ref{subfig:nonlinvelmri2}), the TGV-based reconstruction from the sub-sampled data (Fig. \ref{subfig:nonlinvelmri3}) and a reconstruction from the sub-sampled data via the nonlinear Landweber regularization (Fig. \ref{subfig:nonlinvelmri4}) all clipped to the same intensity range. The latter has been initialized with $u^0(x) = \pi$ (on some compact domain), $p^0 = u^0$ and $\regparam = 200$. The result shown in Figure \ref{subfig:nonlinvelmri4} is the first iterate that violates the. discrepancy principle for $\eta = 1$ and $\delta = 80$. The inner subproblem has again been computed with the PDHGM.

\subsection{Learning} 
A very important question that always pops up when dealing with regularization of inverse problems is the question of how to choose the (regularization) parameters, respectively how to develop a useful parameter choice strategy. For the iterative regularization strategies discussed in Section \ref{sec:iterreg} we have used Morozov's discrepancy principle as an a-posteriori parameter choice rule to determine when to stop the iteration (which is the regularization parameter in case of iterative regularizations), based on the noisy data $\noisy$ and the noise level $\delta$. 
In addition to the standard alternatives, which are a-priori and heuristic parameter choice rules, supervised learning strategies have become popular in recent years. The idea is to choose optimal parameters based on pairs $\{ (\minsol_j, \noisy_j) \}_{j = 1}^m$ of training data by minimizing an empirical risk functional, which is just the empirical expectation of the loss between $\minsol_j$ and a $\regsol_j$ that can be obtained with data $\noisy_j$. 
Based on the previous notation of regularization operators, a relatively generic approach is to estimate optimal parameters $\hat\regparambold \in \regdomain$ via 
\begin{align}
\hat\regparambold \in \argmin_{\regparambold \in \regdomain} \left\{ \frac{1}m \sum_{j = 1}^m \ell_j(\minsol_j, \regsol_j) + \regfct(\regparambold, \betabold) \, \, \,  \text{subject to} \, \, \, \regsol_j \in \regoparg{\noisy_j} \, , \, \forall j \in \{1, \ldots, m\} \right\} \, .\label{eq:paramlearning}
\end{align}
Here $\{ \ell_j \}_{j = 1}^m$, with $\ell_j:\domain \times \domain \rightarrow \R$ for all $j \in \{1, \ldots, m\}$, denotes a family of loss functionals that measures the deviation between the reconstructions $\regsol_j$ and the ground truth signals $\minsol_j$, and $\regfct:\regdomain \times B \rightarrow \R$ is a regularization functional that, together with some parameters $\betabold$ in some parameter domain $B$, incorporates prior knowledge to steer the reconstruction of $\hat\regparambold$ into a certain direction. The operator $\regop:\range \times A \rightrightarrows \domain$ is a regularization operator that takes $\noisy_j$ and $\regparambold$ as an input and produces at least one reconstruction $\regsol_j$ as its output. If $\regsol_j \in \regoparg{\noisy_j}$ stems from an optimization problem, then \eqref{eq:paramlearning} is also known as a bilevel optimization problem \cite{doi:10.1137/120882706,reyes2013image}. It is also quite evident that \eqref{eq:paramlearning} is a regularization problem in itself. An even more generic way to formulate parameter learning would therefore be
\begin{align*}
\hat\regparambold \in P(\{ \minsol_j \}_{j = 1}^m, \{ \noisy_j \}_{j = 1}^m, \betabold) \, ,
\end{align*}
where $P: \domain^m \times \range^m \times B \rightrightarrows A$ is a regularization operator that also depends on some other regularization operator $\regop:\range \times \regdomain \rightrightarrows \domain$. A likely application of this scenario is supervised machine learning with early-stopping of, for instance, stochastic gradient descent methods (see \cite{Johnson2013,Defazio2014,Bertsekas2011a}). However, \eqref{eq:paramlearning} is sufficient to explain the majority of current state-of-the-art parameter learning approaches in the context of inverse problems. These cover the finite-dimensional Markov random field models proposed in \cite{roth2005fields,tappen2007utilizing,domke2012generic,chen2014insights,schmidt2014shrinkage}, the optimal model design approaches in \cite{haber2003learning,haber2009numerical,bui2008model,biegler2011large}, the optimal regularization parameter estimation in variational regularization \cite{calatroni2013dynamic,chung2014optimal,de2016structure,de2017bilevel,calatroni2017infimal,chung2016learning}, to training optimal operators in regularization functionals \cite{chen2013revisiting,chen2014learning}, reaction diffusion process \cite{chen2015learning,chen2017trainable}, so-called variational networks \cite{hammernik2017learning,kobler2017variational,klatzer2017trainable} and other works related to image processing \cite{ochs2015bilevel,hintermuller2015bilevel}.

In the following, we want to focus in particular on the connection between modern deep neural network approaches and iterative regularization methods as discussed in Section \ref{sec:iterreg}.
 
\subsubsection{Iterative Regularization and Deep Neural Networks}
In this section we discuss how certain (deep) neural network architectures are closely related (or even equivalent) to the linearized Bregman iteration described in Section \ref{subsec:linbregiter}, for a data fidelity term with variable metric. This connection will give insight into how more stable neural network architectures can be learned. For an overview on deep learning and neural network architectures we refer to \cite{lecun2015deep}.

We make the assumption that the data fidelity is given in terms of $\fidelity{Ku}{\noisy} = \frac{1}{2} \| Ku - \noisy \|_{Q_k}^2$, for $\| \cdot \|_{Q_k} := \sqrt{\langle Q_k \cdot, \cdot \rangle}$ and some positive definite matrix $Q_k$. We now aim to minimize this data fidelity with the help of Algorithm \ref{alg:linbregiter}, but deviate from the standard procedure by allowing the underlying positive definite matrix $Q_k$ to vary throughout the iterations.

If we reformulate Algorithm \ref{alg:linbregiter} for this particular choice of variable metric data fidelity and linearize around the previous iterate we obtain the following modification of Algorithm \ref{alg:linbregiter}:
\begin{align}
\begin{split}
\regopitarg[\regparambold^{k - 1}]{\noisy, v^{k - 1}} &= \argmin_{u \in \domain} \left\{ \langle K^\ast Q_{k - 1}(Ku^k - \noisy), u \rangle + \bregdis[{\regfctarg[\regparam]{\cdot}}]{p^{k - 1}}{u}{u^{k - 1}} \right\} \\
u^k &\in \regopitarg[\regparambold^{k - 1}]{\noisy, v^{k - 1}}\\
p^k &= p^{k - 1} - K^{\ast} Q_{k - 1}(Ku^{k - 1} - \noisy)
\end{split} \, .\label{eq:neuralnetworklinbregman}
\end{align}
Here we define $\regparambold^{k - 1} = (\regparam, Q_0, Q_1, \ldots, Q_{k - 1})$ and $v^{k - 1} = (u^{k - 1}, p^{k - 1})$. If we now choose $\regfct$ to be of the form $\regfctarg[\regparam]{u} = \frac{1}{2} \| u \|_{L^2(\Omega)}^2 + H(u, \regparam)$, the algorithm simplifies to
\begin{align*}
\begin{split}
u^k &= \prox[H(\cdot, \regparam)]\left((I - K^\ast Q_{k - 1} K)u^{k - 1} + K^\ast Q_{k - 1} \noisy + q^{k - 1} \right)\\
q^k &= u^{k - 1} - u^k + q^{k - 1} - K^{\ast} Q_{k - 1}(Ku^k - \noisy)
\end{split} \, ,
\end{align*}
for $q^k \in \partial H(u^k, \regparam )$, for all $k \in \N$. Here $\prox[H(\cdot, \regparam)]$ denotes the proximal mapping of $H$. If we define $A_k := I - K^\ast Q_k K$ and $b^k := K^\ast Q_k \noisy + q^k$ for all $k \in \N$, and choose $H$ to be the point-wise characteristic functional over the convex set of non-negative real numbers, i.e.
\begin{equation*}
(H(u, \regparam))(x) = (\chi_{\geq 0}(u))(x) = \begin{cases}
0 & u(x) \geq 0 \\ \infty & \text{else} 
\end{cases} \, ,
\end{equation*}
we obtain the standard ReLU neural network architecture
\begin{align*}
u^k = \max\left(0, A_{k - 1} u^{k - 1} + b^{k - 1} \right)
\end{align*}
for the primal update. However, rather than stopping at this analogy, we want to discuss how the insights of Section \ref{subsec:linbregiter} can help to impose rather natural conditions on the learning of the parameters $A_k$ and $b_k$.

Naturally, $A_k$ and $b_k$ have to be of the specific form as described above, but we want to look into more detail of what kind of conditions have to be imposed on the free parameters $Q_k$. We start by defining a surrogate functional that depends on the variable metric data fidelity in the same fashion as we have defined the surrogate functional in Section \ref{subsec:linbregiter}, i.e. we define
\begin{align*}
\regfct_k(u, \regparam) := \regfctarg[\regparam]{u} - \frac{1}{2} \| Ku - \noisy \|_{Q_k}^2 \, .
\end{align*}
If we guarantee convexity of $\regfct_k$, we can guarantee the following monotonic decrease result.
\begin{corollary}[Monotonic decrease]
Suppose $u^0$ satisfies $\| Ku^0 - \noisy \|_{Q_0}^2 < \infty$. Then the iterates of \eqref{eq:neuralnetworklinbregman} satisfy 
\begin{align}
\frac{1}{2} \| Ku^{k + 1} - \noisy \|_{Q_k}^2 + \bregdis[\regfct_k(\cdot, \regparam)]{q^k}{u^{k + 1}}{u^k} \leq \| Ku^k - \noisy \|_{Q_k}^2
\end{align}
for $u^k \in \regopitarg[\regparambold^{k - 1}]{\noisy, v^{k - 1}}$ and $q^k \in \partial \regfct_k(u^k, \regparam)$.
\begin{proof}
The proof is identical to the proof of Corollary \ref{cor:mondeclinbreg}.
\end{proof}
\end{corollary}

If we go back to the assumption $\regfctarg[\regparam]{u} = \frac{1}{2} \| u \|_{L^2(\Omega)}^2 + H(u, \regparam)$, we need to ensure that $Q_k$ is chosen such that not just $Q_k$, but also $I - K^\ast Q_k K$, is positive (semi-)definite for all $k$ in order to guarantee convexity of $\regfct_k$. With the next lemma we even observe that this is already enough to ensure Fej\'{e}r monotonicity of the iterates.
\begin{lemma}
Let $\clean \in \oprange_\fidelfct(K)$, $\minsol \in \select(\clean, \regparam)$ and let $\noisy \in \range$. Then the iterates satisfy the Fej\'{e}r monotonicity 
\begin{align}
\bregdis[\regfct_k(\cdot, \regparam)]{q^k}{\minsol}{u^k} \leq \bregdis[\regfct_{k - 1}(\cdot, \regparam)]{q^{k - 1}}{\minsol}{u^{k - 1}}
\end{align}
as long as $\| K\minsol - \noisy \|_{Q_{k - 1}} \leq \| Ku^k - \noisy \|_{Q_{k - 1}}$ is satisfied, for all $u^k \in \regopitarg[\regparambold^{k - 1}]{\noisy, v^{k - 1}}$ with $q^k \in \partial \regfct_k(u^k, \regparam)$ and $k \in \N$.
\begin{proof}
As in the previous Fej\'{e}r-monotonicity proofs we start with computing
\begin{align*}
\begin{split}
\bregdis[\regfct_k(\cdot, \regparam)]{q^k}{\minsol}{u^k} - \bregdis[\regfct_{k - 1}(\cdot, \regparam)]{q^{k - 1}}{\minsol}{u^{k - 1}} {} = {} &\bregdis[{\regfctarg[\regparam]{\cdot}}]{p^k}{\minsol}{u^k} - \bregdis[{\regfctarg[\regparam]{\cdot}}]{p^{k - 1}}{\minsol}{u^{k - 1}}\\
&+ \bregdis[\frac{1}{2} \| K\cdot - \noisy \|_{Q_{k - 1}}^2]{}{\minsol}{u^{k - 1}} - \bregdis[\frac{1}{2} \| K\cdot - \noisy \|_{Q_k}^2]{}{\minsol}{u^k}
\end{split} \, ,
\end{align*}
for all $k \in \N$. We further compute
\begin{align*}
\begin{split}
\bregdis[{\regfctarg[\regparam]{\cdot}}]{p^k}{\minsol}{u^k} - \bregdis[{\regfctarg[\regparam]{\cdot}}]{p^{k - 1}}{\minsol}{u^{k - 1}} &= - \bregdis[{\regfctarg[\regparam]{\cdot}}]{p^{k - 1}}{u^k}{u^{k - 1}} - \langle p^k - p^{k - 1}, \minsol - u^k \rangle \\
&= - \bregdis[{\regfctarg[\regparam]{\cdot}}]{p^{k - 1}}{u^k}{u^{k - 1}} + \langle K^\ast Q_{k - 1} (Ku^{k - 1} - \noisy), \minsol - u^k \rangle
\end{split} \, ,
\end{align*}
and estimate
\begin{align*}
\begin{split}
\bregdis[\frac{1}{2} \| K\cdot - \noisy\|_{Q_{k - 1}}^2]{}{\minsol}{u^{k - 1}} - \bregdis[\frac{1}{2} \| K\cdot - \noisy \|_{Q_k}^2]{}{\minsol}{u^k} {} \leq {} &\bregdis[\frac{1}{2} \| K\cdot - \noisy\|_{Q_{k - 1}}^2]{}{\minsol}{u^{k - 1}}\\
{} = {} &\frac{1}{2} \| K\minsol - \noisy \|_{Q_{k - 1}}^2 - \frac{1}{2} \| Ku^{k - 1} - \noisy \|_{Q_{k - 1}}^2\\
&- \langle K^\ast Q_{k - 1} (Ku^{k - 1} - \noisy), \minsol - u^{k - 1} \rangle\\
\end{split} \, .
\end{align*}
Thus, we observe 
\begin{align*}
\begin{split}
\bregdis[\regfct_k(\cdot, \regparam)]{q^k}{\minsol}{u^k} - \bregdis[\regfct_{k - 1}(\cdot, \regparam)]{q^{k - 1}}{\minsol}{u^{k - 1}} {} \leq {} &- \bregdis[{\regfctarg[\regparam]{\cdot}}]{p^{k - 1}}{u^k}{u^{k - 1}} \\
&+ \frac{1}{2} \| K\minsol - \noisy \|_{Q_{k - 1}}^2 - \frac{1}{2} \| Ku^{k - 1} - \noisy \|_{Q_{k - 1}}^2\\
&- \langle K^\ast Q_{k -1} (Ku^{k - 1} - \noisy), u^k - u^{k - 1} \rangle \\
{} = {} &- \bregdis[{\regfctarg[\regparam]{\cdot}}]{p^{k - 1}}{u^k}{u^{k - 1}}\\
&+ \frac{1}{2} \| K\minsol - \noisy \|_{Q_{k - 1}}^2 - \frac{1}{2} \| Ku^k - \noisy \|_{Q_{k - 1}}^2\\
&+ \bregdis[\frac{1}{2} \| K\cdot - \noisy \|_{Q_{k - 1}}^2]{}{u^k}{u^{k - 1}} \\
{} = {} &-\bregdis[\regfct_k(\cdot, \regparam)]{q^{k - 1}}{u^k}{u^{k - 1}} \\
&+ \frac{1}{2} \| K\minsol - \noisy \|_{Q_{k - 1}}^2 - \frac{1}{2} \| Ku^k - \noisy \|_{Q_{k - 1}}^2\\
{} \leq {} &\frac{1}{2} \| K\minsol - \noisy \|_{Q_{k - 1}}^2 - \frac{1}{2} \| Ku^k - \noisy \|_{Q_{k - 1}}^2
\end{split} \, .
\end{align*}
Hence, we guarantee Fej\'{e}r monotonicity as long as $\| K\minsol - \noisy \|_{Q_{k - 1}} \leq \| Ku^k - \noisy \|_{Q_{k - 1}}$ is satisfied.
\end{proof}
\end{lemma}

The previous corollary and lemma suggest that a sensible model for learning the parameters $\regparambold^k$ based on a set of training data pairs $\{ (\minsol_j, \noisy_j) \}_{j \in \{1, \ldots, m \}}$ is the following:
\begin{align*}
\begin{split}
\hat{\regparambold}^{k^\ast} {} = {} \argmin_{\regparambold^{k^\ast}} \left\{ \sum_{k = 1}^{k^\ast} \left[ \sum_{j = 1}^m \bregdis[\regfct_k(\cdot, \regparam)]{q^k}{\minsol_j }{u^k_j} + \chi_{\succeq 0}(I - K^\ast Q_{k - 1} K) + \chi_{\succeq 0}(Q_{k - 1}) \right] \right.\\
\left. \quad \text{subject to} \quad u^k_j \in \regopitarg[\regparambold^{k - 1}]{\noisy_j, v^{k - 1}_j} \quad  \vphantom{\left[ \sum_{k = 1}^{k^\ast} \right]} \right\}
\end{split} \, .
\end{align*}
The minimization problem can either be solved simultaneously for all parameters, or subsequently, keeping all previously computed parameters fixed. 
The minimization problem can further be equipped with additional constraints, such as $\| K\minsol_j - \noisy_j \|_{Q_{k - 1}} \leq \| Ku_j^k - \noisy_j \|_{Q_{k - 1}}$ or $\| Ku^{k + 1}_j - \noisy_j \|_{Q_{k + 1}} \leq \| Ku^k_j - \noisy_j \|_{Q_k}$ for all $k \in \{0, \ldots, k^\ast - 1\}$ and $j \in \{1, \ldots, m\}$.


\section{Conclusions \& Outlook}

Modern regularization techniques, in particular those based on (nonsmooth) convex variational models are a versatile tool for improved reconstruction in inverse problems when appropriate prior information is available. Further improvements can be made by constructing iterative regularization methods using the same underlying variational model. Those can reduce systematic errors and bias, but also yield interesting novel insights into scale properties, spectral and multi-scale decompositions, and even link to deep neural network architectures.

Several aspects are expected to play a role in the future development and understanding of regularization methods. A key issue are stochastic models and uncertainty quantification, which we have only touched superficially in this survey. This topic appears to be at a similar stage as the deterministic regularization theory around the year 2000, the Gaussian case (corresponding to linear regularization methods in Hilbert space) seems to be well understood reasonably well now for linear and nonlinear inverse problems. Much less is known about non-Gaussian priors in Banach spaces, but there is a boost of papers tackling those recently. Relevant problems are e.g. the link between Bayesian models and variational approaches, the convergence of posterior distributions, and advanced statistical inference in infinite-dimensional Banach spaces. So far there are also basically no results on the analysis of iterative regularization methods in a stochastic setup.

A topic of strong recent interest are eigenvalue problems and spectral decompositions. While it remains unclear how far they can be pushed for practical purposes, they already yield a new understanding of the geometry of inverse problems and regularization methods, partly closing the gap to the standard tool of singular value decomposition for linear regularization methods.

A topic that has not yet been investigated from a theoretical point of view, but are often used in engineering practice, are methods that effectively compute Nash equilibria instead of minimizers. Such methods arise from problems where two (or more) unknowns are reconstructed in an iterative fashion. Then often one of the variables is frozen and a variational problem with respect to the other one is solved, e.g. in motion-corrected reconstruction when in alternating iteration images are reconstructed from indirect data with given motion and motion is estimated directly from images data. Convergence of such procedures is often observed in practice and yields good results, but so far there is no systematic theory.

From an application point of view high-dimensional and joint reconstruction problems are a key subject for current and future development, many aspects of modelling and analysis are still open in this context. Examples of current interest are joint reconstruction of images and motion in many biomedical applications or reconstructions in dynamic or spectral problems with strong undersampling. 

Finally, machine learning is expected to play an important role in regularization methods for inverse problems (as in other disciplines related to processing data). The learning theory will need to be adapted to the special needs of inverse problems due to the aspects of ill-posedness, which cannot be captured by current learning architectures, and the particular difficulties to obtain meaningful training data for inverse problems.

\section*{Acknowledgements}

The authors thank Eva-Maria Brinkmann and Julian Rasch (WWU M\"unster) for proof-reading, comments improving the paper, and providing computational results related to debiasing and dynamic MR reconstruction. 
MBe acknowledges support from the Leverhulme Trust Early Career Fellowship 'Learning from mistakes: a supervised feedback-loop for imaging applications', the Isaac Newton Trust and the Cantab Capital Institute for the Mathematics of Information. MBu 
acknowledges support by ERC via Grant EU FP 7 - ERC Consolidator Grant
615216 LifeInverse and by the German Ministry for Science and Education (BMBF)
through the project MED4D. The authors would like to thank the Isaac Newton Institute
for Mathematical Sciences, Cambridge, for support and hospitality during the
programme Variational Methods for Imaging and Vision, where work on this
paper was undertaken, supported by EPSRC grant no EP/K032208/1.

\bibliographystyle{siam}
\bibliography{references}

\end{document}